\documentclass[reqno,11pt,a4paper]{amsart}
\usepackage{amsmath,amsfonts,amsthm,amssymb,amscd}
\usepackage{hyperref}
\usepackage[alphabetic]{amsrefs}
\usepackage{tikz}
\usepackage{graphicx}
\usepackage{tikz-cd}
\usepackage{enumitem}

\usepackage[hmargin=34mm,vmargin=41.6mm,footskip=5ex]{geometry}

\numberwithin{equation}{section}	

\makeatletter
\def\ps@headings{\ps@empty
  \def\@evenhead{%
    \setTrue{runhead}%
    \normalfont\scriptsize
    \hfil
    \def\thanks{\protect\thanks@warning}%
    \leftmark{}{}\hfil}%
  \def\@oddhead{%
    \setTrue{runhead}%
    \normalfont\scriptsize \hfil
    \def\thanks{\protect\thanks@warning}%
    \rightmark{}{}\hfil}%
  \def\@oddfoot{\normalfont\scriptsize \hfil\thepage\hfil}%
  \let\@evenfoot\@oddfoot
  \let\@mkboth\markboth
}
\makeatother

\newtheorem{lem}{Lemma}[section]

\newtheorem{tw}[lem]{Theorem}

\newtheorem{cor}[lem]{Corollary}
\newtheorem{prop}[lem]{Proposition}

\newcommand {\E}{\mathbb{E}_{\Delta}^2}
\newcommand {\trol}{L}
\newcommand{\wis}[1]{W(#1)}

\newcommand{\splus}{\hspace{0.05em} {+} \hspace{0.05em}}

\newcommand {\eg}{\underline{E}G}
\newcommand {\eeg}{  \underline{\underline{E}}G  }

\newcommand {\ov}{\overline}
\newcommand {\trl}[1]{L(#1)}
\newcommand	{\cen}[1]{C_G(#1)}

\newcommand {\minset}[1]{\mathrm{Min}#1}
\newcommand {\im}[1]{{#1}}

\newcommand {\g}{\gamma}

\newcommand {\frc}[1]{S^{#1}\mathrm{FRC}}
\newcommand{\rpm}{\sbox0{$1$}\sbox2{$\scriptstyle\pm$}
  \raise\dimexpr(\ht0-\ht2)/2\relax\box2 }

\newcommand {\thi}[1]{Th(#1)}

\newtheorem*{twa}{Theorem A}
\newtheorem*{twb}{Theorem B}
\newtheorem*{twc}{Theorem C}
\newtheorem*{twd}{Theorem D}
\newtheorem*{twe}{Theorem E}
\newtheorem*{twf}{Theorem F}

\newtheorem*{twg}{Theorem G}

\theoremstyle{definition}
\newtheorem{de}[lem]{Definition}
\newtheorem{convent}[lem]{Convention}

\newtheorem{rem}[lem]{Remark}
\newtheorem{ex}[lem]{Example}
\begin{document}

\title[Classifying spaces for systolic groups]{Classifying spaces for families of subgroups for systolic groups}
\author[D.\ Osajda]{Damian Osajda}
\address{Instytut Matematyczny,
Uniwersytet Wroc{\l}awski\\
pl.\ Grunwaldzki 2/4, %\newline
50--384 Wroc\-{\l}aw, Poland}
\address{Institute of Mathematics, Polish Academy of Sciences\\ \'Sniadeckich 8, %\newline  
00--656 War\-sza\-wa, Poland}
\email{dosaj@math.uni.wroc.pl}
\author[T.\ Prytu{\l}a]{Tomasz Prytu{\l}a}
\address{School of Mathematics, University of Southampton, Southampton SO17 1BJ, UK}
\email{t.p.prytula@soton.ac.uk}
\date{\today}
%\thanks{symmm}

\begin{abstract}
We determine the large scale geometry of the minimal displacement set of a hyperbolic isometry of 
a systolic complex. As a consequence, we describe the centraliser of such an isometry in a systolic group.
Using these results, we construct a low-dimensional classifying space for the family of virtually cyclic subgroups of a group acting properly on a systolic complex. Its dimension coincides with the topological dimension of the complex if the latter is at least four.
We show that graphical small cancellation complexes are classifying spaces for proper actions and that the groups acting on them properly admit three-dimensional classifying spaces with virtually cyclic stabilisers.
This is achieved by constructing a systolic complex equivariantly homotopy equivalent to a graphical
small cancellation complex. On the way we develop a systematic approach to graphical small
 cancellation complexes. Finally, we construct low-dimensional models for the family of virtually abelian
subgroups for systolic, graphical small cancellation, and some $\mathrm{CAT}(0)$ groups.
\end{abstract}

\maketitle

\tableofcontents

\section{Introduction}
\label{intro}

Let $G$ be a group and let $\mathcal F$ be a family of subgroups of $G$, that is, a collection of subgroups which is closed under taking subgroups and conjugation by elements of $G$. A \emph{classifying space for the family $\mathcal{F}$ is a $G$--CW--complex $E_{\mathcal F}G$} with stabilisers in $\mathcal F$, such that for any subgroup $F \in \mathcal{F}$ the fixed point set $(E_{\mathcal F}G)^F$ is contractible.
When the family $\mathcal{F}$ consists of just a trivial subgroup, the classifying space $E_{\mathcal{F}}G$ is the universal free $G$--space $EG$, and if $\mathcal{F}$ consists of all finite subgroups of $G$ then  $E_{\mathcal{F}}G$ is the so-called \emph{classifying space for proper actions}, commonly denoted by $\underline EG$.
Recently, much attention has been attracted by the classifying space $\underline{\underline{E}}G$
for the family of all virtually cyclic subgroups. One reason for studying $\underline{\underline{E}}G$ is its appearance in the formulation of the Farrell-Jones conjecture concerning algebraic $K$- and $L$-theory (see e.g.\ \cite{Luck-surv}).

It can be shown that the classifying space $E_{\mathcal{F}}G$ always exists and that it is unique up to a $G$--homotopy equivalence \cite{Luck-surv}. A concern is to provide specific models that are as ``simple" as possible. One, widely used, measure of such simplicity is the (topological) dimension.
For example, having a low-dimensional model for $\underline{\underline{E}}G$ enables one to better understand its homology
that appears in the left-hand side of the assembly map in the formulation of the Farrell-Jones conjecture.
Low-dimensional models for $\underline{\underline{E}}G$ were constructed for many classes of groups including hyperbolic groups \cite{LePi}, groups acting properly on $\mathrm{CAT}(0)$ spaces \cite{Lu09}, and many
two-dimensional groups \cite{Die15}. In all of these constructions the minimal dimension of $\underline{\underline{E}}G$ is related to the minimal dimension of $\underline{E}G$. However, the discrepancy between these two may be arbitrarily large \cite{DD}.
Finally, it is an open question whether finite-dimensional models for $\underline{\underline{E}}G$ exist for all groups admitting a finite dimensional model for $\underline{E}G$.

The main purpose of the current article is to construct low-dimensional models for $\underline{\underline{E}}G$ in the case of groups
acting properly on systolic complexes.\medskip

A simply connected simplicial complex is \emph{systolic} if it is a flag complex, and if every embedded cycle of length $4$ or $5$ has a diagonal.
This condition may be treated as an upper curvature bound, and therefore systolic complexes are also called ``complexes of simplicial non-positive curvature''. 
They were first introduced by Chepoi \cite{Ch4} under the name \emph{bridged complexes}. Their $1$--skeleta, \emph{bridged graphs}, were studied earlier in the frame of Metric Graph Theory \cites{SC,FJ}. Januszkiewicz-{\' S}wi{\c a}tkowski \cite{JS2} and Haglund \cite{H}
rediscovered, independently,  systolic complexes and initiated the exploration of groups acting on them.
Since then the theory of automorphisms groups of systolic complexes has been a powerful tool for providing examples of groups
with often unexpected properties, see e.g.\ \cites{JS3,OS}.\medskip

An exotic large-scale geometric feature of systolic complexes is their ``asymptotic asphericity" -- asymptotically, they do not contain essential spheres. Such a behaviour is typical for complexes of asymptotic dimension one or two, but systolic
complexes exist in arbitrarily high dimensions.
The asphericity property is the crucial phenomenon used in our approach in the current article.
First, we use it to determine the large-scale geometry of the minimal displacement set of a hyperbolic isometry of a systolic complex.

Recall that an isometry (i.e.\ a simplicial automorphism) $h$ of a systolic complex $X$ is called \emph{hyperbolic} if it does not fix any simplex of $X$. For such an isometry one defines the \emph{minimal displacement set} $\minset(h)$ to be the subcomplex of $X$ spanned by the vertices which are moved by $h$ the minimal combinatorial distance. 

\begin{twa}[Theorem~\ref{coarsesplit} and Corollary~\ref{s1frcquasitreecoro}]
The minimal displacement set of a hyperbolic isometry of a uniformly locally finite systolic complex is quasi-isometric to the product $T\times \mathbb R$
of a tree and the line.
\end{twa}

This theorem can be viewed as a systolic analogue of the so-called Product Decomposition Theorem for $\mathrm{CAT}(0)$ spaces \cite[Theorem II.6.8]{BrHa}. Unlike the $\mathrm{CAT}(0)$ case where the splitting is isometric and it is realised within the ambient space, we provide only an abstract quasi-isometric splitting. This is mainly due to the lack of a good notion of products in the category of simplicial complexes. However, our theorem may be seen as more restrictive, since in the $\mathrm{CAT}(0)$ case instead of a tree one can have an arbitrary $\mathrm{CAT}(0)$ space.

This restriction is used to determine the structure of certain groups acting on the minimal displacement set. 
If a group $G$ acts properly on a systolic complex, then one can easily see that the centraliser of a hyperbolic isometry acts properly on the minimal displacement set. If the action of $G$ is additionally cocompact, i.e.\ $G$ is a \emph{systolic group}, Theorem A allows
us to describe the structure of such centraliser. This establishes a conjecture by D.\ Wise \cite[Conjecture 11.6]{Wise}.

\begin{twb}[Corollary~\ref{wisecoro}]
The centraliser of an infinite order element in a systolic group is commensurable with $F_n \times \mathbb{Z}$, where $F_n$ is the free group on $n$ generators for some $n \geqslant 0$.
\end{twb}

Theorem B extends also some results from \cites{JS3,OS,Osge} concerning normal subgroups of systolic groups. Theorem A is the key result in our approach to constructing low-dimensional models for  $\underline{\underline{E}}G$ for groups acting properly on systolic complexes.
We follow a ``pushout method'' to construct the desired complex using low-dimensional models for $\underline{{E}}G$.
In \cite{OCh} it is shown that if a group acts properly on a $d$--dimensional systolic complex $X$ then the barycentric subdivision
of $X$ is a model for $\underline{{E}}G$. We then follow the strategy of W.\ L{\" u}ck \cite{Lu09} used for
constructing models for $\underline{\underline{E}}G$ for $\mathrm{CAT}(0)$ groups. The key point there is, roughly, to determine the structure of the quotient $N_G(h)/\langle h \rangle $ of the normaliser of a hyperbolic isometry $h$.
Using similar arguments as in the proof of Theorem B, we show that this quotient is locally virtually free. This is a strong restriction on  $N_G(h)/\langle h \rangle $ which leads to better dimension bounds, when compared with the $\mathrm{CAT}(0)$ case.

\begin{twc}[Theorem~\ref{mainthm}]
Let $G$ be a group acting properly on a uniformly locally finite $d$--dimensional systolic complex. Then there exists a model for $\underline{\underline{E}}G$ of dimension

\[\mathrm{dim}\underline{\underline{E}}G = \left\{ \begin{array}{cl}

      d+1 & \text{ if } d \leqslant 3, \\
      d & \text{ if } d \geqslant 4.  \\

      \end{array} \right.
 \] 
\end{twc}

In Section~\ref{sec:ex} we provide several classes of examples to which our construction applies.

As a follow-up, we consider the family ${\mathcal{VAB}}$ of all virtually abelian subgroups. 
To the best of our knowledge there have been no known general constructions of low-dimensional classifying spaces
for this family, except for cases reducing to studying the family of virtually cyclic groups (as in the case of hyperbolic groups). In the realm of systolic groups we are able to provide such 
constructions in the full generality.

\begin{twd}[Theorem~\ref{twvab2}] Let $G$ be a group acting properly and cocompactly on a $d$--dimensional systolic complex. Then there exists a model for $E_{\mathcal{VAB}}G$ of dimension %$\mathrm{dim}E_{\mathcal{VAB}}G= 
	$\mathrm{max} \{4, d\}$.
\end{twd}

The most important tool used in the latter construction is the systolic Flat Torus Theorem \cite{E1}.
As an immediate consequence of the methods developed for proving Theorem D, we obtain the following.

\begin{twe}[Corollary~\ref{cat0coro}]Let $G$ be a group acting properly and %by semisimple isometries on a complete proper 
	cocompactly by isometries on a 
	$\mathrm{CAT}(0)$ space $X$ of topological dimension $d>0$. Furthermore, assume that for $n >2$ there is no isometric embedding $\mathbb{E}^n \to X$ where $\mathbb{E}^n$ is the Euclidean space.
	%where $\mathbb{E}^n$ is the Euclidean space with the standard metric. 
	Then there exists a model for $E_{\mathcal{VAB}}G$ of dimension $\mathrm{max} \{4, d+1\}$.
\end{twe}

In particular, this result applies to lattices in rank--$2$  symmetric spaces thus answering a special case
of a question by J.-F.\ Lafont \cite[Problem 46.7]{Gui}.
\medskip

Classical examples of groups acting properly on systolic complexes are small cancellation groups \cite{Wise}. Note that small cancellation
groups are not always hyperbolic and only for some of them a  $\mathrm{CAT}(0)$ structure is provided. There is a natural construction
by D.\ Wise of a systolic complex associated to a small cancellation complex. Therefore, Theorem C and Theorem D apply in the small cancellation
setting. 

In the current article we explore the more general and more powerful theory of \emph{graphical small cancellation}, attributed to Gromov \cite{Gro}. Furthermore, instead of studying small cancellation 
presentations, we consider a slightly more general situation of graphical small cancellation complexes and their automorphism groups. Our approach is analogous to the one by McCammond-Wise \cite{McWi} in the classical small cancellation theory.
We initiate the systematic study of graphical small cancellation complexes, in particular in  Section~\ref{sec:graphical} we prove their basic geometric properties. The theory of groups acting properly on graphical small cancellation complexes provides a powerful tool for constructing groups with prescribed features. Examples include finitely generated groups containing expanders, and non-exact groups with the Haagerup property, both constructed in \cite{O-sc}.

Towards our main application, which is constructing low-dimensional models for classifying spaces,
we define the \emph{Wise complex} of a graphical small cancellation complex. It is the nerve of a particular cover of the graphical complex. We show that the Wise complex of a simply connected $C(p)$ graphical complex is $p$--systolic
(Theorem~\ref{t:ss}), and that the two complexes are equivariantly homotopy equivalent in the presence of
a group action (Theorem~\ref{t:Ghe}). The latter result is achieved by the use of an equivariant version of
the Borsuk Nerve Theorem, which we prove on the way (Theorem~\ref{t:nerve}). As a corollary we obtain the following. 

\begin{twf}[Corollary~\ref{c:gscsys}]
	Let $G$ be a group acting properly and cocompactly on a simply connected 
	$C(p)$ graphical complex for $p\geqslant 6$. Then $G$ acts properly and cocompactly on a 
	$p$--systolic complex, i.e.\ $G$ is a $p$--systolic group.
\end{twf}

The result above allows one to conclude many strong features of groups acting geometrically on
graphical small cancellation complexes. Among them is biautomaticity, proved for classical small cancellation 
groups in \cite{GeSh2}, and for systolic groups in \cite{JS2}.

Using the above techniques we are able to construct low-dimensional models for classifying spaces for
groups acting properly on graphical small cancellation complexes.

\begin{twg}[Theorem~\ref{twclasssmall}]
Let a group $G$ act properly on a simply connected uniformly locally finite $C(6)$ graphical complex $X$. Then: 
\begin{enumerate}
	\item the complex $X$ is a $2$--dimensional model for $\eg$, 
	\item there exists a $3$--dimensional model for $\eeg$,
	\item there exists a $4$--dimensional model for $E_{\mathcal{VAB}}G$, provided the action is additionally cocompact.
        \end{enumerate}

\end{twg}
 
%\medskip
\noindent
\textbf{Organisation.} The article consists of two main parts. The first part (Sections~\ref{sec:preliminaries}--\ref{sec:mainthm}) is concerned mostly with geometry of systolic complexes. In Section~\ref{sec:preliminaries} we give a background on systolic complexes and on classifying spaces for families of subgroups. We also recall a general method of constructing classifying spaces for the family of virtually cyclic subgroups developed in \cite{LuWe}. %that we use in Section~\ref{sec:mainthm}. 
In Section~\ref{sec:minset} we show that the minimal displacement set of a hyperbolic isometry of a systolic complex splits up to quasi-isometry as a product of a real line and a certain simplicial graph. Section~\ref{sec:fillingradius} is devoted to proving that this graph is quasi-isometric to a simplicial tree. The proof relies on the aforementioned asymptotic asphericity properties of systolic complexes. Finally in Section~\ref{sec:mainthm}, using the method described in Section~\ref{sec:preliminaries} we construct models for $\eeg$ and  $E_{\mathcal{VAB}}G$ for groups acting properly on systolic complexes.

In the second part (Sections~\ref{sec:graphical}--\ref{sec:dualisisystolic}) we study graphical small cancellation theory. In Section~~\ref{sec:graphical} we initiate systematic studies of small cancellation complexes and show their basic geometric properties. In Section~\ref{sec:dualisisystolic} we prove that the dual complex of a graphical small cancellation complex is systolic. Then we use this fact to construct models for $\eg$,  $\eeg$ and  $E_{\mathcal{VAB}}G$ for groups acting properly on graphical small cancellation complexes.

We conclude with Section~\ref{sec:ex} where we provide various examples of groups to which our theory applies.\medskip

\noindent
\textbf{Acknowledgements.} We would like to thank Dieter Degrijse for many helpful discussions, in particular pointing us towards the proof of Theorem~\ref{twvab2}, and for a careful proofreading of the manuscript.
We thank Jacek \'Swi\polhk{a}tkowski and the anonymous referee for many valuable remarks.

D.O. was supported by (Polish) Narodowe Centrum Nauki, grant no.\ UMO-2015/\-18/\-M/\-ST1/\-00050.
T.P. was supported by the Danish National Research Foundation through the Centre for Symmetry and Deformation (DNRF92).
%\clearpage

\section{Preliminaries}\label{sec:preliminaries}
\subsection{Classifying spaces with finite or virtually cyclic stabilisers}\label{sec:classifying}
The main goal of this section is, given a group $G$, to describe a method of constructing a model for a classifying space $\eeg$ out of a model for $\eg$. The presented method is due to W.\ L\"uck and M.\ Weiermann \cite{LuWe}. After giving the necessary definitions we describe the steps of the construction, some of which we adjust to our purposes.\medskip

A collection of subgroups $\mathcal{F}$ of a group $G$ is called a \emph{family} if it is closed under taking subgroups and conjugation by elements of $G$. Three examples which will be of interest to us are the family $\mathcal{FIN}$ of all finite subgroups, the family $\mathcal{VCY}$ of all virtually cyclic subgroups, and the family $\mathcal{VAB}$ of all virtually abelian subgroups. Let us define the main object of our study.

\begin{de}
Given a group $G$ and a family of its subgroups $\mathcal{F}$, a \emph{model for the classifying space} $E_{\mathcal{F}}G$ is a $G$--CW--complex $X$ such that for any subgroup $H \subset G$ the fixed point set $X^H$ is contractible if $H \in \mathcal{F}$, and empty otherwise. 
\end{de}

In order to simplify the notation, throughout the article let $\underline{E}G$ denote $E_{\mathcal{FIN}}G$ and let $\underline{\underline{E}}G$ denote $E_{\mathcal{VCY}}G$. This is a standard, commonly used notation. \medskip

A model for $E_{\mathcal{F}}G$ exists for any group and any family; moreover, any two models for $E_{\mathcal{F}}G$ are $G$--homotopy equivalent. For the proofs of these facts see~\cite{Luck-surv}. However, general constructions always produce infinite dimensional models. We will now describe the aforementioned method of constructing a finite dimensional model for $\eeg$ out of a model for $\eg$ and appropriate models associated to infinite virtually cyclic subgroups of $G$. Before doing so, we need one more piece of notation. If $H \subset G$ is a subgroup and $\mathcal{F}$ is a family of subgroups of $G$, let $\mathcal{F} \cap H$ denote the family of all subgroups of $H$ which belong to the family $\mathcal{F}$. More generally, if $\psi \colon H \to G$ is a homomorphism, let $\psi^{\ast}\mathcal{F}$ denote the smallest family of subgroups of $H$ that contains $\psi^{-1}(F)$ for all $ F \in \mathcal{F}$.  \medskip

Consider the collection $\mathcal{VCY} \setminus \mathcal{FIN}$ of infinite virtually cyclic subgroups of $G$. It is not a family since it does not contain the trivial subgroup. Define an equivalence relation on $\mathcal{VCY} \setminus \mathcal{FIN}$ 
by 
\[H_1 \sim H_2 \iff |H_1 \cap H_2| = \infty. \]
Let $[H]$ denote the equivalence class of $H$, and let $[\mathcal{VCY} \setminus \mathcal{FIN}]$ denote the set of equivalence classes. The group $G$ acts on $[\mathcal{VCY} \setminus \mathcal{FIN}]$ by conjugation, and for a class $[H] \in \mathcal{VCY} \setminus \mathcal{FIN}$ define the subgroup $N_G[H] \subseteq G$ to be the stabiliser of $[H]$ under this action, i.e.
\[N_G[H]=\{g \in G \mid |g^{-1}Hg \cap H|= \infty\}.\]
The subgroup $N_G[H]$ is called the \emph{commensurator of} $H$, since its elements conjugate $H$ to the subgroup commensurable with $H$.
For $[H] \in [\mathcal{VCY} \setminus \mathcal{FIN}]$ define the family $\mathcal{G}[H]$  of subgroups of $N_G[H]$ as follows \[\mathcal{G}[H]=\{ K \subset G \mid K \in [\mathcal{VCY} \setminus \mathcal{FIN}], [K]=[H] \} \cup \{ K \in \mathcal{FIN} \cap N_G[H]\}.\]
In other words $\mathcal{G}[H]$ consists of all infinite virtually cyclic subgroups of $G$ which have infinite intersection with $H$ and of all finite subgroups of $N_G[H]$. In this setting we have  the following. %btheorem provides the way of constructing a model for $\underline{\underline{E}}G$.

\begin{tw}\cite[Theorem 2.3]{LuWe} 
\label{luwepushout} Let $I$ be a complete set of representatives $[H]$ of $G$--orbits of $[\mathcal{VCY} \setminus \mathcal{FIN}]$ under the action of $G$ by conjugation. Choose arbitrary %$N_G[H]$--CW--
models for $\underline{E} N_G[H]$ and $E_{\mathcal{G}[H]} N_G[H]$ and an arbitrary %{$G$--CW--}
model for $\underline{E}G$. Let the space $X$ be defined as a cellular $G$--pushout
\[
\begin{CD}
\coprod_{[H] \in I}  G \times_{N_G[H]} \underline{E} N_G[H]     @>i>>  \underline{E}G \\
@VV{\coprod_{[H] \in I} \mathrm{id}_G \times_{N_G[H]} f_{[H]}}V         @VVV\\
\coprod_{[H] \in I}  G \times_{N_G[H]} E_{\mathcal{G}[H]} N_G[H]     @>{\phantom{\text{label}}}>>   X \smallskip
\end{CD}
\]
such that $f_{[H]}$ is a cellular $N_G[H]$--map for every $[H] \in I$, and $i$ is an inclusion of $G$--CW--complexes. Then $X$ is a %$G$--CW--
model for $\eeg$.
\end{tw}

Existence of such pushout follows from universal properties of appropriate classifying spaces, and the fact that if the map $i$ fails to be injective, one can replace it with an inclusion into the mapping cylinder. For details see \cite[Remark 2.5]{LuWe}. This observation leads to the following corollary.
\begin{cor}
\label{commenough}\cite[Remark 2.5]{LuWe} Assume that there exists an $n$--dimensional %{$G$--CW--model} 
model for $\eg$, and for every $[H] \in [\mathcal{VCY} \setminus \mathcal{FIN}]$ there exists an $n$--dimensional %$N_G[H]$--CW--
model for $E_{\mathcal{G}[H]}N_G[H]$, and an $(n-1)$--dimensional %$N_G[H]$--CW--
model for $\underline{E} N_G[H]$. Then there exists an $n$--dimensional %$G$--CW--
model for $\underline{\underline{E}}G$.
\end{cor}

In what follows we need our groups to be finitely generated. The commensurator $N_G[H]$ in general does not have to be finitely generated. The following proposition allows us to reduce the problem of finding various models for $E_{\mathcal{F}}N_G[H]$ to the study of its finitely generated subgroups.

\begin{prop}
\label{limitcoro}
If for every finitely generated subgroup $K\subset N_G[H]$ there exists a model for $E_{\mathcal{G}[H]\cap K}K$ with $\mathrm{dim}E_{\mathcal{G}[H]\cap K}K \leqslant n$, then there exists an $(n+1)$--dimensional model for $E_{\mathcal{G}[H]}N_G[H]$. 
The same holds for models for $\underline{E}N_G[H]$.
\end{prop}

\begin{proof}The proof is a straightforward application of Theorem 4.3 in~\cite{LuWe}, which treats colimits of groups. 
The group $N_G[H]$ can be written as a colimit $N_G[H]={\mathrm{colim}}_{i \in I}K_i$ of a directed system $\{K_i\}_{i \in I}$ of all of its finitely generated subgroups (since $N_G[H]$ is countable, this system is countable as well). Since the structure maps are injective and since every subgroup $F \in \mathcal{G}[H]$ is finitely generated, it is contained in the image of some $\psi_i\colon K_i \hookrightarrow G$. Again by the injectivity of $\psi_i$, we have $\psi^{\ast}_i
\mathcal{G}[H]= \mathcal{G}[H]\cap K_i$. The claim follows from Theorem 4.3 in~\cite{LuWe}.%(or $\psi(K_i)$???). \underset{i\in I}
\end{proof}

The following condition will allow us to find infinite cyclic subgroups which are normal in $K.$

\begin{de}\cite[Condition 4.1]{Lu09}
\label{condition}
 A group $G$ satisfies condition (C) if for every $g, h \in G$ with $|h|=\infty$ and any $k,l \in \mathbb{Z}$ we have \[gh^kg^{-1}=h^l \implies |k|=|l|.\] 
\end{de}

\begin{lem}
\label{normalsubgroup} Let $K \subset N_G[H]$ be a finitely generated subgroup that contains some representative of $[H]$ and assume that the group $G$ satisfies condition $\mathrm{(C)}$. Choose an element $h \in H$ such that $[\langle h \rangle] = [H]$ (any element of infinite order has this property). Then there exists $k \geqslant 1$, such that $\langle h^k \rangle$ is normal in $K$.
\end{lem}

\begin{proof}Let $s_1, \ldots, s_m$ be generators of $K$. Since $K \subset N_G[H]$, for any $s_i$ we have $s_i^{-1}h^{k_i}s_i=h^{l_i}$ for some $k_i, l_i \in \mathbb{Z} \setminus \{0\}$. Then the condition (C) implies that $l_i= \rpm k_i$ for all $i \in \{1, \ldots, m\}$. Thus $k$ defined as $\prod_{i=1}^{n} k_i$ has the desired property. \end{proof}

In order to treat short exact sequences of groups we need the following result.

\begin{prop}\cite[Corollary 2.3]{DD}
\label{corodieter}   Consider the short exact sequence of groups \[0 \longrightarrow N \longrightarrow G \overset{\pi} \longrightarrow Q \longrightarrow 0.\]

Let $\mathcal{F}$ be a family of subgroups of $G$ and $\mathcal{H}$ be a family of subgroups of $Q$ such that $\pi(\mathcal{F}) \subseteq \mathcal{H}$. Suppose that there exists a integer $k \geqslant 0$, such that for every subgroup $H \in \mathcal{H}$ there exists a $k$--dimensional model for $E_{\mathcal{F} \cap \pi^{-1} (H)}  {\pi}^{-1}(H)$. Then given a model for $E_{\mathcal{H}}Q$, there exists a model for $E_{\mathcal{F}}G$ of dimension $k+\mathrm{dim}E_{\mathcal{H}}Q$.
\end{prop}

\subsection{Systolic complexes}\label{subsec:systolic} In this section we establish the notation and define systolic complexes and groups. We do not discuss general properties of systolic complexes, the interested reader is referred to \cite{JS2}.
We also give basic definitions regarding metric on simplicial complexes, including notation which is slightly different from the one usually used. \medskip

Let $X$ be a simplicial complex. We assume that $X$ is finite dimensional and uniformly locally finite.
%but not necessarily locally finite. 
For a subset of vertices $S \subseteq X^{(0)}$ define the subcomplex \emph{spanned by} $S$ to be the maximal subcomplex of $X$ having $S$ as its vertex set. We denote this subcomplex by $\mathrm{span}\, S$. We say that $X$ is \emph{flag} if every set of pairwise adjacent vertices spans a simplex of $X$. For a simplex $\sigma \in X$, define the \emph{link} of $\sigma$ as the subcomplex of $X$ that consists of all simplices of $X$ which do not intersect $\sigma$, but together with $\sigma$ span a simplex of $X$. A \emph{cycle} in $X$ is a subcomplex $\gamma \subset X$ homeomorphic to the $1$--sphere. The length $|\gamma|$ of the cycle $\gamma$ is the number of its edges.  A \emph{diagonal} of a cycle is an edge connecting two of its nonconsecutive vertices.

\begin{de}\cite[Definition 1.1]{JS2} Let $k \geqslant 6$ be a positive integer. A simplicial complex $X$ is $k$\emph{--large} if it is flag and every cycle $\gamma$ of length  $4 \leqslant |\gamma| < k$ has a diagonal.

We say that $X$ is $k$\emph{--systolic} if it is connected, simply connected and the link of every simplex of $X$ is $k$--large.
\end{de}

One can show that $k$--systolic complexes are in fact both $k$--large and flag.
In the case when $k=6$, which is the most interesting to us, we abbreviate $6$\emph{--systolic} to \emph{systolic}. Note that if $m > k$ then $m$--systolicity implies $k$--systolicity. \medskip

Now we introduce the convention used throughout this article regarding the metric on simplicial complexes. Some of our definitions are slightly different from the usual ones, however they seem to be more convenient here in order to avoid technical complications.

\begin{convent}\label{metricconvention}(Metric on simplicial complexes). Let $X$ be a simplicial complex. Unless otherwise stated, when we refer to the \emph{metric} on $X$, we mean its $0$--skeleton $X^{(0)}$, where the distance between two vertices is the minimal number of edges of an edge-path joining these two vertices. Notice that for flag complexes, the $0$--skeleton together with the above metric entirely determines the complex.
By an \emph{isometry} we mean a simplicial map $f \colon Y \to X$, which restricted to $0$--skeleta is an isometry with respect to the metric defined above. In particular, any simplicial isomorphism is an isometry. %and vice versa, every isometry is a simplicial map.
A \emph{geodesic} in a simplicial complex is defined as a sequence of vertices $(v_i)_{i \in I}$ such that $d(v_i, v_j)= |i-j|$ for all $i,j \in I$, where $I\subseteq \mathbb Z$ is a subinterval in integers. Note that we allow $I = \mathbb{Z}$, i.e.\ a geodesic can be infinite in both directions.
A \emph{graph} is a $1$--dimensional simplicial complex. In particular, graphs do not contain loops and multiple edges. 
\end{convent}

\begin{rem}\label{rem:twometrics}For a graph the usually considered metric is the geodesic metric where every edge is assigned length $1$.
If $X$ is a simplicial complex, then the restriction of the geodesic metric on the graph $X^{(1)}$ to $X^{(0)}$ is precisely the metric we defined above.
\end{rem}

Let $v_0$ be a vertex of $X$. Define the \emph{combinatorial ball} of radius $r \in \mathbb{N}$, centred at $v_0$ as a subcomplex
\[B_n(v_0, X)=\mathrm{span}\{v \in X \mid d(v_0, v) \leqslant n\},\] 
and a \emph{combinatorial sphere} as 
\[S_n(v_0, X)=\mathrm{span}\{v \in X \mid d(v_0, v) =n\}.\]

We finish this section with basic definitions regarding group actions on simplicial complexes. Unless stated otherwise,  all groups are assumed to be discrete and all actions are assumed to be \emph{simplicial}, i.e.\ groups act by simplicial automorphisms. 
%Let a discrete group $G$ act on a simplicial complex $X$. 
We say that the action of $G$ on a simplicial complex $X$ is \emph{proper} if for every vertex $v \in X$ and every integer $n \geqslant 0$ the set \[\{g \in G \mid gB_n(v,X) \cap B_n(v,X) \neq  \emptyset\}\] is finite. When $X$ is uniformly locally finite this definition is equivalent to all vertex stabilisers being finite.
We say that the action is \emph{cocompact} if there is a compact subset $K \subset X$ that intersects every orbit of the action.

A group is called \emph{systolic} if it acts properly and cocompactly on a systolic complex. 
%In this case, the complex is uniformly locally finite.
However, most of the time we are concerned with proper actions that are not necessarily cocompact.

\subsection{Quasi-isometry}

Let $(X, d_X)$ and $(Y, d_Y)$ be metric spaces. A (not necessarily continuous) map $f \colon X \to Y$ is a \emph{coarse embedding} if there exist real non-decreasing functions $\rho_1, \rho_2$ with $\mathrm{lim}_{t \to +\infty}\rho_1(t)= +\infty$, such that the inequality
\[ \rho_1(d_X(x_1, x_2))  \leqslant d_Y(f(x_1), f(x_2)) \leqslant \rho_2(d_X(x_1, x_2))  
\]
holds for all $x_1, x_2 \in X$. If both functions $\rho_1, \rho_2$ are affine, we call $f$ a \emph{quasi-isometric embedding}.
Given two maps $f,g \colon X \to Y$ we say that $f$ and $g$ are \emph{close} if there exists $R \geqslant 0$ such that the inequality $d_Y(f(x), g(x)) \leqslant R$ holds for all $x \in X$. We say that the coarse embedding $f \colon X \to Y $ is a \emph{coarse equivalence}, if there exists a coarse embedding $g \colon Y \to X$ such that the composite $g \circ f$ is close to the identity map on $X$ and $g \circ f$ is close to the identity map on $Y$.
Analogously, a quasi-isometric embedding $f  \colon X \to Y$ is called a \emph{quasi-isometry}, if there exists a quasi-isometric embedding $g \colon Y \to X$ such that the appropriate composites are close to identity maps.

The following criterion will be very useful to us: a coarse embedding (quasi-isometric embedding) $f \colon X \to Y$ is a coarse equivalence (quasi-isometry) if and only if it is \emph{quasi-onto}, i.e.\ there exists $R \geqslant 0$ such that for any $y \in Y$ there exists $x \in X$ with $d_Y(y, f(x)) \leqslant R$.

\section{Quasi-product structure of the minimal displacement set}\label{sec:minset}

In this section we describe the structure of the minimal displacement set associated to a hyperbolic isometry of $X$.
We prove that this subcomplex of $X$ is quasi-isometric to the product of the so-called \emph{graph of axes} and the real line.
 This may be viewed as a coarse version of the Product Decomposition Theorem for $\mathrm{CAT}(0)$ spaces (see \cite{BrHa}). Our arguments rely on the work of T.\ Elsner in \cite{E1} and \cite{E2}. \medskip

Let $h$ be an isometry of a simplicial complex $X$. Define the \emph{displacement function} $d_h\colon X^{(0)} \to \mathbb{N}$ by the formula $d_h(x)=d_X(x,h(x))$. The \emph{translation length} $\trl{h}$ is defined as
    \[ \trl{h}=\underset{x \in X^{(0)}}{\mathrm{min}}d_h(x).\smallskip\]
If $h$ does not fix any simplex of $X$, then $h$ is called \emph{hyperbolic}. In such case one has $\trl{h}>0$.
For a hyperbolic isometry $h$, define the \emph{minimal displacement set} $\minset(h)$ as the subcomplex of $X$ spanned by the set of vertices where $d_h$ attains its minimum. Clearly $\minset(h)$ is invariant under the action of $h$. If $X$ is a systolic complex, we have the following.

\begin{lem}\cite[Propositions 3.3 and 3.4]{E2}
\label{minissystolic}
Let $h$ be a hyperbolic isometry of a systolic complex $X$. Then the subcomplex $\minset(h)$ is a systolic complex,  isometrically embedded into $X$.

\end{lem}

An $h$--invariant geodesic in $X$ is called an \emph{axis} of $h$. We say that $\minset(h)$ is the \emph{union of axes}, if for every vertex $x \in \minset(h)$ there is an $h$--invariant geodesic passing through $x$, i.e.\ $\minset(h)$ can be written as follows
\begin{equation} 
\label{minaxes}
\minset(h)= \mathrm{span}\{\bigcup \g \mid \gamma \text{ is an $h$--invariant geodesic}\}. 
\end{equation}
In this case, the isometry $h$ acts on $\minset(h)$ as a translation along the axes by the number $\trl{h}$.

\begin{prop}\label{alltogetherlemma} Let $h$ be a hyperbolic isometry of a systolic complex $X$. Then the following hold:

\begin{enumerate}[label=(\roman*)]

	\item\cite[Proposition 3.1]{E2} For any $n \in \mathbb{Z} \setminus \{0\}$, the isometry $h^n$ is hyperbolic.

	\item \cite[Theorem 3.5]{E2}\label{mostisom} There exists an $n\geqslant 1$ such that there is an $h^n$--invariant geodesic in $X$.

	\item \cite[Remark, p.\ 48]{E2} \label{axialisometry} If there exists an $h$--invariant geodesic then for any vertex $x \in \minset(h)$ there is an $h$--invariant geodesic passing through $x$, i.e.\ the isometry $h$ satisfies \eqref{minaxes}.

	\item \label{axialminsetpower} If $h$ satisfies \eqref{minaxes} then so does $h^n$ for any $n \in \mathbb{Z} \setminus \{0\}$.

\end{enumerate}
\end{prop}

\begin{proof}Observe that \ref{axialminsetpower} follows from \ref{axialisometry} and the fact that an $h$--invariant geodesic is {$h^n$--invariant}.
\end{proof}

For two subcomplexes $X_1, X_2 \subset X$, the distance $d_{min}(X_1, X_2)$ is defined to be \[d_{min}(X_1, X_2)= \min \{d_X(x_1,x_2) \mid x_1 \in X_1, x_2 \in X_2\}.\] Note that in general $d_{min}$ is not even a pseudometric. We are ready now to define the graph of axes.

\begin{de}(Graph of axes). For a hyperbolic isometry $h$ satisfying \eqref{minaxes}, define the simplicial graph $Y(h)$ as follows
\begin{align*} Y(h)^{(0)} &= \{\gamma \mid \gamma \text{ is an $h$--invariant geodesic in } \minset(h)\},\\
  Y(h)^{(1)}     &=   \{  \{\gamma_1, \gamma_2 \} \mid d_{min}(\gamma_1, \gamma_2) \leqslant 1\}. \end{align*}

Let $d_{Y(h)}$\ denote the associated metric on $Y(h)^{(0)}$ (see Convention~\ref{metricconvention}). 
\end{de}

The main goal of this section is to prove the following theorem.

\begin{tw}
\label{coarsesplit}

Let $h$ be a hyperbolic isometry of a uniformly locally finite systolic complex $X$, such that the translation length $\trl{h}> 3$, and the subcomplex $\minset(h)$ is the union of axes. Then there is a quasi-isometry
\begin{equation} c \colon (Y(h) \times \mathbb{Z}, d_h) \to (\minset(h), d_X),\end{equation}
where the metric $d_h$ is defined as  \[d_h((\im{\g_1},t_1),(\im{\g_2},t_2))=  d_{Y(h)}(\im{\g_1},\im{\g_2})+ |t_1-t_2|,\] and $d_X$ is the metric induced from $X$.

\end{tw}

In the remainder of this section let $h$ be a hyperbolic isometry of $X$ such that $\minset(h)$ is the union of axes. By Lemma~\ref{minissystolic}  it is enough to prove Theorem~\ref{coarsesplit} in case where $\minset(h)=X$. In order to define the map $c$ we parameterise geodesics in $Y(h)$, i.e.\ to each $\gamma \in Y(h)$ we assign an origin $\gamma(0)$ and a direction. After this is done, the geodesic $\gamma$ can be viewed as an isometry $\gamma\colon \mathbb{Z} \to X$, and the map $c$ can be defined as $c(\im{\g}, t) = \g(0+t)$.
Before we describe the procedure of parameterising $\gamma$, we need to establish the following metric estimate between $d_{Y(h)}$ and $d_X$.

\begin{lem}
\label{estim1} Let $\im{\g_1}$ and $\im{\g_2}$ be $h$--invariant geodesics. Then:
\begin{enumerate}[label=(\roman*)]
 \item \label{11est} For any vertices $x_1 \in \im{\g_1}$ and $x_2 \in \im{\g_2}$ we have $d_{Y(h)}(\im{\g_1},\im{\g_2}) \leqslant d_X(x_1, x_2)+1$. 
\item \label{22est} For any vertex $x_1 \in \g_1$ there exists a vertex $x_2 \in \g_2$ with $d_X(x_1,x_2) \leqslant {(\trl{h} +1)} d_{Y(h)} (\im{\g_1},\im{\g_2})$.
\end{enumerate}
\end{lem}

\begin{proof}(i) We proceed by induction on $d_X(x_1,x_2)$. If $d_X(x_1, x_2)=0$ then $\im{\g_1}$ and $\im{\g_2}$ intersect, hence $d_{Y(h)}(\im{\g_1},\im{\g_2}) \leqslant 1$. Assume the claim is true for $d_X(x_i,x_j) \leqslant n$, and let $d_X(x_1,x_2)= n+1$. Let $x_n$ be the vertex on a geodesic in $X$ between $x_1$ and $x_2$, with $d_X(x_1,x_n)=n$.
% and $d_X(x_n,x_2)=1$. 
Choose a geodesic $\im{\g_n}\in Y(h)$ passing through $x_n$ (such a geodesic exists since $h$ satisfies (\ref{minaxes})). By inductive hypothesis we have $d_{Y(h)}(\im{\g_1},\im{\g_n})\leqslant n+1$ and clearly $d_{Y(h)}(\im{\g_n},\im{\g_2}) \leqslant 1$, hence the claim follows from the triangle inequality.

(ii) It suffices to prove the claim in the case where $d_{Y(h)}(\im{\g_1},\im{\g_2})=1$. Let $x_1 \in \im{\g_1}$ be any vertex. The vertex $x_2$ is chosen as follows. Let $x_1' \in \im{\g_1}$ be the vertex realising the distance between $\im{\g_1}$ and $\im{\g_2}$  (i.e.\ it is either the vertex of intersection, or the vertex on the edge joining these two geodesics). Choose $x_1'$ to be the closest vertex to $x_1$ with this property. Due to $h$--invariance of $\im{\g_1}$ and $\im{\g_2}$, the distance $d_X(x_1, x_1')$ is not greater than  $\trl{h}$  (even $\frac{\trl{h}}{2}$ in fact).
If $\im{\g_1}$ and $\im{\g_2}$ intersect, define $x_2$ to be $x_1'$. If not then $\g_1$ and $\g_2$ are connected by an edge, one of whose endpoints is $x_1'$. Define the vertex $x_2 \in \g_2$ to be the other endpoint of that edge.\end{proof}

For an $h$--invariant geodesic $\im{\g} \subset X$ define the linear order $\prec$ on the set of vertices of $\im{\gamma}$, by setting $x \prec h(x)$ for some (and hence all) $x\in \im{\g}$. Fix a geodesic $\im{\g_0}$ and identify the set of its vertices with $\mathbb{Z}$, such that the order $\prec$ agrees with the natural order on $\mathbb{Z}$.

Consider the family of combinatorial balls $\{B_{n}(-n,X)\}_{n \in \mathbb{N}}$, where $-n \in \im{\g_0}$. Notice that we have $B_{n}(-n,X) \subseteq B_{n+1}(-(n+1),X)$, i.e.\ the family $\{B_{n}(-n,X)\}_{n \in \mathbb{N}}$ is ascending. The following lemma is crucial in order to define the origin $\g(0)$ of $\gamma$.

\begin{lem}\label{lem:originwelldefined} Let $\gamma$ be an $h$--invariant geodesic. Then there exists a vertex $v \in \gamma$ such that for any vertex $w$ contained in the intersection \[ (\underset{n \geqslant 0}{\bigcup}B_{n}(-n,X)) \cap \im{\gamma}\] we have $w \prec v$. 
\end{lem}

\begin{proof}Observe first that for any $n \geqslant 0$ we have 
\[\underset{\prec}{\mathrm{sup}} \{ B_{n}(-n,X) \cap \im{\gamma_0} \}= 0 \in \gamma_0,\]
therefore taking $0$ as $v$ proves the lemma for $\gamma=\gamma_0$. For an arbitrary $\gamma$, we proceed by contradiction. Assume conversely, that the supremum $\underset{\prec}{\mathrm{sup}} \{ B_{n}(-n,X) \cap \im{\gamma}\}$ is not attained at any vertex of $\gamma$. 
%Therefore every vertex of $\gamma$ belongs to $B_{n}(-n,X)$ for some $n>0$.

Let $x \in \g$ be a vertex which is at distance at most $ K= (\trl{h} +1)d_{Y(h)}(\im{\g_0},\im{\g})$ from $0 \in \g_0$. Lemma~\ref{estim1}.\ref{22est} assures that such a vertex exists. %Let $K$ denote $(\trl{h} +1)d_{Y(h)}(\im{\g_0},\im{\g})$.
 For any $m >0$ consider vertices $h^m(x)$ and $h^m(0)$. By our assumption there exists $n>0$ such that $h^m(x)$ is contained in $B_{n}(-n,X)$ (see Figure~\ref{fig:origindefined}). Therefore by the triangle inequality we get \[d_X(-n, h^m(0)) \leqslant n+ K.\] On the other hand we have \[d_X(-n, h^m(0)) = n+ \trl{h^m},\] since $\g_0$ is a geodesic. For $\trl{h^m} >K$ this gives a contradiction.
\end{proof}

\begin{figure}[!h]
\centering
\begin{tikzpicture}[scale=1]

%gamma0

\draw (-6,0.75) to (6,0.75);
\node [above] at (-2.25, 0.75) {$n$};
\node [above] at (2.25, 0.75) {$\trl{h^m}$};
\node [above] at (5.5,0.75) {$\g_0$};

\draw[fill] (-5,0.75) circle [radius=0.075];
\node [below left] at (-5,0.75) {$-n$};

\draw[fill] (0.5,0.75) circle [radius=0.075];
\node [below right] at (0.5,0.75) {$0$};

\draw[fill] (4, 0.75) circle [radius=0.075];
\node [below right] at (4,0.75) {$h^m(0)$};

%gamma
\draw (-6,-1.5) to (6,-1.5);
\node [above] at (5.5, 2*-0.75) {$\g$};

\node [below] at (1.75, -2*0.75) {$\trl{h^m}$};

\draw[fill] (0, 2*-0.75) circle [radius=0.075];
\node [below right] at (0, 2*-0.75) {$x$};

\draw[fill] (3.5, 2*-0.75) circle [radius=0.075];
\node [below right] at (3.5,2*-0.75) {$h^m(x)$};

%geodesics

\draw [very thin] (-5, 0.75) to (3.5, 2*-0.75);
\node [below] at (-1, -0.3) {$\leqslant n$};

\draw [very thin] (0.5, 0.75) to (0, 2*-0.75);
\node [right] at (0.25, -0.25) {$K$};

\draw [very thin] (4, 0.75) to (3.5, 2*-0.75);
\node [right] at (3.75, -0.25) {$K$};

\end{tikzpicture}
\caption{The vertex $h^m(x)$ cannot belong to $B_{n}(-n,X)$ and be arbitrarily far from $x$ at the same time.}
\label{fig:origindefined}
\end{figure}
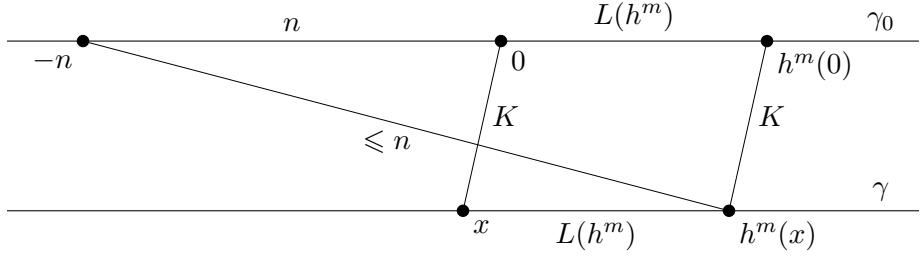

\begin{de}
\label{thedeforigin} Let $\im{\gamma} \subset X$ be an $h$--invariant geodesic. Define the vertex $\gamma(0)$ as 

\begin{equation}
\label{deforigin}
\gamma(0)= \underset{\prec}{\mathrm{max}} (\underset{n \geqslant 0}{\bigcup}B_{n}(-n,X) \cap \im{\gamma}).\smallskip
\end{equation}
The set of vertices of $(\underset{n \geqslant 0}{\bigcup}B_{n}(-n,X) \cap \im{\gamma})$ is bounded from above by Lemma~\ref{lem:originwelldefined} hence the maximal vertex exists. Observe that for $\gamma_0$ we have $\gamma_0(0)=0$.
\end{de}

Having geodesics parameterised, we need the following two technical lemmas that describe certain metric inequalities, which are needed to prove Theorem~\ref{coarsesplit}. 

\begin{lem}
\label{sysstripes} Let $\im{\gamma_0}, \im{\gamma}$ be as in Definition \ref{thedeforigin} and assume that $\trl{h}>3$ and $d_{Y(h)}(\im{\gamma_0}, \im{\gamma}) \geqslant 2 \trl{h}+4$. Then there exists an $n_0 \in \mathbb{N}$, such that for all $n\geqslant n_0$ and for all $t \geqslant 0$ we have 
\begin{equation}    n+t-1 \leqslant d_X(-n,\gamma(t)) \leqslant n+t+1.
\end{equation}
\end{lem}

\begin{proof}The idea is to reduce the problem to the study of $\mathbb{E}_{\Delta}^2$, the equilaterally triangulated Euclidean plane. To do so, we construct an $h$--equivariant simplicial map $f\colon P \to X$, which satisfies the following properties:
\begin{enumerate}

	\item \label{item1flat} the complex $P$ is an $h$--invariant triangulation of a strip $\mathbb{R}\times I$, and $P$ can be isometrically embedded into $\mathbb{E}_{\Delta}^2$,
	\item \label{item2flat} the boundary $\partial P$ is mapped by $f$ onto the disjoint union $\im{\gamma_0} \sqcup \im{\gamma}$ such that the restriction of $f$ to each boundary component is an isometry,
	\item \label{item3flat} for every pair of vertices $u,v \in P$ there is an inequality
	\[d_P(u,v)-1\leqslant d_X(f(u), f(v)) \leqslant d_P(u,v).\]

\end{enumerate} 

The construction of such a map is given in the proofs of Theorems 2.6 and 3.5 (case 1) in \cite{E2}. It requires that $\trl{h}>3$ and $d_{Y(h)}(\im{\gamma_0}, \im{\gamma}) \geqslant 2 \trl{h}+4$, and this is the only place where we need these assumptions.

Fix an embedding of $P$ into $\E$. Since the restriction of $f$ to each boundary component is an isometry, let us keep the same notation for the preimages under $f$ of $\im{\g_0}$ and $\im{\g}$. For a vertex $v \in \partial P$ let $\angle(v)$ denote the number of triangles of $P$ which contain $v$.
We have the following two cases to consider:

\begin{enumerate}
\item[(i)] for every vertex $v$ of $\im{\g}$ we have $\angle(v)=3$,
\item[(ii)] there exists a vertex $v$ of $\im{\g}$ with $\angle(v) \neq 3$.
\end{enumerate}
We treat the case (ii) first. Steps of the proof are indicated in Figure~\ref{fig:strip}. Let $v_0$ be a vertex of $\im{\g}$ such that $\angle(v_0)=2$ and $v_0 \prec \g(0)$. %[(such a vertex exists because..)]. 
Denote by $v_{-1}$ the vertex of $\g$ that is adjacent to $v_0$ and $v_{-1} \prec v_0$. Introduce a coordinate system on $\E$ by setting $v_{-1}$ to be the base point and letting $e_1=\overrightarrow{v_{-1} v_0}$ and $e_2=\overrightarrow{v_{-1} w}$, where $w$ is the unique vertex which lies inside $P$ and is adjacent to both $v_{-1}$ and $v_0$. We will write $v=(x_v, y_v)$ for the coordinates of a vertex $v$ in basis ${e_1, e_2}$. 
It follows from the choice of the coordinate system and the fact that the strip $P$ is $h$--invariant, that for all $k\in \mathbb{Z}$ we have 
\begin{equation}
\label{increasinggeo}
\g(k+1) = \g(k) + e_1 \text{ or } \g(k+1) = \g(k) + e_2,
\end{equation}
and both possibilities occur an infinite number of times. In particular, the second coordinate $y_{\g(k)}$ of $\g(k)$ is a non-decreasing function of $k$ such that \begin{equation}\label{increasinggeo2}y_{\g(k)}  \rightarrow -\infty \text{ as } k \rightarrow -\infty.\end{equation}
By property (3) of the map $f$, the distance $d_{\E}(\im{\g_0}, \im{\g})$ is bounded from above by $d_X(\im{\g_0}, \im{\g})+1$. 
This implies that the geodesic $\g_0$ also satisfies \eqref{increasinggeo} and \eqref{increasinggeo2} (i.e.\ it runs parallel to $\g$, see Figure~\ref{fig:strip}).
Hence, there exists $k_0$, such that for $k \geqslant k_0$ the coordinate $y_{\g_0(-k)}$ is strictly less than $y_{v_0}$, which is in turn less than $y_{\gamma(0)}$, since $v_0 \prec \g(0)$. Therefore for each $r \geqslant d_P(\g_0(-k),v_0)$ the combinatorial sphere $S(\g_0(-k),r)$ intersects $\im{\g}|_{[0, \infty)}$ exactly once, where $\im{\g}|_{[0, \infty)}$ denotes the geodesic ray obtained by restricting the domain of $\g$ to non-negative integers. %(argument with the slope parallel to $e_2-e_1$).
In particular, if $S(\g_0(-k),r) \cap \im{\g}|_{[0, \infty)}=\gamma(s)$ for some $s\geqslant 0$, then $S(\g_0(-k),r+1) \cap \im{\g}|_{[0, \infty)}=\gamma(s+1)$. 

Take $\widetilde{n_0}$ such that for any $n \geqslant \widetilde{n_0}$ we have $\gamma(0)=\underset{\prec}{\mathrm{max}} (B(-n,n) \cap \im{\gamma})^{(0)}$. By property (3) of the map $f$, for such $n$ we have $d_{\E}(-n, \gamma(0)) \in \{n, n+1 \}$. Set $n_0 =\mathrm{max}\{\widetilde{n_0}, k_0\}$. For $n\geqslant n_0$ and for $t \geqslant 0$ we have $d_{\E}(-n, \gamma(t))= d_{\E}(-n, \gamma(0))+t \in \{n+t, n+1+t\}$. Therefore the claim follows from property (3).
\vspace{0.25cm}

Case (i) is proven analogously. Using the notation of case (ii) we introduce the coordinate system as follows. Put $w=\g(0)$, $ v_{-1}=\g(-1)$ and let $v_0$ be the unique vertex which lies outside of $P$ and is adjacent to both $v_{-1}$ and $w$. The rest of the proof is the same as in case (ii).
\end{proof}

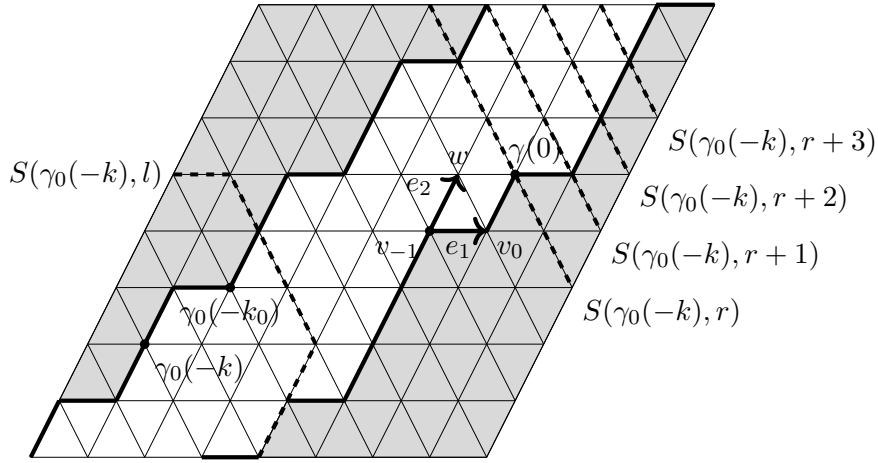
\begin{figure}[!h]
\centering

%filling
 \begin{tikzpicture}[scale=0.75]

 \definecolor{lgray}{rgb} {0.850,0.850,0.850}
\draw [ultra thin, fill=lgray] (-2,-4) -- (-1.5,-3) -- (-0.5, -3) -- (0,-2) -- (1,0) -- (2,0) -- (2.5,1) -- (3.5,1) -- (4,2) --(5,4) -- (6,4)-- (2,-4)-- (-2,-4);

\draw [ultra thin, fill=lgray]  (-6, -4)--(-5.5,-3) --(-4.5,-3) -- (-3.5, -1)-- (-2.5, -1) -- (-1.5, 1) -- (-0.5, 1) -- (0.5, 3) -- (1.5, 3) -- (2, 4)-- (-2,4)--(-6,-4);

%%I dont understand hy i put this 'node' here, but i changed it so it fits more, before it was (1.19,0) 
\node at (1.235,0){
\begin{tikzpicture}[scale=0.75]

\node [below right] at (2,0) {$v_0$};

\draw[fill] (1,0) circle [radius=0.075];
\node [below left] at (1,0) {$v_{-1}$};

%\draw[fill] (1.5,1) circle [radius=0.075];
\node [above ] at (1.5,1) {$w$};

\draw[fill] (2.5,1) circle [radius=0.075];
\node [ above right] at (2.2,1) {$\g(0)$};
%\node [below] at (0.45,-1.4) {$\gamma$};

%coordinates

%\draw [ultra thick, dotted, ->] (1,0) to (1.925,0);
%\draw [ultra thick, dotted, ->] (1,0) to (1 +0.475, 2.236*0.425);
\draw [ultra thick, ->] (1,0) to (2,0);
\draw [ultra thick, ->] (1,0) to (1 +0.5, 1);
\node [below] at (1.5,0) {$e_1$};
\node [above left] at (1.225,0.45)  {$e_2$};

%points sigma0
\draw[fill] (-2.5,-1) circle [radius=0.075];
\node [below ] at (-2.5,-1) {$\gamma_0(-k_0)$};

\draw[fill] (-4,-2) circle [radius=0.075];
\node [below right] at (-4,-2) {$\gamma_0(-k)$};

%spheres

\draw [line width=1.5, dashed] (-2+4,4) to (4,0);
\node [below right] at (4,0) {$S(\gamma_0(-k), r+1)$};

\draw [line width=1.5, dashed] (3,4) to (+4.5,1);
\node [below right] at (4.5,1) {$S(\gamma_0(-k), r+2)$};

\draw [line width=1.5, dashed] (4,4) to (+5,2);
\node [below right] at (5,2) {$S(\gamma_0(-k), r+3)$};

\draw [line width=1.5, dashed] (-2+3,4) to (3.5,-1);
\node [below right] at (3.5,-1) {$S(\gamma_0(-k), r)$};
%\node [below right] at (3.5,-1.75) {$r=d_{\E}(\gamma_0(-k), \gamma(0))$};

%badsphere
\draw [line width=1.5, dashed] (-1.5,-3) to (-1,-2);
\draw [line width=1.5, dashed] (-1,-2) to (-2.5,1);
\draw [line width=1.5, dashed] (-2.5,1) to (-3.5,1);
\node [left] at (-3.5,1) {$S(\gamma_0(-k), l)$};

%sigma

\draw [ultra thick] (-3,-4) -- (-2,-4);

\draw [ultra thick, dashed] (-2,-4) -- (-1.5,-3);

 \draw [ultra thick] (-1.5,-3) -- (-0.5, -3) -- (0,-2) -- (1,0) -- (2,0) -- (2.5,1) -- (3.5,1) -- (4,2) --(5,4) -- (6,4);

%sigma0
\draw [ultra thick] (-6, -4) to (-5.5,-3);
\draw [ultra thick] (-5.5,-3) to (-4.5, -3);
\draw [ultra thick] (-4.5,-3) to (-3.5, -1);
\draw [ultra thick] (-3.5,-1) to (-2.5, -1);
\draw [ultra thick] (-2.5,-1) to (-1.5, 1);
\draw [ultra thick] (-1.5,1) to (-0.5, 1);
\draw [ultra thick] (-0.5,1) to (0.5, 3);
\draw [ultra thick] (0.5,3) to (1.5, 3);
\draw [ultra thick] (1.5,3) to (2, 4);

%background

\draw [ultra thin] (-1.5,-3) to (1.5,3);
\draw [ultra thin] (-1.5-1,-3) to (1.5-1,3);
\draw [ultra thin] (-1.5-2,-3) to (1.5-2,3);
\draw [ultra thin] (-1.5-3,-3) to (1.5-3,3);
\draw [ultra thin] (-1.5+1,-3) to (1,0);
\draw [ultra thin] (1.5,1) to (1.5+1,3);

\draw [ultra thin] (-1.5+2,-3) to (1.5+2,3);
\draw [ultra thin] (-1.5+3,-3) to (1.5+3,3);
\draw [ultra thin] (-1.5+4,-3) to (1.5+4,3);
\draw [ultra thin] (-1.5-4,-3) to (1.5-4,3);

\draw [ultra thin] (-1.5,3) to (1.5,-3);
\draw [ultra thin] (-4,0) to (1.5-4,-3);
\draw [ultra thin] (-3.5,1) to (1.5-3,-3);
\draw [ultra thin] (-3,2) to (1.5-2,-3);
\draw [ultra thin] (-1.5-1,3) to (1.5-1,-3);
\draw [ultra thin] (-1.5+1,3) to (1.5+1,-3);
\draw [ultra thin] (-1.5+2,3) to (3,-2);
\draw [ultra thin] (-1.5+3,3) to (3.5,-1);
\draw [ultra thin] (-1.5+4,3) to (4,0);
\draw [ultra thin] (3.5,3) to (+4.5,1);
\draw [ultra thin] (4.5,3) to (+5,2);

\draw [ultra thin] (-3.5,-3) to (-4.5,-1);
\draw [ultra thin] (-4.5,-3) to (-5,-2);

\draw [ultra thin] (-4,0) to (+1,0);
\draw [ultra thin] (2,0) to (4,0);

\draw [ultra thin] (-3.5,0+1) to (+4.5,0+1);
\draw [ultra thin] (-3,2) to (+5,2);
\draw [ultra thin] (-2.5,3) to (+5.5,3);
\draw [ultra thin] (-4.5,-1) to (3.5,-1);
\draw [ultra thin] (-5,-2) to (+3,-2);
\draw [ultra thin] (-5.5,-3) to (+2.5,-3);

%additinal background outer

%%top

\draw [ultra thin] (-2,4) to (2, 4);

\draw [ultra thin] (-2.5, 3) to (-2,4);
\draw [ultra thin] (-2, 4) to (-1.5,3);
 \draw [ultra thin] (-1.5,3) to (-1, 4);
 \draw [ultra thin] (-1,4) to (-0.5, 3);
 \draw [ultra thin] (-0.5,3) to (0,4);
 \draw [ultra thin] (0,4) to (0.5, 3);
 \draw [ultra thin] (0.5,3) to (1,4);
 \draw [ultra thin] (1,4) to (1.5,3);
 \draw [ultra thin] (1.5,3) to (2,4);

\draw [ultra thin] (5, 4) to (5.5,3);
\draw [ultra thin] (5.5, 3) to (6,4);

%%bottom

\draw [ultra thin] (5-3.5, 4-7) to (5.5-3.5,3-7);
\draw [ultra thin] (5.5-3.5, 3-7) to (6-3.5,4-7);

\draw [ultra thin] (2-4,4-8) to (6-4, 4-8);

\draw [ultra thin] (2-3.5,4-7) to (2.5-3.5, 3-7);
\draw [ultra thin] (2.5-3.5,3-7) to (3-3.5,4-7);
\draw [ultra thin] (3-3.5,4-7) to (3.5-3.5, 3-7);
\draw [ultra thin] (3.5-3.5,3-7) to (4-3.5,4-7);
\draw [ultra thin] (4-3.5,4-7) to (4.5-3.5, 3-7);
\draw [ultra thin] (4.5-3.5, 3-7) to (5-3.5,4-7);

\draw [ultra thin] (-6,4-8) to (-3, 4-8);

%additinal background inner

%top
\draw [ultra thin] (2,4) to (6, 4);

\draw [ultra thin] (2,4) to (2.5, 3);
\draw [ultra thin] (2.5,3) to (3,4);
\draw [ultra thin] (3,4) to (3.5, 3);
\draw [ultra thin] (3.5,3) to (4,4);
\draw [ultra thin] (4,4) to (4.5, 3);
\draw [ultra thin] (4.5, 3) to (5,4);

%bottom

\draw [ultra thin] (-2.5-3.5, 3-7) to (-2-3.5,4-7);
\draw [ultra thin] (-2-3.5, 4-7) to (-1.5-3.5,3-7);
 \draw [ultra thin] (-1.5-3.5,3-7) to (-1-3.5, 4-7);
 \draw [ultra thin] (-1-3.5,4-7) to (-0.5-3.5, 3-7);
 \draw [ultra thin] (-0.5-3.5,3-7) to (0-3.5,4-7);
 \draw [ultra thin] (0-3.5,4-7) to (0.5-3.5, 3-7);
 \draw [ultra thin] (0.5-3.5,3-7) to (1-3.5,4-7);
 \draw [ultra thin] (1-3.5,4-7) to (1.5-3.5,3-7);
 \draw [ultra thin] (1.5-3.5,3-7) to (2-3.5,4-7);

\end{tikzpicture}
};

\end{tikzpicture}

\caption{Geodesics and combinatorial spheres in the coordinate system.}
\label{fig:strip}
\end{figure}

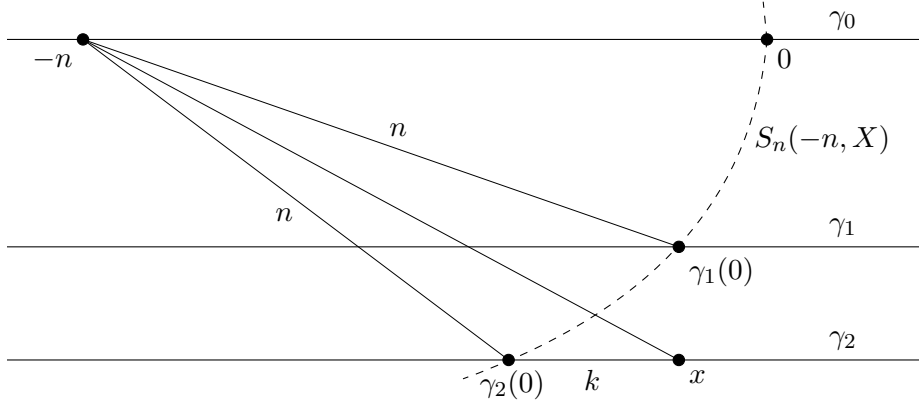
\begin{figure}[!h]
\centering
\begin{tikzpicture}[scale=1]

%gamma0

\draw (-6,2) to (6,2);
\node [above] at (5,2) {$\g_0$};

\draw[fill] (-5,2) circle [radius=0.075];
\node [below left] at (-5,2) {$-n$};

\draw[fill] (4,2) circle [radius=0.075];
\node [below right] at (4,2) {$0$};

%gamma1
\draw (-6,-0.75) to (6,-0.75);
\node [above] at (5,-0.75) {$\g_1$};

\draw[fill] (2.84,-0.75) circle [radius=0.075];
\node [below right] at (2.84,-0.75) {$\g_1(0)$};

%gamma2
\draw (-6,-2.25) to (6,-2.25);
\node [above] at (5,-2.25) {$\g_2$};

\draw[fill] (0.6,-2.25) circle [radius=0.075];
\node [below] at (0.65,-2.25) {$\g_2(0)$};

\draw[fill] (2.84,-2.25) circle [radius=0.075];
\node [below right] at (2.84,-2.25) {$x$};

%circle
%\draw[dotted] (4,2) to [out=250,in=0] (-5,-7.5);
\draw[dashed] (4,2) to [out=265,in=20] (0,-2.5);

\draw[dashed] (4,2) to  (3.95,2.5);

\node [below right] at (3.70, 1) {$S_n(-n,X)$};

% lines
\draw[very thin] (-5,2) to (2.84,-0.75);
\node [below] at (-2.35, -0.1) {$n$};

\draw[very thin] (-5,2) to (2.84,-2.25);
\node [below] at (1.7,-2.25) {$k$};

\draw[very thin] (-5,2) to (0.6,-2.25);
\node at (-0.85, 0.8) {$n$};

\end{tikzpicture}
\caption{Geodesics and their origins.}
\label{fig:comparingorigins}
\end{figure}

We need one more metric estimate.

\begin{lem}
\label{lemupperbound}

Let $\im{\g_1}, \im{\g_2}$ be $h$--invariant geodesics in $X$. Assume additionally, that the translation length $\trl{h} > 3$. Then we have the following inequality 
\[d_X (\g_1(0), \g_2(0)) \leqslant 
2(\trl{h} +1)d_{Y(h)}(\im{\g_1} ,\im{\g_2}) + 4(2\trl{h}+4)(\trl{h} +1).
\]
\end{lem}
\begin{proof}

Let $\trol$ denote the translation length $\trl{h}$ and let $D$ denote the distance $d_{Y(h)}(\im{\g_1} ,\im{\g_2})$. We claim that there exists a vertex $x \in \im{\g_2}$ such that \[d_X(x, \g_1(0)) \leqslant (\trol +1)D +\trol\] and $\g_2(0)\prec x$. By Lemma~\ref{estim1}.\ref{22est} there is a vertex $x \in \im{\g_2}$ which is at distance at most $(\trol +1)D$ from $\g_1(0)$. If $\g_2(0)\prec x$ does not hold, we do the following. Let $\alpha$ be a geodesic joining $\g_1(0)$ and $x$. Apply the isometry $h^m$ to $\alpha$, where $m$ is the smallest integer such that $h^m(x) \succ \g_2(0)$. Then the concatenation of a geodesic segment joining $\g_2(0)$ and $h^m(x)$ with geodesic $h^m(\alpha)$ is a path joining $\g_2(0)$ and $h^m(\g_1(0))$ of length at most $(\trol +1)D +\trol$. Therefore we can switch roles of $\im{\g_1}$ and $\im{\g_2}$ and set $x$ to be $h^m(\gamma_1(0))$. This proves the claim.

Let $k$ denote the distance $d_X(\g_2(0), x)$. We will show that $k \leqslant d_X(x, \g_1(0)) +1$ (see Figure~\ref{fig:comparingorigins}). In order to apply Lemma~\ref{sysstripes}, assume that $d_Y(\im{\gamma_0}, \im{\gamma_2}) \geqslant 2 \trol+4$, and choose $n$ large enough, such that $\g_i(0)= \underset{\prec}{\mathrm{max}} ( B_n(-n,X) \cap \im{\g_i})^{(0)}$ for $i \in \{1,2\}$ and $n \geqslant n_0$, where $n_0$ is the constant appearing in the formulation of Lemma~\ref{sysstripes} (clearly the same holds for any $n_1 \geqslant n$). By Lemma \ref{sysstripes} we have \[d_X(-n,x) \geqslant n+k-1.\] 
By the triangle inequality applied to the vertices $-n, \g_1(0)$ and $x$ we get \[d_X(-n,x) \leqslant n+ d_X(\gamma_1(0), x).\]
Combining the two above inequalities gives us
\[k \leqslant d_X(\gamma_1(0), x) +1.\] By the triangle inequality the distance $d_X(\g_1(0), \g_2(0))$ is at most $d_X(\g_1(0),x)+k$ hence we have \begin{equation*}d_X(\g_1(0), \g_2(0)) \leqslant 2d_X(\g_1(0),x)+1 \leqslant  2(\trol +1)D +2 \trol +1.\end{equation*}
This proves the lemma under the assumption that $d_Y(\im{\gamma_0}, \im{\gamma_2}) \geqslant 2 \trol+4$.

If for both $\g_1$ and $\g_2$ we have $d_{Y(h)}(\im{\gamma_0}, \im{\gamma_i}) \leqslant 2 \trol+4$, then by directly comparing $\g_i$ with $\g_0$, one can show that \[d_X (\g_0(0), \g_i(0)) \leqslant 2   d_{Y(h)}(\im{\g_0} ,\im{\g_i})   (\trol +1)\] for $i \in \{ 1,2 \}$. Using the  triangle inequality one gets \[d_X (\g_1(0), \g_2(0)) \leqslant 2   d_{Y(h)}(\im{\g_0} ,\im{\g_1})   (\trol +1) + 2 d_{Y(h)}(\im{\g_0} ,\im{\g_2})   (\trol +1) \leqslant 4(2\trol+4)(\trol +1). \qedhere\]
\end{proof}
We are now ready to prove Theorem~\ref{coarsesplit}.

\begin{proof}[Proof of Theorem~\ref{coarsesplit}] 

Define the map $c \colon (Y(h) \times \mathbb{Z}, d_h) \to (\minset(h), d_X)$ as
\[ c(\gamma, t) = \gamma(0+t),   \]
where $\gamma(0+t)$ is the unique vertex of $\gamma$ satisfying $ \gamma(0) \prec \gamma(0+t)$ and $d_X(\gamma(0),\gamma(0+t))=t$. We will show that for every two points $(\g_1, t_1)$ and $(\g_2, t_2)$ of $Y(h) \times \mathbb{Z}$ we have the following inequality
\[\rho_1(d_h((\g_1,t_1), (\g_2, t_2))) \leqslant d_X(c(\g_1, t_1), c(\g_2, t_2)) \leqslant \rho_2(d_h((\g_1,t_1), (\g_2, t_2))),\]
where $\rho_1$ and $\rho_2$ are non-decreasing linear functions. \medskip

We first find the function $\rho_2$. %The same as in the Lemma~\ref{lemupperbound}, 
Without loss of generality we can assume that $|t_1| \leqslant |t_2|$ and let $\alpha$ be a  geodesic joining $\g_1(t_1)$ and $\g_2(t_2)$. Denote $\trl{h}$ by $\trol$ and apply the isometry $h^m$ to $\alpha$, where $m$ is chosen such that $d_X(\g_1(0), h^m(\g_1(t_1))) \leqslant \trol$, and $m$ has the smallest absolute value among such numbers. We then have $d_X(h^m(\g_2(t_2)), \g_2(0) ) \leqslant |t_1 - t_2| + \trol$.  Hence, by the triangle inequality we get
\[d_X(\g_1(t_1), \g_2(t_2))= d_X(h^m(\g_1(t_1)), h^m(\g_2(t_2))) \leqslant |t_1-t_2| + \trol +\trol + d_X(\g_1(0), \g_2(0)).\]
By Lemma~\ref{lemupperbound} we obtain 
\begin{equation*}
\begin{split}
			d_X(\g_1(t_1), \g_2(t_2))  & \leqslant |t_1-t_2| + 2\trol+ 2(\trol +1)d_{Y(h)}(\im{\g_1} ,\im{\g_2}) + 4(2\trol+4)(\trol +1) \\
&  \leqslant 2(\trol +1)( d_{Y(h)}(\im{\g_1} ,\im{\g_2}) +|t_1-t_2| ) + 4(2\trol+4)(\trol +1) +2\trol\\
& = 2(\trol +1)d_h((\g_1,t_1), (\g_2, t_2)) + 4(2\trol+4)(\trol +1) +2\trol.
\end{split}
\end{equation*}

We are left with finding $\rho_1$.
Let $K$ denote the distance $d_{Y(h)}(\im{\g_1} ,\im{\g_2}) +|t_1-t_2|$. 
If $d_{Y(h)}(\im{\g_1} ,\im{\g_2}) \geqslant \frac{1}{10(2(\trol +1))}K$, then by Lemma~\ref{estim1}.\ref{11est} we have 
\begin{equation}
\label{eqstep21}
d_X(\g_1(t_1), \g_2(t_2)) \geqslant d_{Y(h)}(\im{\g_2} ,\im{\g_2})-1 \geqslant \frac{1}{10(2(\trol +1))}K-1.
\end{equation}
If $d_{Y(h)}(\im{\g_2} ,\im{\g_2}) < \frac{1}{10(2(\trol +1))}K$ then one has $|t_1-t_2| \geqslant \frac{1}{2}K$.
In this case, using the same argument with translation by $h^m$ and the reverse triangle inequality, we obtain
\begin{equation}
\label{eqstep2}
d_X(\im{\g_1}(t_1) ,\im{\g_2}(t_2)) +\trol \geqslant |t_1-t_2| -\trol - d_X(\g_1(0), \g_2(0)).  
\end{equation}
By Lemma~\ref{lemupperbound} we have \[d_X(\g_1(0), \g_2(0)) \leqslant 2(\trol +1)d_{Y(h)}(\im{\g_2} ,\im{\g_2}) + 4(2\trol+4)(\trol +1),\] 
which combined with our assumption gives 
\[d_X(\g_1(0), \g_2(0)) \leqslant 2(\trol +1)\frac{1}{10(2(\trol \splus  1))}K + 4(2\trol+4)(\trol +1)= \frac{1}{10}K + 4(2\trol+4)(\trol +1). \]
Putting this in the inequality~\eqref{eqstep2} yields
\begin{equation*}
\label{eqstep22}
d_X(\im{\g_1}(t_1), \im{\g_2}(t_2)) +\trol \geqslant \frac{1}{2}K -\trol - \frac{1}{10}K- 4(2\trol+4)(\trol +1),
\end{equation*}
hence finally
\begin{equation}
\label{eqstep23}
d_X(\im{\g_1}(t_1), \im{\g_2}(t_2))  \geqslant \frac{9}{20}K -2\trol - 4(2\trol+4)(\trol +1).
\end{equation}
In both \eqref{eqstep21} and \eqref{eqstep23}, the following inequality holds
\begin{equation*} 
d_X(\im{\g_1}(t_1), \im{\g_2}(t_2))  \geqslant \frac{1}{20(L+1)}K -2\trol - 4(2\trol+4)(\trol +1).
\end{equation*}

This proves that $c$ is a quasi-isometric embedding. We note that by Proposition~\ref{alltogetherlemma}.\ref{axialisometry} the map $c$ is surjective on the vertex sets, hence a quasi-isometry. 
\end{proof}

\section{Filling radius for spherical cycles}\label{sec:fillingradius}

The purpose of this section is to prove that the graph of axes defined in Section~\ref{sec:minset} is quasi-isometric to a simplicial tree. Our main tool is the $\frc{k}$ property of systolic complexes introduced in \cite{JS3}. It is a coarse (and hence quasi-isometry invariant) property that, intuitively, describes ``asymptotic thinness of spheres'' in a given metric space.
We use numerous features of $\frc{k}$ spaces established in \cite{JS3}, some of which we adjust to our setting. 
The crucial observation is Proposition~\ref{s1frcquasitree}, which says that an $\frc{1}$ space satisfying certain homological condition is quasi-isometric to a tree. This extends a result of \cite{JS3}, which treats only the case of finitely presented groups.\medskip

Let $(X, d)$ be a metric space. Given $r >0$, the \emph{Rips complex} $P_r(X)$ is a simplicial complex defined as follows. The vertex set of $P_r(X)$ is the set of all points in $X$. The subset $\{x_1, \ldots, x_n\}$ spans a simplex of $P_r(X)$ if and only if $d(x_i,x_j) \leqslant r$ for all $i,j\in \{1, \ldots, n\}$. Notice that if $R \geqslant r$ then $P_r(X)$ is naturally a subcomplex of $P_R(X)$.\medskip

In what follows we consider simplicial chains with arbitrary coefficients. For detailed definitions see~\cite{JS3}. \medskip

A $k$\emph{--spherical cycle} in a simplicial complex $X$ is a simplicial map $f \colon S^k \to X$ from an oriented simplicial $k$--sphere to $X$. Let $C_f$ denote the image through $f$ of the fundamental (simplicial) $k$--cycle in $S^k$. A \emph{filling} of a $k$--spherical cycle $f$ is a simplicial {$(k+1)$--chain} $D$ such that $\partial D= C_f$. Let $\mathrm{supp}(f)$ denote the image through $f$ of the vertex set of $S^k$, and let $\mathrm{supp}(D)$ denote the set of vertices of all underlying simplices of $D$.

\begin{de}A metric space $(X,d)$ \emph{has filling radius for spherical cycles constant} (or $(X,d)$ is $\frc{k}$) if for every $r >0$ there exists $R \geqslant r$ such that any $k$--spherical cycle $f$ which is null-homologous in $P_r(X)$ has a filling $D$ in $P_R(X)$ satisfying $\mathrm{supp}(D) \subset \mathrm{supp}(f)$.
\end{de}

\begin{prop}\cite[p.\ 16]{JS3}
\label{frcgeometricproperty} Let $(X, d_X)$ be $\frc{k}$ and let $f\colon (Y, d_Y) \to (X, d_X)$ be a coarse embedding. Then $(Y, d_Y)$ is $\frc{k}$.
\end{prop}

\begin{lem}
\cite[Theorem 4.1, Lemma 5.3]{JS3}
\label{systolicisfrc} Let $X$ be a systolic complex. Then $X$ is $\frc{k}$ for any $k\geqslant 2$.
\end{lem}

The following lemma describes the behaviour of property $\frc{k}$ with respect to products.
It was originally proved in \cite{JS3} only for products of finitely generated groups. However, it is straightforward to check that the lemma holds for arbitrary geodesic metric spaces. The metric on a product is chosen to be the sum of metrics on the factors. 

\begin{lem}\cite[Proposition 7.2]{JS3}
\label{sfrcproducts} Let $k \in \{0,1\}$. Assume that $(X, d_X)$ is not $\frc{k}$ and that there is $r>0$
such that every $k$--spherical cycle $f\colon S^k \to X$ is null-homotopic in $P_r(X)$. If $(Y, d_Y)$ is unbounded then the product $(X, d_X) \times (Y, d_Y)$ is not $\frc{k+1}$.
\end{lem}

The following criterion is the key tool that we use in the proof of Proposition~\ref{s1frcquasitree}.

\begin{prop}\cite[Theorem 4.6]{MJ} 
\label{bottleneck}
Let $(X,d)$ be a geodesic metric space. Then the following are equivalent:
\begin{enumerate} 
\item $X$ is quasi-isometric to a simplicial tree, 
\item (bottleneck property) there exists $\delta >0$, such that for any two points $x,y\in X$ there is a midpoint $m=m(x,y)$ with $d(x,m)=d(m,y)=\frac{1}{2}d(x,y)$, and such that any path from $x$ to $y$ in $X$ contains a point within distance at most $\delta$ from $m$.

\end{enumerate}
\end{prop}

\begin{prop}
\label{s1frcquasitree}
Let $X$ be a graph which is $\frc{1}$ and assume that there exists $r>0$, such that any $1$--spherical cycle in $P_1(X^{(0)})$ is {null-homologous} in $P_r(X^{(0)})$. Then $X$ is quasi-isometric to a simplicial tree.
%Moreover if $X$ is uniformly locally finite, then the tree can be chosen so. [follows from Manning's construction]
\end{prop}

Note that to use Proposition~\ref{bottleneck} formally we need to consider a geodesic metric $d_g$ on $X$ -- see Remark~\ref{rem:twometrics}. Proposition~\ref{s1frcquasitree} will be true for our standard metric $d$ as well, since clearly $(X^{(0)}, d)$ is quasi--isometric to $(X, d_g)$.

\begin{proof}
Let $R\geqslant r$ be such that every $1$--spherical cycle $f \colon S^1 \to X$ that is null-homologous in $P_r(X)$ has a filling $D$ in  $P_R(X)$ with $\mathrm{supp}(D) \subset \mathrm{supp}(f)$. 

We proceed by contradiction. Suppose that $X$ is not quasi-isometric to a tree. Let $\delta$ be a natural number larger than $5R$. Then, by the bottleneck property  (Proposition~\ref{bottleneck}), there exist two vertices $v $ and $w$, a midpoint $m$ between them, and a path $\alpha$ between $v$ and $w$ omitting $B_{\delta}(m,X)$. Without loss of generality we  can assume that $m$ is a vertex. Let $\gamma$ denote a geodesic between $v$ and $w$ that contains $m$.

Let $a=\lceil 3\delta/4 \rceil$ and $b=\lfloor \delta/2 \rfloor$. We define subcomplexes $A=P_R(\gamma \cap B_a(m,X))$, and $B=P_R(\alpha \cup (\gamma \setminus B_b(m,X))$ of $P_R(X)$. We claim that the following hold:

\begin{itemize}
\item[(1)] $A$ and $B$ are path-connected,
\item[(2)] $A \cap B$ has the homotopy type of two points,
\item[(3)] $A \cup B = P_R(\alpha \cup \gamma )$.
\end{itemize}
Assertion (1) is straightforward. For (2) observe that \[A \cap B = P_R(\gamma \cap (B_a(m,X) \setminus B_b(m,X)))\] and that $\gamma \cap (B_a(m,X) \setminus B_b(m,X))$ consists of two geodesic segments that are separated by at least $2b > R$. The Rips complex of a geodesic segment is easily seen to be contractible. For (3) we clearly have $A \cup B \subset P_R(\alpha \cup \gamma )$. To prove the other inclusion we need to show that for any edge in $P_R(\alpha \cup \gamma )$ both of its endpoints are either in $A$ or in $B$. This follows from the definition of $A$ and $B$, as for any two vertices $x$ and $y$ with $x \in A\setminus B$ and $y \in B\setminus A$ we have $d(x,y) \geqslant a-b>R$.

 %subsets $A^{(0)} \setminus B^{(0)}$ and $B^{(0)}\setminus A^{(0)}$ are separated by at least $a-b >R$.

%$\gamma$ is a geodesic, and that there are no edges in $P_R(X^{(0)})$ between vertices in $\alpha$ and vertices in $A$.

Now let $\ov{\alpha}$ and $ \ov{\gamma}$ be the continuous paths obtained from $\alpha$ and $\gamma$ by connecting
consecutive vertices by edges. Let $\ov{\alpha} \ov{\gamma}$ be the $1$--spherical cycle in $P_1(X^{(0)})$ obtained by their concatenation. By our assumption the cycle $\ov{\alpha}\ov{\gamma}$ has a filling $D$ in $P_R(\mathrm{supp}(\ov{\alpha}\ov{\gamma}))=P_R(\alpha \cup \gamma)= A \cup B$ and thus $[\ov{\alpha}\ov{\gamma}]=0$ in $H_1(A\cup B)$.
However, in the Mayer-Vietoris sequence for the pair $A, B$ the boundary map 
\[H_1(A\cup B) \to H_0(A\cap B)\]
 sends $[\ov{\alpha}\ov{\gamma}]$ to a non-zero element. This gives a contradiction and hence finishes the proof of the proposition.
\end{proof}

Finally we are ready to prove the main result of this section.

\begin{cor}
\label{s1frcquasitreecoro} For a hyperbolic isometry $h$ whose minimal displacement set is a union of axes
(that is, for $h$ satisfying \eqref{minaxes}) and $L(h)>3$, the graph of axes $(Y(h), d_{Y(h)})$ is quasi-isometric to a simplicial tree. %\marginpar{of a uniformly finite degree.}
\end{cor}
\begin{proof}We will show that $(Y(h), d_{Y(h)})$ satisfies the assumptions of Proposition~\ref{s1frcquasitree}. 
	
First, we show that there exists an $r>0$, such that any $1$--spherical cycle in $P_1(Y(h)^{(0)})$ is null-homotopic in $P_r(Y(h)^{(0)})$. Let $f \colon S^1 \to P_1(Y(h)^{(0)})$ be such a cycle. We will show that $f$ is null-homotopic in $P_2(Y(h)^{(0)})$ by constructing a simplicial map $p \colon \minset(h) \to P_1(Y(h)^{(0)})$ and a $1$--spherical cycle $\tilde{f} \colon S^1 \to \minset(h)$, which is null-homotopic in $\minset(h)$, and such that $p \circ \tilde{f}$ is homotopic to $f$ in $P_2(Y(h)^{(0)})$.

Let $\g_0, \ldots, \g_m$ be vertices of the image of $f$ appearing in this order (i.e.\ $\g_i$ and $\g_{i+1}$ are adjacent and $\g_0=\g_m$). For every $i \in \{0, \ldots, m-1 \}$ pick a vertex $x_i \in \minset(h)$, such that $x_i \in \g_i$ and $x_i \neq x_j$ if $i \neq j$. Since $\g_i$ and $\g_{i+1}$ are adjacent in $Y(h)$, by definition of $Y(h)$ there exist vertices $y_i $ and $ z_{i+1}$ in $\minset(h)$ such that $y_i \in \g_i$ and $z_{i+1} \in \g_{i+1}$ and $y_i$, and $z_{i+1}$ are adjacent in $\minset(h)$ (this can always be done, even if adjacency of $\g_i$ and $\g_{i+1}$ in $Y(h)$ follows from the fact that they intersect in $\minset(h)$). Let $\alpha_i$ be the path defined as the concatenation of the segment $[x_i,y_i]$ of $\g_i$, the edge $\{y_i, z_{i+1}\}$ and the segment $[z_{i+1},x_{i+1}]$ of $\g_{i+1}$. Define $\tilde{f}$ as the concatenation of paths $\alpha_i$ for all $i \in \{0, \ldots, m-1 \}$ (see Figure~\ref{fig:twocycles}).

\begin{figure}[!h]
\label{fig:twocycles}
\centering
\begin{tikzpicture}[scale=0.5]

%titles
\node [above] at (1,3.5) {$\minset(h)$};

\node [above] at (14,3.5) {$P_1(Y(h)^{(0)})$};

%image points

\draw[fill] (14,2)  circle [radius=0.075]; 
\node [above] at (14,2) {$\g_0$};

\draw[fill] (15,1)  circle [radius=0.075]; 
\node [above right] at (14.75,1) {$\g_1$};

\draw[fill] (15,0)  circle [radius=0.075];

\draw[fill] (14,-1)  circle [radius=0.075]; 
\node [below] at (14,-1) {$\g_i$};

\draw[fill] (13,-0.5)  circle [radius=0.075];

\draw[fill] (12.5,0.5)  circle [radius=0.075];

\draw[fill] (13,1.5)  circle [radius=0.075]; 

%image line

\draw[] (14,2)--  (15,1) --(15,0)--(14,-1)  --(13,-0.5) --(12.5,0.5)  --(13,1.5)--(14,2);

% the map
%\node [above] at (10.5,0.25) {$\overset{\bar{f}} \longrightarrow$};

\node [above] at (10.5,0.25) {$\overset{p}{\longrightarrow}$};

%gammas

\draw (-6,2) to (8,2);
\node [above] at (7.5,2) {$\g_0$};

\draw (-3,1) to (1,1);
\draw [out=0, in=180] (1,1) to (4,0);
\draw (4,0) to (9,0);
\node [above] at (8.5,1) {$\g_1$};

\draw (-3,0) to (1,0);
\draw [out=0, in=180] (1,0) to (4,1);
\draw (4,1) to (9,1);

\draw (-4,-1) to (8,-1);

\draw (-9,-0.5) to (1,-0.5);

\draw (-9.5,0.5) to (0,0.5);

\draw (-9,1.5) to (1,1.5);

%edges
\draw [very thick] (6,2) -- (6,1) --(4,1)--(4,0) -- (7, 0) --(7,-1) -- (0,-1)-- (0,-0.5)--(-6, -0.5)--(-6, 0.5)--(-8, 0.5)-- (-8, 1.5)--(-5, 1.5)--(-5, 2)--(6,2);

%vertices

\draw[fill] (6,2)  circle [radius=0.075]; 
\node [above] at (6,2) {$y_0$};

\draw[fill] (6,1) circle [radius=0.075]; 
\node [below] at (6,1) {$z_1$};

\draw[fill] (5,1) circle [radius=0.075]; 

\draw[fill] (4,1) circle [radius=0.075]; 

\draw[fill] (4,0) circle [radius=0.075]; 
\draw[fill] (5,0) circle [radius=0.075]; 
\draw[fill] (6,0) circle [radius=0.075]; 
\draw[fill] (7, 0) circle [radius=0.075]; 
\draw[fill] (7,-1) circle [radius=0.075]; 
\draw[fill] (6,-1) circle [radius=0.075];
\node [below] at (7,-1) {$z_i$};

\draw[fill] (5,-1) circle [radius=0.075]; 
\draw[fill] (4,-1) circle [radius=0.075]; 
\draw[fill] (3,-1) circle [radius=0.075]; 

\draw[fill] (2,-1) circle [radius=0.075]; 
\draw[fill] (1,-1) circle [radius=0.075]; 
 \draw[fill]  (0,-1) circle [radius=0.075]; 
 \node [below] at (0,-1) {$y_i$};
  \draw[fill] (0,-0.5) circle [radius=0.075];

\draw[fill] (-1,-0.5) circle [radius=0.075];
\draw[fill] (-2,-0.5) circle [radius=0.075];
\draw[fill] (-3,-0.5) circle [radius=0.075];
\draw[fill] (-4,-0.5) circle [radius=0.075];
\draw[fill] (-5,-0.5) circle [radius=0.075];
  \draw[fill] (-6, -0.5) circle [radius=0.075]; 
  \draw[fill] (-6, 0.5) circle [radius=0.075];
  \draw[fill] (-7,0.5) circle [radius=0.075]; 
  \draw[fill] (-8, 0.5) circle [radius=0.075]; 
  \draw[fill]  (-8, 1.5) circle [radius=0.075]; 
  \draw[fill]  (-7, 1.5) circle [radius=0.075]; 
 \draw[fill]  (-6, 1.5) circle [radius=0.075]; 

   \draw[fill] (-5, 1.5) circle [radius=0.075]; 
   \node [below right] at (-5.5,1.5) {$y_{m-1}$};
  \draw[fill]  (-5, 2) circle [radius=0.075]; 
  \node [above] at (-5,2) {$z_0$};
  \draw[fill]  (-4, 2) circle [radius=0.075]; 
  \draw[fill]  (-3, 2) circle [radius=0.075];
    \draw[fill]  (-2, 2) circle [radius=0.075]; 
  \draw[fill]  (-1, 2) circle [radius=0.075]; 
  \draw[fill]  (0, 2) circle [radius=0.075]; 
  \draw[fill]  (1, 2) circle [radius=0.075]; 
  \draw[fill]  (2, 2) circle [radius=0.075]; 
\node [above] at (2,2) {$x_0$};

  \draw[fill]  (3, 2) circle [radius=0.075]; 
  \draw[fill]  (4, 2) circle [radius=0.075]; 
  \draw[fill]  (5, 2) circle [radius=0.075]; 
  \draw[fill]  (6,2) circle [radius=0.075]; 

%protruding shit

\draw[fill] (-2,-1) circle [radius=0.075]; 
\draw[fill] (-1,-1) circle [radius=0.075]; 

\draw[very thick] (-2, -1) to (0,-1);

  \node [below] at (-2,-1) {$x_i$};

\end{tikzpicture}
\caption{Cycles $\tilde{f}$ and $f$.}
\end{figure}
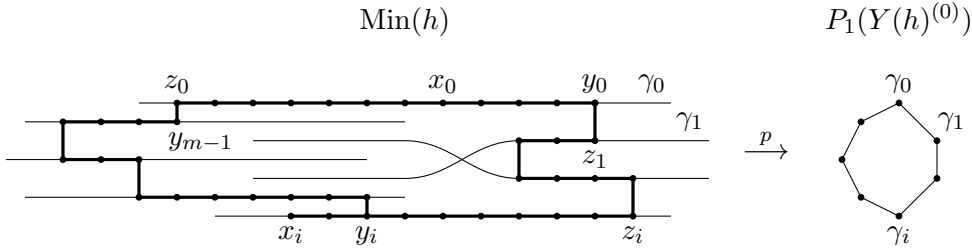

Define the map $p \colon \minset(h) \to P_1(Y(h)^{(0)})$ on vertices by $x \mapsto \g_x$, where $\g_x$ is an $h$--invariant geodesic passing through $x$ (such a geodesic in general is not unique, we choose one for each vertex). We claim that $p$ is a simplicial map. Indeed, if vertices $x$ and $y$ form an edge in $\minset(h)$, then this edge connects geodesics $\g_x$ and $\g_y$, hence by definition of the graph $Y(h)$ vertices $\g_x $ and $ \g_y$ are adjacent. Since both complexes are flag, the claim follows.

The complex $\minset(h)$ is systolic, hence in particular it is simply connected. Thus the cycle $\tilde{f}$ is null-homotopic and so is $p \circ \tilde{f}$, since $p$ is simplicial. It remains to prove that cycles $f$ and $p \circ \tilde{f}$ are homotopic in $P_2(Y(h)^{(0)})$. To see that, notice that if $v$ is a vertex in the image of $\tilde{f}$ and $v \in \gamma_i$ then its image $p(v)$ is at distance at most $1$ from $\gamma_i$.\medskip %Thus the cycle $f$ is null-homotopic in $P_2(Y(h)^{(0)})$.
%, so their associated spherical cycles are homotopic.

We are left with showing that the graph of axes $(Y(h), d_{Y(h)})$ is $\frc{1}$. We proceed as follows. By Lemma~\ref{minissystolic} the minimal displacement set $\minset(h)$ is systolic, hence $\frc{2}$ by Lemma~\ref{systolicisfrc}. By Theorem~\ref{coarsesplit} we have a quasi-isometry
\[c \colon (Y(h) \times \mathbb{Z}, d_h) \to (\minset(h), d_X).\]

Therefore by Proposition~\ref{frcgeometricproperty} we conclude that the product $(Y(h) \times \mathbb{Z}, d_h)$ is $\frc{2}$. Finally, since $\mathbb{Z}$ is unbounded and since any $1$--spherical cycle in $Y(h)^{(0)}$ is null-homotopic in $P_2(Y(h)^{(0)})$,  Lemma~\ref{sfrcproducts} implies that $(Y(h), d_{Y(h)})$ is $\frc{1}$.\end{proof}	

For the sake of completeness we include the following well-known result.

\begin{lem}\label{noquasi}Let $G$ be a finitely generated group which acts properly by isometries on a quasi-tree $(Q, d_Q)$. Then $G$ is virtually free. 
\end{lem}

\begin{proof} 
Fix a finite generating set $S$ of $G$ and let $d_S$ denote the associated word metric. Since the action of $G$ is proper, the orbit map \[(G, d_S) \to (Q, d_Q)\] is a coarse embedding. Composing it with a quasi-isometry \[(Q, d_Q) \to (T, d_T)\] gives a coarse embedding of $G$ into a tree $T$. Let $X$ denote the image of this embedding with the metric restricted from $T$. The subspace $X$ is quasi-connected, thus an appropriate thickening $N_R(X)$ is a connected subset of a tree, hence a tree. Clearly $N_R(X)$ is quasi-isometric to $X$. The composition \[G \to X \to N_R(X)\] is a coarse equivalence of geodesic metric spaces, hence a quasi-isometry; see e.g.\ \cite[Lemma 1.10]{Roe}. 
This implies that $G$ is $\delta$--hyperbolic, and its Gromov boundary is $0$--dimensional. It follows that $G$ is virtually free.
\end{proof}

\section{Classifying spaces for systolic groups}\label{sec:mainthm}

\subsection{Classifying space with virtually cyclic stabilisers}

In this section we gather results from Sections \ref{sec:preliminaries}, \ref{sec:minset} and \ref{sec:fillingradius} in order to prove Theorem~\ref{mainthm}.
\begin{tw}
\label{mainthm}
Let $G$ be a group acting properly %and simplicially 
on a uniformly locally finite systolic complex $X$ of dimension $d$. Then there exists a %$G$--CW-
model for $\underline{\underline{E}}G$ of dimension 
\[ \mathrm{dim}\underline{\underline{E}}G = \left\{ \begin{array}{cl}

      d+1 & \text{ if } d \leqslant 3, \\
      d & \text{ if } d \geqslant 4.  \\

      \end{array} \right.
 \] 

\end{tw} 
In the remainder of this section, let $G$ be as in the statement of the above theorem. The model for $\eeg$ we construct is given by the cellular $G$--pushout of Theorem~\ref{luwepushout}. Therefore we need to construct a model for $\eg$ and for every commensurability class of infinite virtually cyclic subgroups $[H]$, models for $\underline{E}N_G[H]$ and $E_{\mathcal{G}[H]}N_G[H]$. The first model was constructed by P.\ Przytycki \cite[Theorem 2.1]{PP}, and later ``refined'' by V.\ Chepoi and the first author.

\begin{tw}\cite[Theorem E]{OCh}
\label{modelpiotra} %Let $G$ act properly on a systolic complex $X$.
The systolic complex $X$ is a model for $\underline{E}G$.
\end{tw} 

In order to construct models for the commensurators $N_G[H]$ we need a little preparation. First we show that the group $G$ satisfies Condition (C) of Definition~\ref{condition}. Using this, in every finitely generated subgroup $K \subseteq N_G[H]$ that contains $H$ we find a suitable normal cyclic subgroup, and show that the quotient group acts properly on a quasi-tree. This together with Propositions~\ref{limitcoro} and~\ref{corodieter} allows us to construct the desired models.

\begin{lem}
The group $G$ satisfies condition $\mathrm{(C)}$ of Definition~\ref{condition}.\end{lem}
\begin{proof} The proof is a slight modification of the one given in \cite[proof of Theorem 1.1]{Lu09}). Take arbitrary $g, h \in G$ such that $|h|=\infty$, and assume there are $k,l \in \mathbb{Z}$ such that $g^{-1}h^kg=h^l$. We have to show that $|k|=|l|$. Since the action of $G$ on $X$ is proper, the element $h$ acts as a hyperbolic isometry and by Proposition~\ref{alltogetherlemma}.\ref{mostisom} there is an $h^n$--invariant geodesic $\gamma \in X$ for some $n \geqslant 1$. We get the claim by considering the following sequence of equalities for the translation length:
\[|k|\trl{h^n}=\trl{h^{nk}}=\trl{g^{-1}h^{nk}g}=\trl{h^{\rpm nl}}=|l|\trl{h^n}.\]

The first and the last of the equalities follow from the fact, that the translation length of an element can be measured on an invariant geodesic, the second one is an easy calculation and the third one is straightforward.\end{proof}

\begin{lem}\label{quotientacts} Let $K$ be a finitely generated subgroup of $G$, and $h \in K$ a hyperbolic isometry satisfying $\eqref{minaxes}$, such that $\langle h \rangle$ is normal in $K$. Then the proper action of $G$ on $X$ induces a proper 
action of $K/\langle h \rangle$ on the graph of axes  $Y(h)$.%$(Y(h), d_{Y(h)})$. 
\end{lem}

\begin{proof}
Since $\langle h\rangle$ is normal in $K$, the subcomplex $\minset(h)$ is invariant under $K$. Indeed, if $d_X(x, hx)= L(h)$, then for any $g \in K$ we have 
\[d_X(gx, hgx)=d_X(x, g^{-1}hgx)=d_X(x, h^{\rpm 1}x)=L(h).\] 

Since $h$ satisfies $\eqref{minaxes}$, the subcomplex $\minset(h)$ is spanned by the union of $h$--invariant geodesics. The group $K$ acts by simplicial isometries, hence it maps $h$--invariant geodesics to $h$--invariant geodesics. This gives an action of $K$ on the set of vertices of $Y(h)$. This action extends to the action on the graph $Y(h)$,
%$(Y(h), d_{Y(h)})$
 because the adjacency relation between vertices of $Y(h)$ is preserved under simplicial isometries. The subgroup $\langle h \rangle$ acts trivially, hence there is an induced action of the quotient group $K/\langle h \rangle$. 

It is left to show that the latter action is proper. For any vertex $\im{\g} \in Y(h)$ we show that its stabiliser 
$\mathrm{Stab}_{K/\langle h \rangle}(\im{\g})$ is finite. Denote by $\pi$ the quotient map $K\to K/\langle h \rangle$, and consider the preimage $\pi^{-1}(\mathrm{Stab}_{K/\langle h \rangle}(\im{\g}))$. Elements of $\pi^{-1}(\mathrm{Stab}_{K/\langle h \rangle}(\im{\g}))$ are precisely these isometries, which map geodesic $\im{\g}$ to itself. Thus we can define a map $p: \pi^{-1}(\mathrm{Stab}_{K/\langle h \rangle}(\im{\g})) \to D_{\infty}$, where $D_{\infty}$ is the infinite dihedral group, interpreted as the group of simplicial isometries of the geodesic line $\im{\g}$. We claim that the kernel $\mathrm{ker}(p)$ is finite. Indeed, the kernel consists of elements which act trivially on the whole geodesic $\im{\g}$, hence it is contained in the stabiliser $\mathrm{Stab}_{G}(x)$ of any vertex $x\in \im{\g}$. The group $\mathrm{Stab}_{G}(x)$ is finite, since the action of $G$ on $X$ is proper.
Therefore the group $\pi^{-1}(\mathrm{Stab}_{K/\langle h \rangle}(\im{\g}))$ is virtually cyclic, as it maps into a virtually cyclic group $D_{\infty}$ with finite kernel. The infinite cyclic group $\langle h \rangle$ is contained in $\pi^{-1}(\mathrm{Stab}_{K/\langle h \rangle}(\im{\g}))$, hence the quotient group $\pi^{-1}(\mathrm{Stab}_{K/\langle h \rangle}(\im{\g}))/\langle h \rangle = \mathrm{Stab}_{K/\langle h \rangle}(\im{\g})$ is finite.
\end{proof}

\begin{lem}\label{shortexactsequence} Let $K$ be a finitely generated subgroup of $N_G[H]$ that contains $H$. Then there is a short exact sequence 
\[0 \longrightarrow \langle h \rangle \longrightarrow K \longrightarrow K/ \langle h \rangle \longrightarrow 0,\]
such that $h \in H$ is of infinite order and %$[\langle h \rangle] =[H]$ 
the group $K/ \langle h \rangle$ is virtually free.
\end{lem}

\begin{proof}

Choose an element of infinite order $\tilde{h} \in H$ satisfying the following two conditions:

\begin{enumerate}
\item[(i)] the set $\minset(\tilde{h})$ is the union of axes (see \eqref{minaxes}),
\item[(ii)] the translation length $L(\tilde{h}) > 3$.
\end{enumerate}
Both (i) and (ii) can be ensured by rising $\tilde{h}$ to a sufficiently large power. Indeed, by Proposition~\ref{alltogetherlemma}.\ref{mostisom} there exists $n\geqslant 1$ such that $\tilde{h}^n$ satisfies condition (i). If $L(\tilde{h}^n) \leqslant 3$ then replace it with $\tilde{h}^{4n}$. The element $\tilde{h}^{4n}$ satisfies both conditions (see Proposition~\ref{alltogetherlemma}.\ref{axialminsetpower}). %and clearly we have $[\langle \tilde{h}^{4n} \rangle ]=[\langle \tilde{h} \rangle ]$. 
Notice that if an element satisfies conditions (i) and (ii) then, by Proposition~\ref{alltogetherlemma}.\ref{axialminsetpower} so does any of its powers.
Since $G$ satisfies Condition (C), by  Lemma~\ref{normalsubgroup} there exists an integer $k\geqslant 1$ such that $\langle \tilde{h}^k \rangle$ is normal in $K$. 

Put $h=\tilde{h}^k$. By Lemma~\ref{quotientacts} the group $K/\langle h \rangle$ acts properly by isometries on the graph of axes $(Y(h), d_{Y(h)})$, which is a quasi-tree by Corollary~\ref{s1frcquasitreecoro}. Finally, Lemma~\ref{noquasi} implies that the group $K/\langle h \rangle$ is virtually free.
\end{proof}

\begin{lem}\label{commmodels}For every $[H] \in [\mathcal{VCY} \setminus \mathcal{FIN}]$ there exist

\begin{enumerate}[label=(\roman*)]
\item \label{jedencommmodels} a $2$--dimensional model for $E_{\mathcal{G}[H]} N_G[H]$,

\item \label{dwacommmodels} a $3$--dimensional model for $\underline{E} N_G[H]$.
\end{enumerate}
\end{lem}

\begin{proof}
By Proposition~\ref{limitcoro} it is enough to construct for every finitely generated subgroup $K \subseteq N_G[H]$, a $1$--dimensional model for $E_{\mathcal{G}[H] \cap K} K$ and a $2$--dimensional model for $\underline{E} K$. Notice that every finitely generated subgroup $K'$ of $G$ is contained in the finitely generated subgroup $K$ that contains $H$ (take $K= \langle K',H \rangle $). Therefore it is enough to consider only finitely generated subgroups of $G$ that contain $H$.

 By Lemma~\ref{shortexactsequence} for any such $K$ there is a short exact sequence 
\[0 \longrightarrow\langle h \rangle \longrightarrow K \overset{\pi} \longrightarrow K/\langle h \rangle \longrightarrow 0,\]
where $K/ \langle h \rangle$ is virtually free. The key observation is that the group $K/ \langle h \rangle$ acts properly on a simplicial tree \cite[Theorem 1]{Karrass} and therefore a tree is a  $1$--dimensional model for $\underline{E}K/\langle h \rangle$.
The claim follows then from Proposition~\ref{corodieter} in the following way. First notice that for every subgroup $H \in \mathcal{G}[H]$ the image $\pi(H)$ is finite. 

The preimage under $\pi$ of any finite subgroup $F \in K/ \langle h \rangle$ is a virtually cyclic group containing $\langle h \rangle$. In this case the intersection $\pi^{-1}(F) \cap \langle h \rangle $ is infinite, hence by definition of $\mathcal{G}[H]$ the group $\pi^{-1}(F)$ belongs to the family $\pi^{-1}(F) \cap \mathcal{G}[H]$. Thus the one point space is a $0$--dimensional model for $E_{\mathcal{G}[H]  \cap\pi^{-1}(F)  } \pi^{-1}(F)$. This proves \ref{jedencommmodels}.

To prove \ref{dwacommmodels} notice that since $\pi^{-1}(F)$ is virtually cyclic, it acts on the real line with finite stabilisers \cite[Proposition 4]{LePi}. Therefore a line is a $1$--dimensional model for $\underline{E} \pi^{-1}(F)$. 
\end{proof}

\begin{proof}[Proof of Theorem~\ref{mainthm}] By Corollary~\ref{commenough} choosing a model for $\underline{E}G$,  and for every $[H] \in [\mathcal{VCY} \setminus \mathcal{FIN}]$  models for $\underline{E}N_G[H]$ and $E_{\mathcal{G}[H]} N_G[H]$, gives a model for $\eeg$ that satisfies the following inequality
\[\mathrm{dim}\eeg \leqslant \mathrm{max} \{ \mathrm{dim}\underline{E}G, \underset{[H]}{\mathrm{sup}} \{\mathrm{dim}\underline{E}N_G[H] \}+1  , \underset{[H]}{\mathrm{sup}} \{\mathrm{dim} E_{\mathcal{G}[H]} N_G[H]\}  \}.   \]
If $d \geqslant 4$, then by Theorem~\ref{modelpiotra} and Lemma~\ref{commmodels} we have 
\[\mathrm{dim}\eeg \leqslant \mathrm{max} \{ d, 4 , 2\}=d.   \]
If $d \leqslant 3$, we can take the model for $\underline{E}G$ as a model for $\underline{E} N_G[H]$ instead of the one provided by Lemma~\ref{commmodels}, and obtain 
\[\mathrm{dim}\eeg \leqslant \mathrm{max} \{ d, d+1 , 2\}=d+1.  \qedhere \]
\end{proof}

\subsection{Centralisers of cyclic subgroups} 

As a corollary of our results we give the description of centralisers of infinite order elements in \emph{systolic groups}, i.e.\ groups acting properly and cocompactly %simplicially 
on systolic complexes, therefore confirming a conjecture of D.\ Wise. %(\cite{Wise}).

\begin{prop}\label{prop:wiseconj} Let $G$ be a group that acts properly on a systolic complex $X$ and let $h \in G$ be of infinite order.
Suppose that $K$ is a finitely generated subgroup of the centraliser $C_G(h)$ and $\langle h \rangle \subset K$. Then $K$ is commensurable with $F_n \times \mathbb{Z}$ where $F_n$ denotes the free group on $n$ generators for some $n \geqslant 0$.
\end{prop}

\begin{proof}
The group $\cen{h}$ is contained in the commensurator $N_G[\langle h \rangle]$, hence so is $K$. Thus by Lemma~\ref{shortexactsequence} there is a short exact sequence
\[0 \longrightarrow\langle h^m \rangle \longrightarrow K \overset{p} \longrightarrow VF_n \longrightarrow 0,\]
where $VF_n$ is a virtually free group and $m$ is some positive integer. Taking the free subgroup $F_n \subset VF_n$ gives rise to the following
\[0 \longrightarrow \mathbb Z \longrightarrow p^{-1}(F_n) \overset{p} \longrightarrow F_n \longrightarrow 0.\] 
Since $F_n$ is free, the above sequence splits. Therefore, as a central extension, $p^{-1}(F_n)$ is of the form $\mathbb Z \times F_n$.
This finishes the proof, as 
$[ K: p^{-1}(F_n)] \leqslant [VF_n : F_n] < \infty$.
\end{proof}

\begin{cor}\cite[Conjecture 11.6]{Wise}\label{wisecoro} Let $G$ be a systolic group. Then for any element $h \in G$ of infinite order, the centraliser $C_G(h)$ is commensurable with $F_n \times \mathbb{Z}$ for some $n \geqslant 0$.
\end{cor}

\begin{proof} The group $G$ is biautomatic by \cite[Theorem E]{JS2}, and it follows that
the centraliser $\cen{h}$ is biautomatic as well \cite[Proposition 4.3]{GeSh}. In particular $\cen{h}$ is finitely generated. Thus the claim follows from Proposition~\ref{prop:wiseconj}.
\end{proof}

\subsection{Virtually abelian stabilisers}\label{s:vab2}

In this section we study the family of all virtually abelian subgroups of a group $G$. We show that if $G$ is systolic, then there exists a finite dimensional model for the classifying space for this family. This is due to a very special structure of abelian subgroups of systolic groups, which is in turn a consequence of the systolic Flat Torus Theorem. Our construction also carries through for certain $\mathrm{CAT}(0)$ groups.  \medskip

Given a group $G$, let $\mathcal{VAB}$ denote the family of all virtually abelian subgroups of $G$ and let  $\mathcal{VAB}_{fg}$ denote the family of all finitely generated virtually abelian subgroups of $G$. Every subgroup in the family $\mathcal{VAB}_{fg}$ %subgroup $H \in \mathcal{VAB}$ 
 contains a finite-index free abelian subgroup of rank $n \geqslant 0$, therefore if we denote by $\mathcal{VAB}_n$ the family of all virtually abelian subgroups of rank at most $n$, we obtain the following filtration of the family $\mathcal{VAB}_{fg}$:
\[%\mathcal{VAB} = 
\mathcal{VAB}_0 %= \mathcal{FIN} 
\subset \mathcal{VAB}_1 %=\mathcal{VCY} 
\subset \mathcal{VAB}_2 \subset \ldots \smallskip\]
Notice that  $\mathcal{VAB}_0 = \mathcal{FIN}$ and  $\mathcal{VAB}_1 = \mathcal{VCY}$. Moreover, if $G$ is a systolic group then by Theorem~\ref{ftt}.(\ref{ftt:rank2}) it does not contain free abelian groups of rank higher than $2$, and therefore the above filtration reduces to
\[\mathcal{FIN} \subset \mathcal{VCY} \subset \mathcal{VAB}_2=\mathcal{VAB}_{fg}.\]

Moreover, in Proposition~\ref{prop:sysfingen} we show that every virtually abelian subgroup of a systolic group is in fact finitely generated, and therefore for systolic groups we have $\mathcal{VAB}_{fg}=\mathcal{VAB}$.
The following is the main theorem of this section.

\begin{tw}\label{twvab2}Let $G$ be a group acting properly and cocompactly on a $d$--di\-men\-sion\-al systolic complex. Then there exists a model for $E_{\mathcal{VAB}}G$ of dimension %$\mathrm{dim}E_{\mathcal{VAB}}G= 
$\mathrm{max} \{4, d\}$.
\end{tw}

The construction which we use is 
a pushout construction of \cite{LuWe} (cf.\ Section~\ref{sec:classifying}) applied to the inclusion of families $\mathcal{VCY} \subset \mathcal{VAB}$. More precisely we want to apply \cite[Corollary 2.8]{LuWe} which requires the collection of subgroups $\mathcal{VAB} \setminus \mathcal{VCY}$ of  $G$ to satisfy the following two conditions: 
\begin{description}
\item[(NM1)] 
any $H \in \mathcal{VAB} \setminus \mathcal{VCY}$ is contained in a unique maximal $M \in \mathcal{VAB} \setminus \mathcal{VCY}$,
\item[(NM2)] for any maximal subgroup $M$ of $\mathcal{VAB} \setminus \mathcal{VCY}$ we have $N_G(M)=M$.
\end{description}

These conditions correspond to conditions $M_{\mathcal{VCY} \subset \mathcal{VAB}}$ and $NM_{\mathcal{VCY} \subset \mathcal{VAB}}$ of \cite[Notation 2.7]{LuWe}. We will keep our notation for the sake of clarity.

\begin{lem}\label{lemmanm}
Let $G$ be a systolic group. Then $G$ satisfies conditions $\mathrm{(NM1)}$ and $\mathrm{(NM2)}$.
\end{lem}
Assuming the lemma, we proceed with the construction of the desired model.

\begin{proof}[Proof of Theorem~\ref{twvab2}]
Let $\mathcal{M}$ denote the complete set of representatives of conjugacy classes in $G$ of subgroups which are maximal in $\mathcal{VAB} \setminus \mathcal{VCY}$. Since $G$ satisfies (NM1) and (NM2), it follows from \cite[Corollary 2.8]{LuWe} that a model for $E_{\mathcal{VAB}}G$ is given by
the cellular $G$--pushout
\[
\begin{CD}
\coprod_{M  \in \mathcal{M}}  G \times_{M}  \underline{\underline{E}}M    @>i>>  \underline{\underline{E}}G \\
@VV{\coprod_{M  \in \mathcal{M}} p_M}V         @VVV\\
\coprod_{M  \in \mathcal{M}}  G/M     @>{\phantom{\text{abel}}}>>   E_{\mathcal{VAB}}G,\smallskip
\end{CD}
\]
where $i$ is an inclusion of CW--complexes and $p_M$ is the canonical projection \[G \times_{M}  \underline{\underline{E}}M \to  G \times_{M} \ast \cong G/M.\]

By Theorem~\ref{mainthm} there exists a $d$--dimensional model for $\underline{\underline{E}}G$ as long as $d \geqslant4$. It follows from \cite[Theorem 5.13.(iii)]{LuWe} that there exists a $3$--dimensional model for $ \underline{\underline{E}}M  
$ (since $M$ contains a finite-index subgroup isomorphic to $\mathbb{Z}^2$) and it is in fact a model of the lowest possible dimension.  The existence of the map $i$ follows from the universal property of the classifying space $\underline{\underline{E}}G$. To ensure that $i$ is injective, we replace it with an inclusion into the mapping cylinder (cf.\ Corollary~\ref{commenough}). 
Finally, we have that $G/M$ has dimension $0$. Therefore applying the above pushout to these models gives us a model for $E_{\mathcal{VAB}}G$ of dimension $\mathrm{max}\{0,4,d\}$.
\end{proof}

It remains to prove Lemma~\ref{lemmanm}. The main tool that we use in the proof is the systolic Flat Torus Theorem of T.\ Elsner. Before stating the theorem we need to recall some terminology. For details we refer the reader to \cite{E1}.

\begin{de}\label{def:flat}Let $\mathbb{E}_{\Delta}^2$ denote the equilaterally triangulated Euclidean plane. A \emph{flat} in a systolic complex $X$ is a simplicial map $F \colon \mathbb{E}_{\Delta}^2 \to X$ which is an isometric embedding. We will identify $F$ with its image and treat it as a subcomplex of $X$.
\end{de}

We say that two flats are \emph{equivalent} if they are at finite Hausdorff distance. This gives an equivalence relation on the set of all flats which we call a \emph{flat equivalence}.
Let $Th(F)$ denote the subcomplex of $X$ spanned by all the flats that are equivalent to $F$. We call $Th(F)$ the \emph{thickening} of $F$. Any two equivalent flats are in fact at Hausdorff distance $1$ \cite[Theorem 5.4]{E1}. Therefore for any $F'$ that is equivalent to $F$, the inclusion $F' \hookrightarrow Th(F)$ is a quasi-isometry.

If $H \subset G$ is a free abelian subgroup of a systolic group $G$ we define minimal displacement set of $H$ as follows \[\minset(H)= \bigcap_{h\in H \setminus \{e\}} \minset(h).\]
		
\begin{tw}[Flat Torus Theorem]\cite[Theorem 6.1]{E1}\label{ftt}
Let $G$ be a systolic group and let $H \subset G$ be a free abelian subgroup of rank at least $2$. Then:

\begin{enumerate}
\item \label{ftt:rank2} the group $H$ is isomorphic to $\mathbb{Z}^2$,
\item \label{ftt:existence} there exists an $H$--invariant flat $F$, unique up to flat equivalence,
\item \label{ftt:minset} we have $\minset(H)=Th(F)$ for an $H$--invariant flat $F$.
\end{enumerate}
\end{tw}

\begin{proof}[Proof of Lemma~\ref{lemmanm}]
(NM1) First we show that any rank $2$ virtually abelian subgroup $H \subset G$ is contained in a maximal one. This is equivalent to the statement that any ascending chain of rank $2$ virtually abelian subgroups $H_1 \subset H_2 \subset \ldots$ stabilises, i.e.\ we have $H_i=H_{i+1}$ for $i$ sufficiently large.

Suppose $H_1 \subset H_2 \subset \ldots$ is such a chain and let $A$ be a finite-index subgroup of $H_1$ isomorphic to $\mathbb{Z}^2$. Since for every $i$ the group $H_i$ contains a finite-index subgroup isomorphic to $ \mathbb{Z}^2$, it follows that the index of $A$ in $H_i$ is finite. We will show that this index is bounded from above by a constant which is independent of $i$. 

By \cite[Corollary 6.2]{E1} the group $H_i$ preserves the thickening of an $A_i$--invariant flat $F_i$ where $A_i$ is a certain finite-index subgroup of $H_i$.  Since $A$ and $A_i$ are finite-index subgroups of $H_i$, so is their intersection $A \cap A_i$. By Theorem~\ref{ftt}.(\ref{ftt:existence}) there exists an $A$--invariant flat $F$. Note that $A \cap A_i \cong \mathbb{Z}^2$ and both $F$ and $F_i$ are $A \cap A_i$--invariant. Therefore, again by Theorem~\ref{ftt}.(\ref{ftt:existence}), we have $Th(F_i)=Th(F)$. Therefore any $H_i$ preserves $Th(F)$.

Now, since $G$ acts properly and cocompactly, for any integer $R >0$ there exists an integer $N_R$ such that for every vertex $v \in X$ the cardinality of the set $\{g \in G \mid d(gv,v) \leqslant R \}$ is at most $N_R$. Since $A$ acts cocompactly on $F$ and since $Th(F)$ is quasi-isometric to $F$, there is an integer $R >0$ such that for any vertex $w \in F$, the orbit of a combinatorial ball $B_R(w,X)$ under $A$ covers the thickening $Th(F)$. Fix a vertex $v \in F$. For any $h\in H_i$ there exists $a \in A$ such that $d(v, ahv) \leqslant R$. It follows that the index of $A$ in $H_i$ is bounded by $N_R$. \medskip

Now we prove the uniqueness. Assume that $H_1$ and $H_2$ are maximal subgroups in $\mathcal{VAB} \setminus \mathcal{VCY}$ that contain $H$. By \cite[Corollary 6.2.(2)]{E1} there are flats $F_1$ and $F_2$ such that $H_1 =\mathrm{Stab}_G(Th(F_1))$ and $H_2 =\mathrm{Stab}_G(Th(F_2))$.
By \cite[Corollary 6.2.(1)]{E1} there exists a flat $F$, unique up to flat equivalence, such that $H$ preserves $Th(F)$. Since $H$ is contained in both $H_1$ and $H_2$, the thickenings $Th(F_1)$ and $Th(F_2)$ are both $H$--invariant. Hence we have $Th(F_1) = Th(F_2) =Th(F)$ and therefore $H_1=H_2$. \medskip

(NM2) Let $H \in \mathcal{VAB} \setminus \mathcal{VCY}$ be a maximal subgroup and let $A'$ be a finite-index subgroup of $H$ that is isomorphic to $\mathbb{Z}^2$. Define the subgroup $A \subset H$ as the intersection of all subgroups of $H$ of index $[H:A']$. Since $H$ is finitely generated, there is finitely many of subgroups of this kind. Therefore $A \subset H$ is a rank $2$ free abelian subgroup of finite index, and by construction $A$ is a characteristic subgroup of $H$. It~follows that the group $N_G(H)$ normalises $A$, and hence it preserves the subcomplex $\minset(A)$. 

The action of $N_G(H)$ on $\minset(A)$ is proper and therefore the induced action of $N_G(H)/A$ on $\minset(A)/A$ is proper. By Theorem~\ref{ftt}.(\ref{ftt:minset}) we have $\minset(A) = Th(F)$ where $F$ is an $A$--invariant flat. This implies that the action of $A$ on $\minset(A)$ is cocompact. Since the quotient $N_G(H)/A$ acts properly on a compact space $\minset(A)/A$, it follows that $N_G(H)/A$ is a finite group. Therefore $N_G(H)$ is a rank $2$ virtually abelian group and hence we have $N_G(H)=H$ by the maximality of $H$.
\end{proof}

The methods used above apply also to a certain class of $\mathrm{CAT}(0)$ groups, namely the groups acting geometrically on $\mathrm{CAT}(0)$ spaces that do not contain flats of dimension greater than $2$. For details about $\mathrm{CAT}(0)$ spaces and groups we refer the reader to \cite{BrHa}. 

\begin{cor}\label{cat0coro}Let $G$ be a group acting properly and %by semisimple isometries on a complete proper 
cocompactly by isometries on a complete
$\mathrm{CAT}(0)$ space $X$ of topological dimension $d>0$. Furthermore, assume that for $n >2$ there is no isometric embedding $\mathbb{E}^n \to X$ where $\mathbb{E}^n$ is the Euclidean space.
Then there exists a model for $E_{\mathcal{VAB}}G$ of dimension $\mathrm{max} \{4, d+1\}$.
\end{cor}

Among the $\mathrm{CAT}(0)$ spaces satisfying the assumptions of the above corollary there are $\mathrm{CAT}(0)$ spaces of dimension $2$, e.g.\ $\mathrm{CAT}(0)$ square complexes,  and rank--$2$ symmetric spaces. In particular, the corollary applies to lattices in rank--$2$ symmetric spaces, thus answering a special case of a question by J.-F.\ Lafont \cite[Problem 46.7]{Gui}. On the other hand, our approach fails if $X$ contains flats of dimension bigger than $2$. The construction of models for $E_{\mathcal{VAB}}G$ in this case would require techniques significantly different from ours.

\begin{proof}[Proof of Corollary~\ref{cat0coro}.] By \cite{Lu09} there exists a model for $\eeg$ of dimension at most $d+1$, as long as $d\geqslant 3$. Since $X$ does not contain isometrically embedded $\mathbb{E}^n$ for $n>2$ it follows from the Flat Torus Theorem \cite[Theorem II.7.1]{BrHa} that $G$ does not contain free abelian subgroups of rank bigger than $2$. This together with the fact that every virtually abelian subgroup of $G$ is finitely generated \cite[Corollary II.7.6]{BrHa} implies that the family $\mathcal{VAB}$ reduces to $\mathcal{VAB}_2$. It remains to show that conditions (NM1) and (NM2) are satisfied.
The proof of this is analogous to the proof of Lemma~\ref{lemmanm}. The ``existence'' part of (NM1) follows from \cite[Theorem II.7.5]{BrHa}. Both the ``uniqueness'' part of (NM1) and condition (NM2) follow from \cite[Corollary II.7.2]{BrHa}. 
\end{proof}

We finish this section with the aforementioned proposition.

\begin{prop}\label{prop:sysfingen}Let $G$ be a group acting properly and cocompactly on a finite dimensional systolic complex $X$. Then every virtually abelian subgroup of $G$ is finitely generated.
\end{prop}

\begin{proof}
It is enough to prove that every abelian subgroup $A$ of $G$ is finitely generated. 
Since $G$ acts properly and cocompactly on $X$, there is a uniform bound on the order of finite subgroups of $G$. Therefore the torsion subgroup of $A$ must be finitely generated. Now let $A' \subset A$ be the torsion-free part and let $\mathrm{rk}(A')$ denote its rank. By Theorem~\ref{ftt}.(\ref{ftt:rank2}) we have $\mathrm{rk}(A') \leqslant 2$. 

If $\mathrm{rk}(A')=1$ then we claim that $A' \cong \mathbb{Z}$. To show this, by the classification of torsion-free abelian groups of rank $1$, it is enough to show that for any $a \in A'$ there are only finitely many positive integers $n$, such that there exists $b \in A'$ with $a= nb$. Suppose we have $a= nb$ for some $b$ and $n$. Since both $a$ and $b$ are hyperbolic isometries of $X$, we can compare their translation lengths. If $b$ has an axis, then it is straightforward to see that $\trl{nb} =n \cdot \trl{b}$. If $b$ has no axis, then by \cite[Theorem 1.3]{E2} it has a ``thick axis'' of thickness $k \leqslant \mathrm{dim}X$, and using \cite[Fact 3.7]{E2} one easily checks that  $\trl{nb} \geqslant 
\left \lfloor{\frac{n}{k}}\right \rfloor \cdot \trl{b}$. 
Therefore, in both cases the following holds: \begin{equation}\label{classificationequation}\trl{a}=\trl{nb} \geqslant 
\left \lfloor{\frac{n}{k}}\right \rfloor \cdot \trl{b} \geqslant  \left \lfloor{\frac{n}{k}}\right \rfloor. \smallskip \end{equation} 
Since $k \leqslant \mathrm{dim}X$, for a fixed element $a$ there are only finitely many positive integers $n$ satisfying \eqref{classificationequation}. Therefore we get that $A' \cong \mathbb{Z}$.

If $\mathrm{rk}(A')=2$ then proceeding as in the proof of Lemma~\ref{lemmanm}.(NM2) we obtain that $A'$ acts properly and cocompactly on a thickening of an $A''$--invariant flat where $A'' \subset A'$ is a subgroup isomorphic to $\mathbb{Z}^2$. Therefore $A'/A''$ is finite and thus $A'$ is finitely generated.
\end{proof}

\begin{rem}All results in this section hold under the following weakened assumptions. Instead of a cocompact action we assume that $X$ is uniformly locally finite (which is automatically true if the action is cocompact) and that there is a uniform bound on the order of finite subgroups of $G$. 

In the $\mathrm{CAT(0)}$ case, instead of a cocompact action we assume that $X$ is  proper,
 the action is via semisimple isometries and the set of translation lengths of hyperbolic elements is discrete at $0$.
\end{rem}

\section{Graphical small cancellation complexes}\label{sec:graphical}

In this section we begin the study of graphical small cancellation complexes. 
Our goal is to show that for any group $G$ acting properly on graphical small cancellation complex,
there is a (canonical) systolic complex on which $G$ acts properly, and use the latter to construct low-dimensional models for various classifying spaces for $G$. This requires substantial preparations, including  notation and terminology.

We begin with introducing combinatorial and graphical $2$--complexes. Then we state and prove a version of the so-called Lyndon-van Kampen Lemma, and use it to establish certain combinatorial properties of graphical small cancellation complexes. In Section~\ref{sec:dualisisystolic} we define the dual complex of a graphical small cancellation complex and show that these two are $G$--homotopy equivalent, where $G$ is any group that acts on a graphical small cancellation complex. Finally we give the construction of classifying spaces for families $\mathcal{FIN}, \mathcal{VCY}$ and $ \mathcal{VAB}$ for graphical small cancellation groups.

\subsection{Combinatorial $2$--complexes}\label{sec:comb2}

The purpose of this section is to give the basic definitions and to establish terminology regarding combinatorial $2$--complexes. In our exposition we mainly follow \cite{McWi}.\medskip

A map $X \to Y$ of CW--complexes is \emph{combinatorial} if its restriction to every open cell of $X$ is a homeomorphism onto an open cell of $Y$. A CW--complex is \emph{combinatorial} if the attaching map of every  $n$--cell is combinatorial for a suitable subdivision of the sphere $S^{n-1}$.
An \emph{immersion} is a combinatorial map that is locally injective. \medskip

Unless otherwise stated, all combinatorial CW--complexes that we consider are $2$--dimensional and all the attaching maps are immersions. We will refer to them simply as ``$2$--complexes''. Consequently all the maps between $2$--complexes  are assumed to be combinatorial.\medskip

Notice that according to the above definition, a \emph{graph} may contain loops and multiple edges, as opposed to graphs considered in Sections \ref{sec:preliminaries}--\ref{sec:mainthm}.

\begin{ex}[Presentation complex]
Let $ \langle S | R \rangle$ be a group presentation. The \emph{presentation complex} is a $2$--complex that has a single $0$--cell, a directed labeled $1$--cell for each generator $s \in S$, and a $2$--cell attached along the closed combinatorial path corresponding to each relator $r \in R$.
\end{ex}

 A \emph{polygon} is a $2$--disc with the cell structure that consists of $n$ vertices, %$0$--cells
  $n$ edges %$1$--cells 
 and a single $2$--cell. For any $2$--cell $C$ of $2$--complex $X$ there exists a %combinatorial 
 map $R \to X$, where $R$ is a polygon and the attaching map for $C$ factors as $S^{1} \to \partial R \to X$. In the remainder of this section by a \emph{cell} we will mean a %combinatorial 
 map $R \to X$ where $R$ is a polygon. An \emph{open cell} is the image in $ X$ of the single $2$--cell of $R$. \medskip

A \emph{path} in $X$ is a combinatorial
map $P \to X$ where $P$ is either a subdivision of the interval or a single vertex. If $P$ is a vertex, we call path $P\to X$ a \emph{trivial} path. If the target space is clear from the context, we will refer to the path $P \to X$ as ``the path $P$''. The \emph{interior} of the path is the path minus its endpoints. Let $P^{-1}$ denote the path $P$ traversed in the opposite direction.
Given paths $P_1\to X$ and $P_2 \to X$ such that the terminal point of $P_1$ is equal to the initial point of $P_2$, their \emph{concatenation} is an obvious path $P_1P_2 \to X$ whose domain is the union of $P_1$ and $P_2$ along these points. 
A \emph{cycle} is a %combinatorial 
map $C \to X$, where $C$ is a subdivision of the circle $S^1$. The cycle $C\to X$ is \emph{non-trivial} if it does not factor through a map to a tree. Therefore a homotopically non-trivial cycle is non-trivial, but the converse is not necessarily true.
A path or cycle is \emph{simple} if it is injective on vertices. Notice that a simple cycle (of length at least $3$) is non-trivial.
A \emph{length} of a path $P$ or a cycle $C$ denoted by $|P|$ or $|C|$ respectively is the number of $1$--cells in the domain.
A subpath $Q \to X$ of a path $P \to X$ (or a cycle) is a path that factors as $Q \to P \to X$ such that $Q \to P$ is an injective map. Notice that the length of a subpath does not exceed the length of the path.\medskip

A \emph{disc diagram} is a contractible finite $2$--complex $D$ with a specified embedding into the plane. We call $D$ \emph{nonsingular} if it is homeomorphic to the $2$--disc. Otherwise $D$ is called \emph{singular}. The \emph{area} of $D$ is the number of $2$--cells.
The boundary cycle $\partial D$ is the attaching map of the $2$--cell that contains the point $\{\infty\}$, when we regard $S^2= \mathbb{R}^2 \cup \{\infty\}$.
A \emph{boundary path} is any path $P\to D$ that factors as $P\to \partial D \to D$. An \emph{interior path} is a path such that none of its vertices, except for possibly endpoints, lie on the boundary of $D$.

 If $X$ is a $2$--complex a \emph{disc diagram in} $X$ is a  map $D \to X$.\medskip

The following definition is crucial in small cancellation theory.

\begin{de}\label{def:discpiece}A \emph{piece} in a disc diagram $D$ is a path $P \to D$ for which there exist two different lifts to $2$--cells of $D$, i.e. there are $2$--cells $R_i \to D$ and $R_j \to D$ such that $P \to D$ factors both as $P \to R_i \to D$ and $P \to R_j \to D$, but there does not exist a map $R_j \to R_i$ making the diagram
\[
\begin{tikzcd}
 P \arrow{r} \arrow{d} & R_i \arrow{d} \\
         R_j \arrow{ru} \arrow{r} &  D 
\end{tikzcd}
\]
commute. (Note that it might still be that $R_i=R_j$.)
\end{de}

Now we turn to graphical complexes.

\begin{de}\label{def:thickcomplex}
Let ${\bf \Gamma} \to \Theta$ be an immersion of graphs and assume that $\Theta$ is connected. For convenience we will write ${\bf \Gamma}$ as the union of its connected components \[{\bf \Gamma}=\bigsqcup_{i \in I}\Gamma_i,\] and refer to the connected graphs $\Gamma_i$ as \emph{relators}.

 A \emph{thickened graphical complex} $X$ is a $2$--complex with $1$--skeleton $\Theta$ and a $2$--cell attached along every immersed cycle in ${\bf \Gamma}$, i.e.\ if a cycle $C \to {\bf \Gamma}$ is immersed, then in $X$ there is a $2$--cell  attached along the composition $C \to {\bf \Gamma} \to \Theta$.
\end{de}

The term ``thickened'' comes from the fact, that for any connected component $\Gamma_i$, we have a ``thick cell'' $\thi{\Gamma_i}$ which is formed by gluing $2$--cells along all immersed cycles in $\Gamma_i$. As long as $\Gamma_i$ is not a tree, there is infinitely many $2$--cells in $\thi{\Gamma_i}$.
This definition may seem odd, however, it allows us to avoid certain technical complications in the proof of a version of the Lyndon-van Kampen Lemma in Section~\ref{s:lvk}.

 \begin{de}\label{def:graphpiece}Let $X$ be a thickened graphical complex. A \emph{piece} in $X$ is a path $P \to X$ for which there exist two different lifts to ${\bf \Gamma}$, i.e.\ there are two relators $\Gamma_i$ and $\Gamma_j$ such that the path $P \to X$ factors as $P \to \Gamma_i \to X$ and $P \to \Gamma_j \to X$,
but there does not exists a map $\thi{\Gamma_i} \to \thi{\Gamma_j}$ such that the diagram 
 \[
\begin{tikzcd}
 P \arrow{r} \arrow{d} & \thi{\Gamma_i} \arrow{d} \\
         \thi{\Gamma_j} \arrow{ru} \arrow{r} &  X 
\end{tikzcd}
\]
commutes.
\end{de}

\subsection{The Lyndon-van Kampen Lemma}\label{s:lvk}

\begin{de}Let $X$ be a thickened graphical complex. A disc diagram $D \to X$ is \emph{reduced} if for every piece  $P \to D$ the composition  $P \to D \to X$ is a piece in $X$.\end{de}

Observe that the definitions of a \emph{piece} in $D$ and in $X$ are different (cf.\ Definition~\ref{def:discpiece} and Definition~\ref{def:graphpiece}). We use the same name as it will always be clear out of context what piece we consider.

\begin{lem}[Lyndon-van Kampen Lemma]\label{l:vank}
Let $X$ be a thickened graphical complex and let $C \to X$ be a closed homotopically trivial path. Then
\begin{enumerate} 

\item \label{l:vk1} there exists a (possibly singular) disc diagram $D\to X$ such that the path $C$ factors as $C \to \partial D \to X$, and $C \to \partial D$ is an isomorphism,

\item \label{l:vk2} if a diagram $D \to X$ is not reduced, then there exists a diagram $D_1 \to X$ with smaller area and the same boundary cycle in the sense that there is a commutative diagram:
 \[\begin{tikzcd}
 \partial D_1 \arrow{rd} \arrow{r}{\cong}   & \partial D \arrow{d}\\  &X
                          \end{tikzcd}
                          \]
                          
\item \label{l:vk3} any minimal area diagram $D \to X$ such that $C$ factors as $C \xrightarrow{\cong} \partial D \to X$ is reduced.
\end{enumerate}
\end{lem}
\begin{proof}(1) Since $C$ is null-homotopic, there exists a disc diagram $D\to X$ such that the map $C \to X$ factors as $C \to D\to X$ and  $C \to D$ is the boundary cycle of $D$ (see~\cite[Section 2.2]{ECHL} for a proof). 
%[Note that if $R \to D$ is a  $2$--cell of $D$ then there is a unique lift of the map $\partial R \to X$ to the map $\partial R \to A$.]

(2) Since $D \to X$ is not reduced,  there is a piece $P \to D$ such that $P \to D \to X$ is not a piece. Let $R_1 \to D$ and $R_2 \to D$ be the $2$--cells such that $P$ factors through both of them. %We consider two possible cases:

We will first treat the case when $R_1=R_2$. Let $p_1,$ $ p_2 \colon P \to R_1$ denote the two different maps. % $P\to R_1$. 
 Since $P\to D \to X$ is not a piece, the map 
$\partial R_1/(p_1 \sim p_2) \to X$ lifts to ${\bf \Gamma}$ (the graph $\partial R_1/(p_1 \sim p_2)$ is the quotient of the boundary of $R_1$, obtained by identifying the images $p_1(P)$ and $p_2(P)$ pointwise). 
Assume that $P$ is maximal, i.e.\ it is not a proper subpath of a piece $P' \to D$. The attaching map for $R_1$ can be written as the concatenation $PS_1P^{-1}S_2 \to D$, such that $S_1$ and $S_2$ are closed paths, see Figure~\ref{fig:vkamp}. Either $S_1$ or $S_2$ bounds a (possibly singular) subdiagram $D'$ of $D$, assume that it is $S_1$. Remove from $D$ the open cell $R_1$ together with the path $P$ (retaining its initial vertex) and the subdiagram $D'$ bounded by $S_1$. Call the resulting complex $D''$ (formally $D''$ is not a diagram as it is not contractible). Observe that $D''$ has a hole, whose boundary cycle is precisely $S_2$.

The lift $S_2 \to {\bf \Gamma}$ (given by $S_2 \to \partial R_1/(p_1 \sim p_2) \to {\bf \Gamma}$) is immersed everywhere, except for possibly at its initial vertex. Write $S_2 \to D$ as $Q_1SQ_2  \to D$ where $Q_1$ and $Q_2$ are the maximal paths such that lifts $Q_1 \to {\bf \Gamma}$ and $Q_2^{-1} \to {\bf \Gamma}$ are the same.

First assume that $Q_1 \to D$ and $Q_2^{-1} \to D$ do not meet at any vertex except for the initial one and consider the quotient $\tilde{D}$ of $D''$ obtained by identifying the domains of $Q_1$ and $Q_2^{-1}$. The boundary cycle of a hole is now equal to $S$, and by construction $S$ lifts to an immersed cycle $S \to {\bf \Gamma}$. Therefore we can glue to $\tilde{D}$ a $2$--cell $\tilde{R}$ determined by $S \to {\bf \Gamma}$. The area of the resulting diagram $\tilde{D} \cup \tilde{R}$ is smaller than the one of $D$.

If $Q_1 \to D$ and $Q_2^{-1} \to D$ have some common vertices, write $Q_1 =U_1V_1$ and $Q_2 =V_2U_2$ such that $U_1$ and $U_2^{-1}$ have the same termial vertex and $V_1$ and $V_2^{-1}$ have no common vertices except for the initial one. Remove from $D''$ the subdiagram bounded by the closed loop $U_1U_2$ together with paths $U_1$ and $U_2$, retaining the terminal vertex of $U_1$. The boundary cycle of the hole is now equal to $V_1SV_2$. Since $V_1$ and $V_2^{-1}$ do not have common vertices other than the initial one, we can identify their domains and glue to the resulting diagram a $2$--cell determined by the immersion $S \to {\bf \Gamma}$. This gives a lower area diagram and thus finishes the proof of the case where $R_1=R_2$.

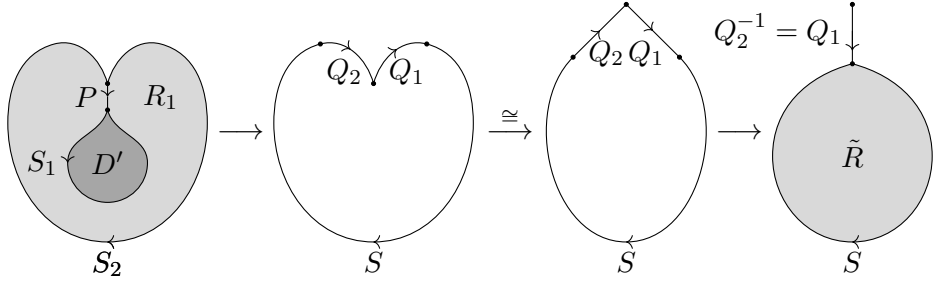
\begin{figure}[!h]
\centering
\begin{tikzpicture}[scale=0.35]
\definecolor{lgray}{rgb} {0.850,0.850,0.850}
\definecolor{dgray}{rgb} {0.650,0.650,0.650}
%titles

%big circle
\node [below] at (0,-4.5) {$S_2$};
\draw [<-] (-0.001,-4.5) to (.001,-4.5);

\draw [fill=lgray]  (0,1.5) [out=105, in=15]  to (-2,3) [out=195, in=90]  to (-3.75,0) [out=270, in=180] to (0,-4.5) [out=0, in=270] to (3.75,0)[out=90, in=-15] to (-2+4,3) [out=165, in=75] to (-4+4,1.5);

\node [below] at (0,-4.5) {$S_2$};
\draw [<-] (-0.001,-4.5) to (.001,-4.5);

%end big circle

\node  at (2,1) {$R_1$};

%piece p

\node [left] at (0,1) {$P$};

\draw [ ->] (0,1.5) -- (0,1);

\draw (0,1)--(0, 0.5);

\draw[fill] (0,1.5)  circle [radius=0.075]; 

\draw[fill] (0,0.5)  circle [radius=0.075]; 

%small circle

\draw [fill=dgray] (0,0.5) [out=255, in=90] to (-1.5, -1.5) [out=270, in=180] to (0,-3) [out=0, in=270] to (1.5, -1.5) [out=90, in=285] to (0, 0.5);

\node [ left] at (-1.5,-1.5) {$S_1$};
\draw [->] (-1.5,-1.459) to (-1.5,-1.501);

\node  at (0,-1.5) {$D'$};

\begin{scope}[shift={(10,0)}]
\node [above] at (-5,-1) {$\longrightarrow$};

\draw[fill] (-2,3)  circle [radius=0.075]; 
\draw[fill] (2,3)  circle [radius=0.075];

\draw[fill] (0,1.5)  circle [radius=0.075]; 

\node [below] at (0,-4.5) {$S$};
\draw [<-] (-0.001,-4.5) to (.001,-4.5);

\node [below] at (-1.1, 2.8) {$Q_2$};
\draw [->] (-0.875, 2.8) to (-0.874, 2.7993) ;

\node [below] at (1.25, 2.8) {$Q_1$};
\draw [<-] (0.875, 2.8) to (0.874, 2.7993) ;

\draw  (0,1.5) [out=105, in=15]  to (-2,3) [out=195, in=90]  to (-3.75,0) [out=270, in=180] to (0,-4.5) [out=0, in=270] to (3.75,0)[out=90, in=-15] to (-2+4,3) [out=165, in=75] to (-4+4,1.5);

\end{scope}

\begin{scope}[shift={(19.5,0)}]

\node [above] at (-4.375,-1) {$\overset{\cong}{\longrightarrow}$};

\draw[fill] (-2,2.5)  circle [radius=0.075]; 
\draw[fill] (2,2.5)  circle [radius=0.075];

\draw[fill] (0,4.5)  circle [radius=0.075]; 

\node [below] at (0,-4.5) {$S$};
\draw [<-] (-0.001,-4.5) to (.001,-4.5);

\node [below] at (-0.75,3.5) {$Q_2$};
\draw [->] (-1,3.5) to (-0.99,3.51);

\node [below] at (0.75,3.5) {$Q_1$};
\draw [<-] (1,3.5) to (0.99,3.51);

\draw  (0,4.5) [out=225, in=45,]  to (-2,2.5) [out= 225, in=180]  to (0,-4.5) [out=0, in=-45]  to (2, 2.5) [out=135, in=-45] to (-4+4,4.5);

\end{scope}

\begin{scope}[shift={(28,0)}]
\node [above] at (-4.25,-1) {$\longrightarrow$};

\draw [fill=lgray]  (0,4.5) to (0,2.25) [out=195, in=90]  to (-3,-1.25) [out=270, in=180] to (0,-4.5) [out=0, in=270] to (3,-1.25) [out=90, in=-15] to  (-4+4,2.25);

\draw[fill] (0,4.5)  circle [radius=0.075]; 
\draw[fill] (0,2.25) circle [radius=0.075]; 

\node [below] at (0,-4.5) {$S$};
\draw [<-] (-0.001,-4.5) to (.001,-4.5);

\node [left] at (0,3.5) {$Q_2^{-1}=Q_1$};

\draw [->] (0,3.5) to (0,2.75);
\node [below] at (0,-0.25) {$\tilde{R}$};

\end{scope}

\end{tikzpicture}
\caption{Replacing the open cell $R_1$ and the subdiagram $D' \cup \mathrm{Int}P$ with the $2$--cell determined by the cycle $S$.}\label{fig:vkamp}
\end{figure}

Now suppose that $P\to D$ factors through two distinct cells $R_1 \to D$ and $R_2 \to D$. Assume that $P$ is maximal and consider the lift $ \partial R_1 \cup_P \partial R_2 \to {\bf \Gamma}$. Let $S_1 \to D$  and $S_2 \to D$ be paths such that the concatenations $PS_1 \to D$ and $PS_2 \to D$ are attaching maps for $R_1$ and $R_2$ respectively. Consider the lift $S_1S_2^{-1} \to {\bf \Gamma}$ of the closed path $S_1S_2^{-1}\to D$. 
If the paths $S_1 \to {\bf \Gamma}$ and $S_2 \to {\bf \Gamma}$ are equal, then we cut out from $D$ open cells $R_1$ and $R_2$ together with the interior of the image of the path $P$ and we ``sew up'' the resulting hole. For a proof of this see~\cite[Lemma 2.16]{McWi}.

We may therefore assume that $S_1 \to {\bf \Gamma}$ and $S_2 \to {\bf \Gamma}$ are not equal. Write $S_1 \to {\bf \Gamma}$ as the concatenation $J_1S_1'T_1 \to {\bf \Gamma}$ and $S_2 \to {\bf \Gamma}$ as $J_2S_2'T_2 \to {\bf \Gamma}$, such that the lifts $J_1 \to {\bf \Gamma} $ and $J_2 \to {\bf \Gamma}$ (resp. $T_1^{-1} \to {\bf \Gamma} $ and $T_2^{-1} \to {\bf \Gamma}$) agree, and both pairs are chosen to be maximal among paths having this property, see Figure~\ref{fig:vkamp2}. Similarly to the case where $R_1=R_2$, we can assume that $J_1 \to D$ and $J_2 \to D$ (resp. $T_1^{-1} \to D$ and $T_2^{-1} \to D$) do not have common vertices except for the initial one, by removing subdiagrams bounded by the appropriate subpaths of $J_1 \to D$ and $J_2 \to D$ (resp. $T_1^{-1} \to D$ and $T_2^{-1} \to D$) if necessary.
 %The subpaths $S_1'$ and $S_2'$ cannot be trivial since $S_1$ is not equal to $S_2$, and therefore the concatenation $S_1'{S_2'}^{-1} \to {\bf \Gamma}$ is a closed immersed path. 

Now remove from $D$ open cells $R_1$ and $R_2$ together with the interior of the image of the path $P$, and consider the quotient $D'$ of $D$ obtained by identifying domains of paths  $J_1$ and $J_2$ and of paths $T_1$ and $T_2$ respectively. The resulting diagram $D'$ has a hole, whose boundary cycle lifts to the closed immersed path $S_1'{S_2'}^{-1}\to {\bf \Gamma}$. Therefore we can attach a $2$--cell $\tilde{R}$ along this path, thus removing the hole. This establishes (2) as the area of the resulting diagram is smaller than the area of $D$.

\begin{figure}[!h]
\centering
\begin{tikzpicture}[scale=0.35]
\definecolor{lgray}{rgb} {0.850,0.850,0.850}
%titles

%first cell

\draw [fill=lgray]  (0.5,1) [out=105, in=-45]  to (-1, 3.25) [out=135, in= 0]   to (-2.5,4) [out=180, in=90]  to (-5,0) [out=270, in=180] to (-2.5,-4) [out=0, in=225] to  (-1, -3.25)  [out=45, in= 255]   to (0.5,-1) [out=90, in=270] to (0.5,1);

\draw [->] (-0.999, 3.249) to (-1, 3.25);
\draw [<-] (-0.999, -3.249) to (-1, -3.25);

\node [above right] at (-1.15, 3.1) {$J_1$};
\draw[fill] (-2.5,4)  circle [radius=0.075]; 

\node [below right] at (-1.15, -3.1) {$T_1$};
\draw[fill] (-2.5,-4)  circle [radius=0.075]; 

\node [left] at (-1.5,2) {$R_1$};

\draw [->] (-5, 0) to (-5, -0.001);

\node [right] at (-5,0) {$S_1'$};

%% second cell

\node [above] at (8,-0.5) {$\longrightarrow$};
\node [above] at (23,-0.5) {$\longrightarrow$};

\begin{scope}[shift={(1,0)}, rotate=180]
\draw [fill=lgray]  (0.5,1) [out=105, in=-45]  to (-1, 3.25) [out=135, in= 0]   to (-2.5,4) [out=180, in=90]  to (-5,0) [out=270, in=180] to (-2.5,-4) [out=0, in=225] to  (-1, -3.25)  [out=45, in= 255]   to (0.5,-1) [out=90, in=270] to (0.5,1);

\draw [<-] (-0.999, 3.249) to (-1, 3.25);
\draw [->] (-0.999, -3.249) to (-1, -3.25);

\node [below left] at (-1.15, 3.1) {$T_2$};
\draw[fill] (-2.5,4)  circle [radius=0.075]; 

\node [above left] at (-1.15, -3.1) {$J_2$};
\draw[fill] (-2.5,-4)  circle [radius=0.075]; 

\node [right] at (-1.5,-2) {$R_2$};

\draw [<-] (-5, 0) to (-5, -0.001);

\node [left] at (-5,0) {$S_2'$};

\end{scope}

%piece

\draw[fill] (0.5,1)  circle [radius=0.075]; 
\draw[fill] (0.5,-1)  circle [radius=0.075]; 

\draw[->] (0.5,0) to (0.5,0.001);

\node [left] at (0.5,0) {$P$};

\begin{scope}[shift={(15,0)}]

\draw   (0.5,1) [out=105, in=-45]  to (-1, 3.25) [out=135, in= 0]   to (-2.5,4) [out=180, in=90]  to (-5,0) [out=270, in=180] to (-2.5,-4) [out=0, in=225] to  (-1, -3.25)  [out=45, in= 255]   to (0.5,-1);

\draw [->] (-0.999, 3.249) to (-1, 3.25);
\draw [<-] (-0.999, -3.249) to (-1, -3.25);

\node [above right] at (-1.15, 3.1) {$J_1$};
\draw[fill] (-2.5,4)  circle [radius=0.075]; 

\node [below right] at (-1.15, -3.1) {$T_1$};
\draw[fill] (-2.5,-4)  circle [radius=0.075]; 

%\node [left] at (-1.5,0) {$R_1$};

\draw [->] (-5, 0) to (-5, -0.001);

\node [right] at (-5,0) {$S_1'$};

%% second cell

\begin{scope}[shift={(1,0)}, rotate=180]
\draw (0.5,1) [out=105, in=-45]  to (-1, 3.25) [out=135, in= 0]   to (-2.5,4) [out=180, in=90]  to (-5,0) [out=270, in=180] to (-2.5,-4) [out=0, in=225] to  (-1, -3.25)  [out=45, in= 255]   to (0.5,-1);

\draw [<-] (-0.999, 3.249) to (-1, 3.25);
\draw [->] (-0.999, -3.249) to (-1, -3.25);

\node [below left] at (-1.15, 3.1) {$T_2$};
\draw[fill] (-2.5,4)  circle [radius=0.075]; 

\node [above left] at (-1.15, -3.1) {$J_2$};
\draw[fill] (-2.5,-4)  circle [radius=0.075]; 

%\node [right] at (-1.5,0) {$R_2$};

\draw [<-] (-5, 0) to (-5, -0.001);

\node [left] at (-5,0) {$S_2'$};

\draw[fill] (0.5,1)  circle [radius=0.075]; 
\draw[fill] (0.5,-1)  circle [radius=0.075];

\end{scope}

\end{scope}

\begin{scope}[shift={(30,0)}]

\draw   (-2.5, 6) to (-2.5,3.5) ;

\draw [fill=lgray] (-2.5,3.5) [out=195, in=90]  to (-5,0) [out=270, in=165] to (-2.5,-3.5) [out=15, in=270] to (0,0) [out= 90, in=-15] to (-2.5, 3.5);

\draw (-2.5,-3.5)  [out=270, in=90] to (-2.5, -6) ;

\node at (-2.5, 1.5) {$\tilde{R}$};

\node [left] at (-2.5, 4.75) {$J_1=J_2$};

\draw[->] (-2.5,4.75)  to (-2.5, 4.749);  
\draw[fill] (-2.5,3.5)  circle [radius=0.075];

\node [left] at (-2.5, -4.75) {$T_1=T_2$};
\draw[<-] (-2.5,-4.75)  to (-2.5, -4.749);

\draw[fill] (-2.5,-3.5)  circle [radius=0.075];

\draw [->] (-5, 0) to (-5, -0.001);

\draw[fill] (-2.5,6)  circle [radius=0.075]; 
\draw[fill] (-2.5,-6)  circle [radius=0.075]; 

\node [right] at (-5,0) {$S_1'$};
\begin{scope}[shift={(-5,0)}, rotate=180]

\draw [<-] (-5, 0) to (-5, -0.001);

\node [left] at (-5,0) {$S_2'$};

\end{scope}

\end{scope}

\end{tikzpicture}
\caption{Replacement procedure.}\label{fig:vkamp2}
\end{figure}
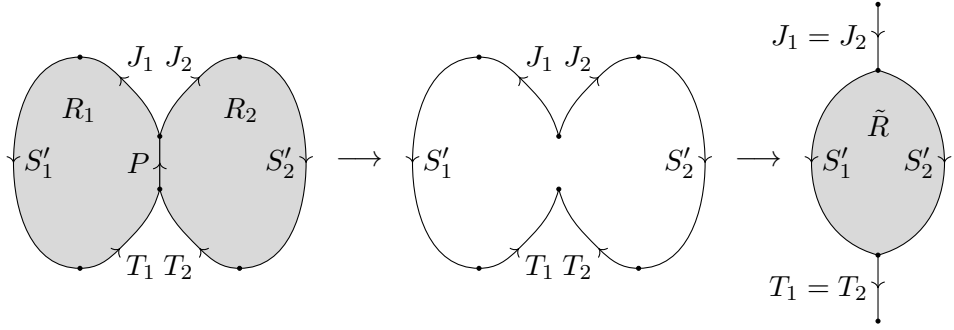

(3) Let $D \to X$ be a minimal area diagram and suppose that it is not reduced. Then applying (2) gives a diagram of lower area and with the same boundary cycle, which contradicts the minimality of $D$.
  \end{proof}

\subsection{Properties of graphical small cancellation complexes}

In this section we define $C(p)$ and $C'(\lambda)$ small cancellation conditions and prove basic results about relators in graphical small cancellation complexes.

\begin{de}\label{de:smallcancgraph} Let $X$ be a thickened graphical complex, and let $p$ be a positive integer and $\lambda$ a positive real number. We say that $X$ satisfies the

\begin{itemize}
\item $C(p)$ \emph{small cancellation condition} if no non-trivial cycle $C \to X$ that factors as $C \to \Gamma_i \to X$
is the concatenation of less than $p$ pieces.

\item $C'(\lambda)$ \emph{small cancellation condition} if for every piece $P \to X$ that is a subpath of a simple cycle $C \to \Gamma_i \to X$ we have $|P| < \lambda \cdot |C|$.

\end{itemize}
We abbreviate the $C(p)$ small cancellation condition to the ``$C(p)$ condition'' and call $X$ a ``$C(p)$ thickened graphical complex'' (we use the same abbreviations in the $C'(\lambda)$ case). Mostly we will be concerned with the $C(p)$ condition for $p \geqslant 6$. Notice that if $p \geqslant q$ then the $C(p)$ condition implies the $C(q)$ condition. Therefore some results will be stated and proven in the $C(6)$ case only.

\end{de}

If $D$ is a disc diagram we define small cancellation conditions in a very similar way, except that a \emph{piece} is understood in the sense of Definition~\ref{def:discpiece}. For clarity we include the definition.

\begin{de} Let $D$ be a disc diagram. We say that $D$ satisfies the

\begin{itemize}
\item 
$C(p)$ small cancellation condition if no boundary cycle $\partial R$ of a $2$--cell $R$ is the concatenation of less than $p$ pieces. 
\item
$C'(\lambda)$ small condition if for every piece $P$ that factors as $P \to \partial R \to D$ for some $2$--cell $R$, we have $|P| < \lambda \cdot |\partial R|$.

\end{itemize}
\end{de}

One can show that the  $C'(\lambda)$ condition implies the $C(\left \lfloor{\frac{1}{\lambda}}\right \rfloor +1)$ condition. This follows from the fact that it is enough to check the $C(p)$ condition on simple cycles.

\begin{prop}\label{reduceddiagram} If $X$ is a $C(p)$ (respectively $C'(\lambda)$) thickened graphical complex and $D \to X$ is a reduced disc diagram, then $D$ is $C(p)$ (respectively $C'(\lambda)$) diagram.
\end{prop}

\begin{proof} The assertion follows immediately from the definitions of a reduced map and a piece.
\end{proof}

The next lemma is the crucial tool in small cancellation theory. It describes the possible shapes of the $C(6)$ disc diagrams. Before stating the lemma we need the following definition.

Let $D$ denote a disc diagram. A \emph{spur} is an edge of a boundary path of $D$ that has a vertex of valence $1$. In this case the boundary path is not immersed. Let $i \geqslant 0$ be an integer. A $2$--cell $R\to D$ is called an \emph{$i$--shell} if its boundary cycle $\partial R$ is the concatenation $P_1 \cdots P_i Q$, such that every $P_j$ is a simple interior path (and hence a piece), and $Q$ is a simple boundary path of $D$. We call $Q$ the \emph{outer path} of $\partial R$.

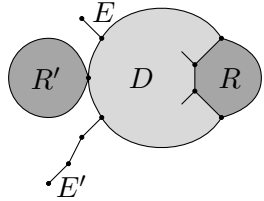
\begin{figure}[!h]
\centering
\begin{tikzpicture}[scale=0.35]
\definecolor{lgray}{rgb} {0.850,0.850,0.850}
\definecolor{dgray}{rgb} {0.650,0.650,0.650}

% disc diagram

\draw [fill=lgray] (0.5,1.5) [out= 120, in =60] to (-4,1.5) [out=240, in=90]  to (-4.5,0) [out=270 , in= 120] to (-4,-1.5) [out=300, in= 240] to (0.5,-1.5) [out=90, in=270] to (0.5, 1.5);

%3shell 
\draw [fill=dgray] (0.5,1.5) to (-0.5, 0.5) to (-0.5, -0.5) to (0.5, -1.5) [out=15, in=270] to (2,0) [out=90, in= -15] to (0.5, 1.5);

\draw (-0.5, 0.5) to (-1, 1);

\draw (-0.5, -0.5) to (-1, -1);

%spur 
\draw  (-4,1.5) to (-4.75, 2.25);

\node [ above right] at (-4.75, 1.75) {$E$};

\node [ right] at (-6.1,-4.1) {$E'$};

\node [right] at (0,0) {$R$};

\draw  (-4,-1.5) to (-4.75, -2.25) to (-5.25, -3.25) to (-6, -4);

\draw[fill] (-4,1.5)  circle [radius=0.075]; 
\draw[fill] (-4.75, 2.25) circle [radius=0.075]; 
\draw[fill] (-4,-1.5)  circle [radius=0.075]; 
\draw[fill] (-4.75, -2.25) circle [radius=0.075]; 
\draw[fill] (-5.25, -3.25)  circle [radius=0.075]; 

\draw[fill] (-6, -4)  circle [radius=0.075];

%0shell
\draw [fill=dgray] (-4.5,0) [out= 105, in= 0] to (-6,1.5)  [out=180, in= 90] to  (-7.5, 0) [out=270, in =180] to (-6,-1.5) [out=0, in= 255] to (-4.5,0); 

\node [left] at (-5.15,0) {$R'$};
\draw[fill] (-4.5,0)  circle [radius=0.075]; 
\node [left] at (-1.65,0) {$D$};

\draw[fill]  (-0.5, 0.5)  circle [radius=0.075];

\draw[fill]  (-0.5, -0.5)  circle [radius=0.075]; 

\draw[fill]  (0.5,1.5) circle [radius=0.075]; 

 \draw[fill] (0.5,-1.5) circle [radius=0.075]; 

\end{tikzpicture}
\caption{Diagram $D$ with spurs $E$ and $E'$, a $0$--shell $R'$ and a $3$--shell $R$. }\label{fig:spurshell}
\end{figure}

\begin{tw}[Greendlinger's Lemma]\cite[Theorem 9.4]{McWi}
\label{tw:greendlin} Let $D$ be a $C(6)$ disc diagram. Then one of the following holds:

\begin{enumerate}

	\item $D$ is a single $0$--cell or it has exactly one $2$--cell,

	%\item $D$ is a ladder,

	\item $D$ has at least two %(three if ladder is around) 
	spurs or/and $i$--shells with $i \leqslant 3$.

\end{enumerate}
\end{tw}

The statement of Theorem~\ref{tw:greendlin} is actually weaker than the quoted Theorem 9.4 of~\cite{McWi}, which distinguishes two further subcases of case (2). We present the simplified statement for the sake of clarity, as it is sufficient for our purposes.

\begin{lem}\label{lem:intersections}Let $X$ be a simply connected $C(6)$ thickened graphical complex.
%with $p \geqslant 6$. 
Then the following hold:

\begin{enumerate}[label=(\roman*)]

\item \label{1intersect} For every relator $\Gamma_i$, the map $\Gamma_i \to X$ is an embedding.
\item \label{2intersect} The intersection of (the images of) any two relators is either empty or it is a finite tree.
\item \label{3intersect} If three relators pairwise intersect then they triply intersect and the intersection is a finite tree.

\end{enumerate}

\end{lem}

\begin{proof}

 (i) Assume conversely that $\Gamma_i \to X$ is not an embedding. Therefore there exist two distinct vertices of $\Gamma_i$ which are mapped to a single vertex of $X$. Let $P \to \Gamma_i$ be a path joining these vertices. We can assume that $P$ is an immersion. By construction $P$ is non-closed and the projection $P \to \Gamma_i \to X$ is closed. Since $X$ is simply connected, the path $P\to X$ is homotopically trivial. By Lemma~\ref{l:vank}.(\ref{l:vk1}) there exists a disc diagram  $D \to X$ with the boundary cycle $P \to X$. Assume that $D$ is chosen such that the area of $D$ is minimal among all examples of paths $P$ of this type. Hence by Lemma~\ref{l:vank}.(\ref{l:vk3}) diagram $D$ is reduced, and therefore by Proposition~\ref{reduceddiagram} it satisfies $C(6)$ condition.

Thus one of the assertions of Theorem~\ref{tw:greendlin} applies to $D$. Clearly $D$ cannot be trivial as in that case the path $P$ would be trivial. It also does not contain spurs since $P \to X$ is an immersion. Thus %by Theorem~\ref{tw:greendlin} 
it consists of either a single $2$--cell or it contains at least one $i$--shell $R$ with $i \leqslant 3$ and outer path $Q$, such that the endpoint of $P \to X$ is not contained in the interior of $Q$.

In the case when $D$ consists of a single $2$--cell, its boundary path $P\to D \to X$ lifts to a closed path in some $\Gamma_j$. 
This lift cannot be equal to the path $P \to \Gamma_i$ we started with, since by the assumption $P\to \Gamma_i$ is not a closed path. Hence $P\to X$ is a piece and since it is a non-trivial closed path, this violates the $C(6)$ hypothesis.

Now suppose $R$ is an $i$--shell with $i \leqslant 3$ and the interior of its outer path $Q \to X$ avoids the endpoint of $P \to X$, see Figure~\ref{fig:1intersect}. We claim that $Q \to X$ is a piece. If it is not the case, then the lift $Q \to \Gamma_i$ determined by the path $P \to \Gamma_i$ extends to a lift $\partial R \to \Gamma_i$.

\begin{figure}[!h]
\centering
\begin{tikzpicture}[scale=0.35]
\definecolor{lgray}{rgb} {0.850,0.850,0.850}
\definecolor{dgray}{rgb} {0.650,0.650,0.650}
%titles

%Gamma

\draw   (0,1.5) [out=105, in=15]  to (-2,3) [out=195, in=90]  to (-5,-1) [out=270, in=180] to (0,-5.5) [out=0, in=270] to (5,-1)[out=90, in=-15] to (-2+4,3) [out=165, in=75] to (-4+4,1.5);

\draw[fill] (0,1.5)  circle [radius=0.075]; 

%path p

\draw [fill=lgray] (0,1.5) [out=249, in=60] to (-2.5, -0.5) [out=240, in =135] to (-2.25, -3) [out=315, in=180] to (0,-4) [out=0, in=225] to (2.25, -3) [out=90, in=270] to (2.5, -0.5) [out=135, in=300] to (0, 1.5);

\node [left] at (-2.5, -0.5) {$P$};
\draw[->] (-2.5, -0.5) to (-2.501, -0.502)  ; 

%ishell
\draw [fill=dgray] (2.5, -0.5) to (1, -1) to (1,-2) to (2.25, -3) [out=45, in=270] to (2.9, -1.5) [out=90, in=-45] to (2.5, -0.5); 

\draw[fill] (2.5, -0.5) circle [radius=0.075]; 

%\draw[fill] (1, -1) circle [radius=0.075];
%\draw[fill]  (1,-2) circle [radius=0.075];
\draw[fill]  (2.25, -3) circle [radius=0.075];

\node [right] at (2.9, -1.5) {$Q$};
\draw[->] (2.9, -1.5) to (2.9, -1.499);

\draw (1,-1) to (0.5, -0.5);
\draw(1,-2) to (0.5, -2.5);
\node [right] at (1+0.15,-1.5 -0.15) {$R$};

\node [right] at (1.25, 1.5) {$\Gamma_i$};
%\draw [->] (-1.5,-1.459) to (-1.5,-1.501);

\node [left] at (0,-1.5) {$D$};

\end{tikzpicture}
\caption{Diagram $D$ with an $i$--shell $R$.}\label{fig:1intersect}
\end{figure}
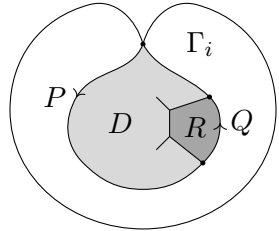

Thus we can remove from $D$ the open cell $R$ together with the interior of the path $Q$, and obtain a lower area diagram $D'$ whose boundary path $P'$ lifts to a non-closed path in $\Gamma_i$. If the resulting path $P' \to \Gamma_i$ is not immersed, we can fold the boundary of $D'$ until all back-tracks are removed. The obtained diagram $D'$ contradicts the minimality of $D$ and hence proves the claim.

Given that $Q\to X$ is a piece,  observe that the cycle $\partial R \to X$ is the concatenation of at most $4$ pieces as $R\to D$ is an $i$-shell with $i \leqslant 3$. This contradicts the $C(6)$ hypothesis and hence establishes \ref{1intersect}.

\vspace{0.25cm}
(ii) Given a relator $\Gamma_i$ recall that a thick cell $Th(\Gamma_i)$ is a $2$--complex obtained by gluing $2$--cells along all immersed cycles in $\Gamma_i$. We shall argue by contradiction.
%[Here we treat relators $\Gamma_i$ as $2$--complexes with $2$--cells attached along immersed cycles] 
Let $\Gamma_1$ and $\Gamma_2$ be two relators that meet along maximal disjoint connected subgraphs $U$ and $V$ and let \[C= Th(\Gamma_1) \cup_{U \sqcup V} Th(\Gamma_2) \smallskip \]
be a $2$--complex obtained by gluing $Th(\Gamma_1)$ and $Th(\Gamma_2)$ along $U$ and $V$. Note that there is an immersion $C \to X$ and consider the closed immersed path $P \to C \to X$ such that $P \to C$ is a generator for the fundamental group $\pi_1(C) \cong \mathbb{Z}$. Let $D \to X$ be a disc diagram whose boundary cycle is $P$ and assume that the area of $D$ is minimal among all examples of this type (i.e.\ among all possible pairs $\Gamma_1$ and $\Gamma_2$ and paths $P$ as above). Hence $D$ is a non-trivial diagram without spurs and the map $D \to X$ is reduced. By Theorem~\ref{tw:greendlin} there is an $i$--shell $R$ in $D$ with $i \leqslant 3$ (if $D$ consists of a single $2$--cell we treat this cell as a $0$--shell). Let $Q$ denote the outer path of $R$ in $D$. 

We claim that any edge of $Q$ is a piece in $X$. To show this assume the contrary, that there is an edge $E \to Q$ that is not a piece. Without loss of generality we can assume that the image of $E$ in $C$ (determined by the path $P \to C$) is contained in the relator $\Gamma_1$. Since $E \to X$ is not a piece, there exists a lift of the boundary $\partial R$ to $\Gamma_i$ extending the lift $E \to \Gamma_1$. Therefore as in the case \ref{1intersect} above, we can remove from $D$ the open cell $R$ together with the interior $\mathrm{Int}(Q)$ and obtain a lower area diagram $D'$ whose boundary path $P'$ is 
obtained from $P$ by pushing the subpath $Q$ through $R$. The paths $P$ and $P'$ are homotopic in $C$ and therefore $P'$ is a generator for $\pi_1(C)$.
Thus $D'$ is a lower area counterexample which contradicts the minimality of $D$ and hence proves the claim (if $D$ consists of a single $2$--cell $R$, then $Q$ is equal to the entire boundary $\partial R$ and therefore pushing $Q$ through $R$ collapses $D$ to a trivial cycle, hence contradicting the fact that $U$ and $V$ are disjoint).

Hence the path $Q\to X$ is the concatenation of $n$ pieces, where $n$ is a positive integer. Since $R$ is an $i$--shell with $i \leqslant 3$, the $C(6)$ hypothesis implies that $n \geqslant 3$. The only situation when this can happen (up to changing roles of $\Gamma_1$ and $\Gamma_2$) is when the path $Q \to X$ travels in $\Gamma_1$ then passes to $\Gamma_2$ through the subgraph $U$ and it comes back to $\Gamma_1$ through the subgraph $V$. More precisely the path $Q$ has a subpath that is the concatenation $U'WV'$ where $U'$ and $V'$ are paths in $\Gamma_1$ which are not entirely contained in $\Gamma_1 \cap \Gamma_2$ and $W$ is path in $\Gamma_2$ such that its initial vertex belongs to the subgraph $U$ and its terminal vertex belongs to the subgraph $V$, see Figure~\ref{fig:2intersect}.

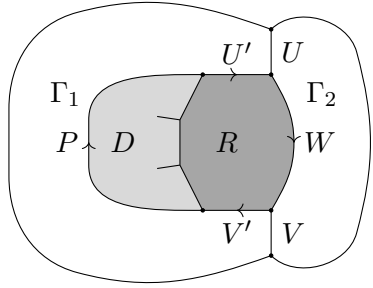
\begin{figure}[!h]
\centering
\begin{tikzpicture}[scale=0.3]
\definecolor{lgray}{rgb} {0.850,0.850,0.850}
\definecolor{dgray}{rgb} {0.650,0.650,0.650}
%titles

%Gamma_2

\draw (0,3) to (0,5) [out= 45, in= 105] to (3.75,4) [out= 285, in= 75] to (3.75,-4) [out=-105, in= 315] to (0,-5) [out=90, in=270] to (0,-3) to (0.5,0) to (0,3);

\node [right] at (0, 4) {$U$};
\node [right] at (0,-4) {$V$};

%D

\draw [fill=lgray] (-3,3) [out= 180, in=90] to (-8,1) [out=270, in=90] to (-8,-1) [out=270, in=180] to (-3,-3) [out=90, in=270] to (-3,3);

\draw[<-] (-8,0) to (-8,-0.001);
\node [left] at (-8,0) {$P$};

%ishell
\draw [fill=dgray] (0,3) to (-3, 3) to (-4, 1) to (-4, -1) to (-3,-3) to (0,-3) [out=60, in= 270] to (1,0) [out=90, in=300] to (0,3);

\draw[->] (-1.5,3) to (-1.499,3);
\node [above] at (-1.5,3) {$U'$};

\draw[<-] (-1.5,-3) to (-1.499,-3);
\node [below] at (-1.5,-3) {$V'$};

\node [right] at (1,0) {$W$};
\draw[->] (1,0) to (1,-0.001);

\draw (-4,1) to (-5, 1.15);
\draw (-4,-1) to (-5,-1.15);

%gamma1

\draw (0,5) [out=160, in=15] to (-8.5,5.75) [out=195, in=90] to (-11.5,1.5) [out= 270, in=90] to (-11.5,-1.5) [out=270, in=165] to (-8.5,-5.75) [out=-15, in=200] to (0,-5);

%path p

\draw[fill]  (0, 3) circle [radius=0.075];

\draw[fill]  (0, 5) circle [radius=0.075];

\draw[fill]  (0, -3) circle [radius=0.075];

\draw[fill]  (0, -5) circle [radius=0.075];

\draw[fill]  (-3, -3) circle [radius=0.075];
\draw[fill]  (-3, 3) circle [radius=0.075];

\node [right] at (1.1, 2.15) {$\Gamma_2$};
\node [right] at (-10.15, 2.15) {$\Gamma_1$};
%\draw [->] (-1.5,-1.459) to (-1.5,-1.501);

\node [left] at (-5.5,0) {$D$};

\node [left] at (-1,0) {$R$};

\end{tikzpicture}
\caption{The outer path of the $i$--shell $R$ is the concatenation $U'WV'$.}\label{fig:2intersect}
\end{figure}

Notice that if both endpoints of $W$ belong to one component of $\Gamma_1 \cap \Gamma_2$, say $U$ (but $W$ is not entirely contained in $U$), then we have a contradiction, as taking any path $W' \to U$ connecting endpoints of $W$ gives a cycle $WW' \to X$ that is a concatenation of $2$ pieces. 

Hence assume that we are in the situation shown in Figure~\ref{fig:2intersect}. We have two cases to consider:

\begin{itemize}
	\item[a)] The path  $U'WV' \to X$ is not closed. Let $\Gamma_3$ denote the relator containing a lift of the cycle $\partial R$ and let $\overline{U}$ and $\overline{V}$ be the maximal connected components of $\Gamma_1 \cap \Gamma_3$ that contain paths $U'$ and $V'$ respectively. We claim that the intersection $\overline{U} \cap \overline{V}$ is empty. Assume it is not the case and pick a path $T \to \overline{U} \cup \overline{V} \subset \Gamma_1 \cap \Gamma_3$ joining the endpoint of $V'$ to the origin of $U'$. The concatenation $U'WV'T \to X$ is then a closed path that is the concatenation of two pieces: $W$ (lifts to $\Gamma_2$ and $\Gamma_3$)  and $V'TU'$ (lifts to $\Gamma_1$ and $\Gamma_3$). This is a contradiction provided that the cycle $U'WV'T \to X$ is non-trivial, i.e.\ it does not factor through a map to a tree. However, if it was trivial then $T$ would be equal to $(U'WV')^{-1}$ and hence there would be a path in $\Gamma_1 \cap \Gamma_2$ joining subgraphs $U$ and $V$, contradicting the assumption that $U$ and $V$ are disjoint. Consequently, the intersection $\overline{U} \cap \overline{V}$ is empty.

Thus we can replace $C$ with $C'= Th(\Gamma_1) \cup_{ \overline{U} \sqcup \overline{V}} Th(\Gamma_3)$ and $D$ with $D'=D \setminus (\mathrm{Int}(R) \cup \mathrm{Int}(Q))$ and the path $P$ with the path $P'$ obtained by pushing the subpath $Q$ through $R$. After removing possible back-tracks (in order for $P'$ to be an immersion), we get a lower area counterexample.

	\item[b)] The path  $U'WV' \to X$ is closed. Let $\Gamma_3$ be the same as in case a) above. Then we get a contradiction as $U'WV' \to X$  is the concatenation of two pieces: $V'U'$ and $W$. Notice that the terminal vertex of $V'$ and the initial vertex of $U'$ %have to 
	lift to the same vertex of $\Gamma_3$ for otherwise $\Gamma_3 \to X$ is not an embedding what contradicts (i). 
\end{itemize}

This shows that the intersection $\Gamma_1 \cap \Gamma_2$ is connected. Notice that there is no simple cycles in $\Gamma_1 \cap \Gamma_2$ as any simple cycle $C \to \Gamma_1 \cap \Gamma_2$ would be a piece itself. Therefore $\Gamma_1 \cap \Gamma_2$ is a tree.

\vspace{0.15cm}

(iii) Let $\Gamma_1$, $\Gamma_2$ and $\Gamma_3$ be relators that pairwise intersect but do not triply intersect, and let $U_{ij}= \Gamma_i \cap \Gamma_j$ for $i,j \in \{1,2,3\}$. Notice that the $U_{ij}$'s are disjoint from each other as otherwise there would be a triple intersection. Let $C$ denote the union of $Th(\Gamma_1)$, $Th(\Gamma_2)$ and $Th(\Gamma_3)$ along subgraphs $U_{12}$, $U_{13}$ and $U_{23}$ and let $P \to C \to X$ be an immersed path such that $P\to C$ is a generator for $\pi_1(C)$. Let $D\to X$ be a disc diagram for $P$ and suppose that the area of $D$ is minimal among all examples of triples $\Gamma_1$, $\Gamma_2$ and $\Gamma_3$ and paths $P$ as above. Thus $D \to X$ is reduced. Proceeding as in the proof of (ii) we conclude that $D$ contains an $i$--shell $R$ with $i\leqslant 3$, with outer path $Q$, such that every edge of $Q$ is a piece. Since $R$ is an $i$--shell with $i\leqslant 3$, the $C(6)$ hypothesis implies that $Q$ has a subpath $V_1V_2V_3$ such that $V_1$, $V_2$ and $V_3$ are non-trivial paths in $\Gamma_1$, $\Gamma_2$ and $\Gamma_3$ respectively, neither of them being contained entirely in an appropriate double intersection, and such that the origin of $V_2$ belongs to $U_{12}$ and the endpoint of $V_2$ belongs to $U_{23}$, see Figure~\ref{fig:3intersect}.

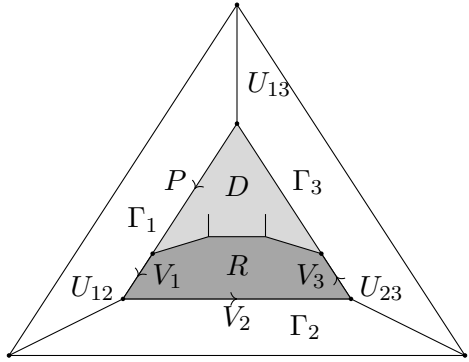
\begin{figure}[!h]
\centering
\begin{tikzpicture}[scale=0.3]
\definecolor{lgray}{rgb} {0.850,0.850,0.850}
\definecolor{dgray}{rgb} {0.650,0.650,0.650}
%titles

%Gamma_2

%disc
\draw   [fill=lgray] (-5,-3) to (5, -3) to (0,4.75) to (-5,-3);

\draw[->] (-1.85, 1.875) to (-1.855,1.867);

\node [left] at (-1.75, 2.25) {$P$};

% ishell
\draw [fill=dgray]  (-3.7,-1) to (-1.25, -0.25) to (1.25, -0.25)  to (3.7,-1) to (5,-3) to (-5,-3) to (-3.7, -1);

\draw[fill] (3.7,-1) circle [radius=0.075];
\draw[fill] (-3.7,-1) circle [radius=0.075];

\draw[<-]  (4.35, -2) to (4.356, -2.01);

\node [left] at (4.35, -2) {$V_3$};
\node [right] at (-4.15, -2) {$V_1$};
\node  at (0, -3.85) {$V_2$};

\draw[->]  (0, -3) to (0.001, -3);

\draw[->]  (-4.35, -2) to (-4.356, -2.01);

\draw  (-1.25, -0.25) to (-1.25, 0.75);

\draw  (1.25, -0.25) to (1.25, 0.75);

\draw (-5,-3) to (-10, -5.5);

\draw[fill]  (-5,-3)  circle [radius=0.075];

\draw[fill]  (-10, -5.5)  circle [radius=0.075];

\node [above right] at (-7.75, -3.5) {$U_{12}$};

\node [above left] at (7.75, -3.5) {$U_{23}$};

\node [right] at (0, 6.5) {$U_{13}$};

\draw[fill]  (5,-3)  circle [radius=0.075];
\draw[fill]  (10, -5.5)  circle [radius=0.075];

\draw (5,-3) to (10, -5.5);

\draw[fill]  (0, 4.75)  circle [radius=0.075];
\draw[fill]  (0, 10)  circle [radius=0.075];

\draw (0,4.75) to (0, 10);

\draw (-10,-5.5) to (10, -5.5) to (0,10) to (-10,-5.5);

\node [left] at (4.25, 2.15) {$\Gamma_3$};
\node [right] at (-5.25, 0.5) {$\Gamma_1$};
\node  at (3, -4.25) {$\Gamma_2$};

\node at (0, 2) {$D$};

\node  at (0,-1.5) {$R$};

\end{tikzpicture}
\caption{The outer path of the $i$--shell $R$ is the concatenation $V_1V_2V_3$.}\label{fig:3intersect}
\end{figure}

Similarly as in (ii) we consider two cases:
\begin{itemize}
\item[a)] The path $V_1V_2V_3 \to X$ is not closed. Let $\Gamma_4$ denote the relator containing a lift of the cycle $\partial R$ and let $U_{14}=\Gamma_1 \cap \Gamma_4$ and $U_{34}=\Gamma_3 \cap \Gamma_4$. We claim that the triple intersection $\Gamma_1 \cap \Gamma_3 \cap \Gamma_4$ is empty. Assume conversely that there exists a vertex $v \in \Gamma_1 \cap \Gamma_3 \cap \Gamma_4$. Choose paths $T_1 \to U_{14} \to X$ joining $v$ to the initial vertex of $V_1$ and $T_2 \to U_{34} \to X$ joining the terminal vertex of $V_3$ to $v$. These paths exist because by (ii) the intersections $U_{14}$ and $U_{34}$ are connected. The concatenation $T_1V_1V_2V_3T_2 \to X$ is a non-trivial closed path that is the concatenation of three pieces: $T_1V_1$, $V_2$ and $V_3T_2$. This contradicts the $C(6)$ hypothesis and hence proves the claim.
%Now let $\overline{V_1}$ be the connected component of $\Gamma_1 \cap \Gamma_4$ that contains path $V_1$ and let $\overline{V_3}$ be the connected component of $\Gamma_3 \cap \Gamma_4$ that contains path $V_3$.  

Now replace $C$ with $C'$ which is the union of $Th(\Gamma_1)$, $Th(\Gamma_3)$ and $Th(\Gamma_4)$ along the subgraphs $U_{13}$, $U_{14}$ and $U_{34}$, replace $D$ with $D'=D \setminus (\mathrm{Int}(R) \cup \mathrm{Int}(Q))$ and replace the path $P$ with the path $P'$ obtained by pushing the subpath $Q$ through $R$. This gives a lower area counterexample.

\item[b)] The path $V_1V_2V_3 \to X$ is closed. Then it is the concatenation of $3$ pieces, hence we get a contradiction with the $C(6)$ hypothesis.
%\qedhere
\end{itemize}

It remains to show that the intersection $\Gamma_1 \cap \Gamma_2 \cap \Gamma_3$ is a tree. First we show that it is connected. Assume the converse and let $u$ and $v$ be vertices lying in different connected components of  $\Gamma_1 \cap \Gamma_2 \cap \Gamma_3$. Since double intersections are connected we can pick paths $P \to \Gamma_1 \cap \Gamma_2$ and $Q \to \Gamma_1 \cap \Gamma_3$ both joining $u$ to $v$. The concatenation $PQ^{-1}$ is then a closed path which is non-trivial since $u$ and $v$ lie in different connected components of $\Gamma_1 \cap \Gamma_2 \cap \Gamma_3$. Since $PQ^{-1}$ is the concatenation of two pieces we get a contradiction. Therefore $\Gamma_1 \cap \Gamma_2 \cap \Gamma_3$ is connected. The proof that it is a tree is the same as in case \ref{2intersect} above. 
\end{proof}

\begin{lem}\label{lem:flag} Let $X$ be a simply connected $C(6)$ thickened graphical complex and consider a finite collection of relators $\{\Gamma_i \to X\}_{i \in \{0, \ldots, n\}}$. If for every $i,j \in \{0,\ldots, n\}$ the intersection $\Gamma_i \cap \Gamma_j$ is non-empty then the intersection $\bigcap_{i \in \{0, \ldots, n\}} \Gamma_i$ is a non-empty tree.
\end{lem}

\begin{proof}

Consider first the intersection $\Gamma_0 \cap (\Gamma_1 \cup \ldots \cup \Gamma_n)$. This intersection is connected by Lemma~\ref{lem:intersections}.\ref{2intersect}\ref{3intersect}. We claim that $\Gamma_0 \cap (\Gamma_1 \cup \ldots \cup \Gamma_n)$
 is a tree. 

Assuming the claim we proceed with the proof of the lemma.\ By Lemma~\ref{lem:intersections}.\ref{2intersect}\ref{3intersect} all intersections $\{\Gamma_0 \cap \Gamma_i\}_{i\in \{1,\ldots, n\}}$ are pairwise intersecting, non-empty subtrees of a tree $\Gamma_0 \cap (\Gamma_1 \cup \ldots \cup \Gamma_n)$. Therefore by the Helly property of trees the intersection $\bigcap_{i \in \{0, \ldots, n\}} \Gamma_i$ is a non-empty tree.

It remains to prove the claim. Assume conversely that there is a non-trivial simple cycle $C \to \Gamma_0 \cap (\Gamma_1 \cup \ldots \cup \Gamma_n)$. Let $\Gamma_j$ be any relator different from $\Gamma_0$ through which $C$ passes and let $P_j$ be a maximal subpath of $C$ that lifts to $\Gamma_j$. Choose $P_i$ and $P_k$ to be the paths that lift to $\Gamma_i$ and $\Gamma_k$ respectively, such that the concatenation $P_iP_jP_k$ is a maximal subpath $C$ with these properties (for $\Gamma_j$ fixed), see Figure~\ref{fig:flag}. If $P_{i}P_jP_{k}=C$ then we get a contradiction with the fact that $X$ is a $C(6)$ complex (the same happens if already $P_{i}P_j$ or $P_{j}P_k$ is equal to $C$).

\begin{figure}[!h]
\centering
\begin{tikzpicture}[scale=0.3]
\definecolor{lgray}{rgb} {0.850,0.850,0.850}
\definecolor{dgray}{rgb} {0.650,0.650,0.650}

\draw (0,2.5) ellipse (2.25cm and 3.5cm);
\draw (0,-2.5) ellipse (2.25cm and 3.5cm);

\begin{scope}[shift={(0,0)}, rotate=180]
\draw [domain=-60:0] plot ({2.25*cos(\x)}, {2.25*sin(\x)});

\draw [domain=-60:0] plot ({5.5*cos(\x)}, {5.5*sin(\x)});

\draw  ({2.25*cos(0)}, {2.25*sin(0)}) to ({5.5*cos(0)}, {5.5*sin(0)});
\draw  ({2.25*cos(-60)}, {2.25*sin(-60)}) to ({5.5*cos(-60)}, {5.5*sin(-60)});
\end{scope}

\draw [domain=-60:60] plot ({2.25*cos(\x)}, {2.25*sin(\x)});

\draw [domain=-60:60] plot ({5.5*cos(\x)}, {5.5*sin(\x)});

\draw  ({2.25*cos(60)}, {2.25*sin(60)}) to ({5.5*cos(60)}, {5.5*sin(60)});
\draw  ({2.25*cos(-60)}, {2.25*sin(-60)}) to ({5.5*cos(-60)}, {5.5*sin(-60)});

\draw [fill=white] (0, 0)  circle [radius=3];
\draw [gray,  thick] (0, 0)  circle [radius=3];

\begin{scope}[shift={(0,0)}, rotate=180]
\draw [dashed, domain=-60:0] plot ({2.25*cos(\x)}, {2.25*sin(\x)});

\draw [dashed, domain=-60:0] plot ({5.5*cos(\x)}, {5.5*sin(\x)});

\draw [dashed] ({2.25*cos(0)}, {2.25*sin(0)}) to ({5.5*cos(0)}, {5.5*sin(0)});
\draw [dashed] ({2.25*cos(-60)}, {2.25*sin(-60)}) to ({5.5*cos(-60)}, {5.5*sin(-60)});
\end{scope}

\draw [dashed, domain=-60:60] plot ({2.25*cos(\x)}, {2.25*sin(\x)});

\draw [dashed, domain=-60:60] plot ({5.5*cos(\x)}, {5.5*sin(\x)});

\draw [dashed] ({2.25*cos(60)}, {2.25*sin(60)}) to ({5.5*cos(60)}, {5.5*sin(60)});
\draw [dashed] ({2.25*cos(-60)}, {2.25*sin(-60)}) to ({5.5*cos(-60)}, {5.5*sin(-60)});

\draw [dashed] (0,2.5) ellipse (2.25cm and 3.5cm);
\draw [dashed] (0,-2.5) ellipse (2.25cm and 3.5cm);

%arcs 

\draw [ thick, domain=42:138] plot ({3*cos(\x)}, {3*sin(\x)});

\draw [ thick,  domain=138:180] plot ({3*cos(\x)}, {3*sin(\x)});

\draw [thick,  domain=-60:42] plot ({3*cos(\x)}, {3*sin(\x)});
\draw [thick,  domain=-138:-60] plot ({3*cos(\x)}, {3*sin(\x)});

\draw[fill]  ({3*cos(138)}, {3*sin(138)})  circle [radius=0.15];
\draw[fill]  ({3*cos(180)}, {3*sin(180)})  circle [radius=0.15];
\draw[fill]  ({3*cos(-60)}, {3*sin(-60)})  circle [radius=0.15];
\draw[fill]  ({3*cos(42)}, {3*sin(42)})  circle [radius=0.15];
\draw[fill]  ({3*cos(-138)}, {3*sin(-138)})  circle [radius=0.15];

\node  at ({3.7*cos(200)}, {3.7*sin(200)})  {$C$};
%path Q
\draw (0,3) [out=255, in= 75] to (0,-3);
\draw[fill=white] (0,3) circle [radius=0.15];
\draw[fill=white] (0,-3) circle [radius=0.15];

\node  at (0.6,0) {$Q$};

\node [above] at (-0.5,3) {$v_j$};
\node [below ] at (-.5,-3) {$v_l$};

%nodes 
\node [right] at (-0.25,4.5) {$\Gamma_j$};
\node at (4.25,0) {$\Gamma_k$};
\node at (-3.6,2.25) {$\Gamma_i$};
\node [right] at (-0.25,-4.5) {$\Gamma_l$};

\end{tikzpicture}
\caption{Cycle $C$ partially covered by relators.}\label{fig:flag}
\end{figure}
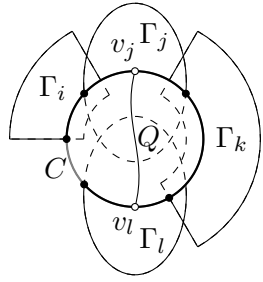

 Hence assume that it is not the case and let $P_l$ be the subpath of $C$ that lifts to $\Gamma_l$ and appears right after $P_k$. Choose vertices $v_j \in P_j$ and $v_l \in P_l$, such that $v_l \notin P_k \cup P_i$. 

 Since the intersection $\Gamma_j \cap \Gamma_0 \cap \Gamma_l$ is non-empty, there is a path 
 $Q \to \Gamma_0 \cap (\Gamma_j \cup \Gamma_l)$ joining $v_j$ to $v_l$. If $Q$ is equal to the subpath of $C$ from $v_j$ to $v_l$ which contains $P_k$ then we get a contradiction with the choice of $P_k$, as in such case $P_l$ appears right after $P_j$ and covers a larger portion of $C$. Similarly if $Q$ is equal to the subpath of $C$ from $v_j$ to $v_l$ which contains $P_i$ then we get a contradiction with the choice of $P_i$. Otherwise the concatenation of the subpath of $C$ from $v_j$ to $v_l$ containing $P_k$ with $Q$ is a non-trivial cycle in $\Gamma_0$ that is covered by the images of three relators $\Gamma_j, \Gamma_k$ and $\Gamma_l$. This contradicts the $C(6)$ condition and therefore finishes the proof of the claim.
\end{proof}

\section{Dual of a $C(p)$ complex is $p$--systolic}\label{sec:dualisisystolic}

Let $X$ be a simply connected $C(p)$ thickened graphical complex $X$ and suppose that $p \geqslant 6$. The purpose of this section is to construct a $p$--systolic simplicial complex $W(X)$ such that any group acting on $X$, acts naturally on $W(X)$.
Furthermore, after replacing $X$ with a ``non-thickened'' graphical complex $X'$ we show that $X'$ and $W(X)$ are $G$--homotopy equivalent. This replacement is necessary, as in general the thickened complex contains non-trivial $2$--spheres, whereas systolic complexes are contractible.
Roughly speaking, the non-thickened graphical complex has the same $1$--skeleton as the thickened one, but instead of thick cells, it has topological cones glued along relators. The non-thickened complex, combinatorially being equivalent to the thickened one, has better topological properties (in particular it is contractible).

\subsection{Equivariant nerve theorem}
Our main tool in showing that $X$ and $W(X)$ are $G$--homotopy equivalent is the Equivariant Nerve Theorem. This theorem is formulated in the abstract language of $G$--posets, therefore we begin by recalling some terminology.\medskip

A $G$--poset is a partially ordered set with an order-preserving action of a group $G$. A \emph{geometric realisation} of a poset $X$ is a simplicial complex  $|X|$  whose $n$--simplices are chains $x_0 \leqslant x_1 \leqslant \ldots \leqslant x_n$ in $X$.
If $X$ is a $G$--poset then its geometric realisation $|X|$ is naturally a $G$--simplicial complex. 
In the realm of $G$--posets and $G$--simplicial complexes, natural morphisms to consider are
\emph{$G$--maps}, that is, $G$--equivariant maps. Using $G$--maps one defines \emph{$G$--homotopy equivalences} and \emph{$G$--contractibility} the same way as for non-equivariant versions. 
For an element $y$ of a $G$--poset (a simplex $\sigma$ of a $G$--simplicial complex) its stabilizer is denoted by $G_y$ (resp.\ $G_{\sigma}$).
All topological notions applied to a poset $X$ are to be understood as corresponding notions applied to its geometric realisation $|X|$.
For an element $y$ of a poset $Y$ define the subposet \[Y_{\leqslant y} = \{x \in Y \mid x \leqslant y\}.\]

The following theorem is an equivariant analogue of the celebrated Quillen's ``Theorem A''.

\begin{tw}[{\cite[Theorem 1]{TW}}]
\label{t:A}
Let $G$ be a group and let $f\colon X\to Y$ be a $G$--map between $G$--posets $X$ and $Y$. If for every $y\in Y$, the preimage $f^{-1}(Y_{\leqslant y})$ is $G_y$--contractible then $f$ is a $G$--homotopy equivalence.
\end{tw}

Let $(X,\leqslant)$ be a poset. We say that a subset $U \subset X$ is closed with respect to $\leqslant$, if for any $x\in U$ and any $y$ such that $y \leqslant x$, we have $y\in U$.
A \emph{cover} of a poset $(X,\leqslant)$ is a family $\mathcal{U}$ of subsets of $X$, such that every $U \in \mathcal{U}$ is closed with respect to $\leqslant$, and  $\bigcup_{U \in \mathcal{U}}U=X$.

The  \emph{nerve} $N(\mathcal{U})$ of a cover $\mathcal{U}$ is a simplicial complex whose vertex set is $\mathcal{U}$, and vertices $U_0, \ldots, U_n$ span an $n$--simplex of $N(\mathcal{U})$  if and only if $\bigcap_{0 \leqslant i \leqslant n}U_i \neq \emptyset$. If $G$ acts on $X$ and for any element $U \in \mathcal{U}$ and any $g \in G$ we have $gU \in \mathcal{U}$ then we say that $\mathcal{U}$ is \emph{$G$--cover}. In this case the $G$--action on $X$ induces the $G$--action on $N(\mathcal{U})$. In particular, element $g\in G$ stabilises a simplex $\sigma$ of $N(\mathcal{U})$ if and only if
$g$ leaves the intersection $\bigcap_{U\in \sigma} U \subset X$ invariant.

\begin{de}
A $G$--cover $\mathcal{U}$ of a $G$--poset $X$ is \emph{$G$--contractible} if for any simplex $\sigma$ of $N(\mathcal{U})$, the subposet $\bigcap_{U\in \sigma} U$ is a $G_{\sigma}$--contractible subposet of $X$, where $G_{\sigma}$ denotes the $G$--stabiliser of $\sigma$.
\end{de}

The following result and its proof are immediate equivariant analogues of \cite[Theorem 4.5.2]{Sm}. To the best of our knowledge there is no proof of this theorem in the literature.

\begin{tw}[Equivariant Nerve Theorem]
\label{t:nerve}
Let $G$ be a group and let $X$ be a $G$--poset. Let $\mathcal{U}$ be a $G$--contractible cover of $X$. Then $N(\mathcal{U})$ is $G$--homotopy equivalent to $|X|$. 
\end{tw}

\begin{proof}

We work with the face poset $N'(\mathcal{U})$ of $N(\mathcal{U})$, with the reversed inclusion order. More precisely, the elements of $N'(\mathcal{U})$ are simplices of $N(\mathcal{U})$, i.e.\  tuples $\{U_i\}_{i \in I}$ such that $\bigcap_{i \in I}U_i \neq \emptyset $ and $\{U_i\}_{i\in I} \leqslant \{U_j\}_{j \in J}$ in $N'(\mathcal{U})$  if and only if $J \subseteq I$.

The geometric realisation of $N'(\mathcal{U})$ is homeomorphic to $N(\mathcal{U})$ and the $G$--action on $N(\mathcal{U})$ induces a $G$--action on $N'(\mathcal{U})$. We define the map $f\colon X \to N'(\mathcal{U})$ as \[f(x)=\{ U\in \mathcal{U} \mid x\in U \}.\] This is a map of posets since
if $y\leqslant x$ then $y\in U$ whenever $x\in U$, by closedness of $U$. It is straightforward to check that $f$ is a $G$--map.

Let $U_I =  \{U_i\}_{i \in I}$ be an element of $N'(\mathcal{U})$. If $x\in f^{-1}(N'(\mathcal{U})_{\leqslant U_I })$ then $x\in U$, for every $U\in U_I $. Therefore 
\[f^{-1}(N'(\mathcal{U})_{\leqslant U_I })=\bigcap_{U\in U_I }U= \bigcap_{i\in I}U_i.\]
Since the cover $\mathcal{U}$ is $G$--contractible, each preimage $f^{-1}(N'(\mathcal{U})_{\leqslant U_I })$ is $G_{U_I }$--con\-tract\-ible.
Therefore, by Theorem~\ref{t:A}, the map $f$ is a $G$--homotopy equivalence.
\end{proof}

In the remainder of this section we show how to apply Theorem~\ref{t:nerve} to the case of $C(p)$ graphical complexes. For this we need to introduce the ``non-thickened'' graphical complex.\medskip % (see Remark~\ref{rem:nothick}).\newline

Let $\Gamma$ be a finite graph. A \emph{cone} on $\Gamma$ is the quotient space \[ C(\Gamma) =\Gamma \times [0,1]/ \Gamma \times \{1\}.\] 

\begin{de}\label{def:graphcomplex}Let $ \varphi \colon {\bf \Gamma} \to \Theta$ be an immersion of graphs and assume that $\Theta$ is connected. Write ${\bf \Gamma}$ as the union of its connected components ${\bf \Gamma}= \bigsqcup \Gamma_i$ and let $\varphi_i$ denote the composition $\Gamma_i \to {\bf \Gamma} \to \Theta$. Therefore $\varphi= \sqcup \varphi_i$.

A \emph{graphical complex} $X$ is a $2$--complex %with $1$--skeleton $\Theta$ 
obtained by gluing a cone $C(\Gamma_i)$ along each $ \varphi_i \colon \Gamma_i \to \Theta$:
\[X= \Theta \cup_{\varphi} \bigsqcup_{i\in I} C(\Gamma_i).\] 
For a map $\Gamma_i \to X$ a \emph{cone-cell} is the corresponding map $C(\Gamma_i) \to X$.
\end{de}
Notice that $X$ is not a $2$--complex in the sense of Section~\ref{sec:comb2}. However, one can put a structure of a combinatorial $2$--complex on $X$ (or even a simplicial complex) by appropriately subdividing every cone. For most of our purposes though, it will be enough to treat entire cone-cells as ``$2$--cells''. Consequently we would like to treat the graph $\Theta$ as the $1$--skeleton of $X$. In particular, any path $P \to X$ necessarily factors as $P \to \Theta \to X$.

\begin{rem}To an immersion of graphs $\varphi \colon {\bf \Gamma} \to \Theta$ we assigned two complexes: a thickened graphical complex (see Definition~\ref{def:thickcomplex}) and a graphical complex (see Definition~\ref{def:graphcomplex}). Let us denote them by $Th(X)$ and $X$ respectively. We emphasise that both constructions depend only on the map $\varphi \colon {\bf \Gamma} \to \Theta$ and therefore one construction determines another.

Moreover, notice that the fundamental groups of $Th(X)$ and $X$ are isomorphic. Indeed one can construct a map $Th(X) \to X$ which is the identity on $1$--skeleton, and which sends $2$--cells of $Th(X)$ to the cone-cells of $X$. After a suitable subdivision this map becomes combinatorial, and one can easily show that it induces an isomorphism on fundamental groups.
\end{rem}

We now proceed with the definitions of small cancellation conditions for a graphical complex. Notice that both definitions of a piece (Definition~\ref{def:graphpiece}) and of small cancellation conditions (Definition~\ref{de:smallcancgraph}) for a thickened graphical complex depend only on the map ${\bf \Gamma} \to \Theta$. Therefore we can use the exact same definitions for a graphical complex. For the sake of completeness we include the following (tautological) definition.

\begin{de}\label{truegraphicalpieces} Let $X$ be a graphical complex and let $Th(X)$ denote the corresponding thickened graphical complex. 
A path $P\to X$ is a \emph{piece} if the corresponding path $P \to Th(X)$ is a piece. 
Consequently we say that $X$ satisfies $C(p)$ or $C'(\lambda)$ condition if $Th(X)$ does so.
\end{de}

\begin{rem} Definition~\ref{truegraphicalpieces} together with the fact that $\pi_1(Th(X)) \cong \pi_1(X)$ implies that Lemma~\ref{lem:intersections} and Lemma~\ref{lem:flag} hold for a $C(6)$ graphical complex $X$ as well.
\end{rem}

From now on a \emph{graphical complex} will always be understood in the sense of Definition~\ref{def:graphcomplex}. We proceed with the definition of the aforementioned simplicial complex $W(X)$.

\begin{de}\label{def:unionofconecells} Let $X$ be a simply connected $C(p)$ graphical complex for $p \geqslant 6$. Assume that $X$ is the union of its cone-cells,  i.e.\ that every edge and vertex of $X$ is in the image of $C(\Gamma_i) \to X$ for some relator $\Gamma_i$.
Notice that by Lemma~\ref{lem:intersections}.\ref{1intersect} every map $\Gamma_i \to X$ is an embedding, and therefore we can identify a cone-cell $C(\Gamma_i)\to X$ with its image.
Let \[{\bf U}= \{C(\Gamma_i) \mid \Gamma_i \subset {\bf \Gamma}\}\]
be the covering of $X$ by its cone-cells.  Define the simplicial complex $W(X)$ to be the nerve of the covering ${\bf U}$. 
This complex was introduced by D.\ Wise in the classical $C(p)$ setting \cite{Wise}, therefore we will refer to $W(X)$ as the \emph{Wise complex}. 
\end{de}

Notice that any cellular $G$--action on $X$ (i.e.\ cellular on $1$--skeleton and maps cone-cells to cone-cells) induces a simplicial $G$--action on $W(X)$. Our goal is to show that in fact $X$ and $W(X)$ are $G$--homotopy equivalent. To show this, we will present $X$ as a realisation of a certain $G$--poset, and we will find a $G$--cover of this poset whose nerve will be isomorphic to $W(X)$. The claim will then follow from Theorem~\ref{t:nerve}. 

\begin{rem}
We remark that the assumption in Definition~\ref{def:unionofconecells} is not very restrictive. Indeed, if $X$ contains such ``free edges'', i.e.\ edges not contained in any cone-cell, one can consider a new complex $X'$ obtained by gluing to $X$ a cone over every free edge (this cone is homeomorphic to the triangle in this case). The complex $X'$ satisfies the assumptions of Definition~\ref{def:unionofconecells} and any cellular $G$--action on $X$ induces a cellular $G$--action on $X'$. It is straightforward to check that the quotient map $X' \to X$ which retracts every cone over the free edge onto this edge is a $G$--homotopy equivalence.
\end{rem}

Let $X$ be as in Definition~\ref{def:unionofconecells}. We define an associated poset $\mathcal{X}$ as follows. Elements of
$\mathcal{X}$ are cone-cells, edges, and vertices of $X$ ordered by inclusion. The geometric realisation $|\mathcal{X}|$ of the poset $\mathcal{X}$ is homeomorphic to $X$, and if $G$ acts on $X$ then there is an induced action on $\mathcal{X}$, and the homeomorphism is equivariant.

Let $\mathcal{U}$ be the cover of $\mathcal{X}$ given by \[\mathcal{U}= \{ \mathcal{X}_{\leqslant c} \mid c \text{ is a cone-cell of } X \} .\] By construction the cover $\mathcal{U}$ is closed with respect to $\leqslant$ and it is straightforward to check that it is a $G$--cover of $\mathcal{X}$. Observe that the geometric realisation of any element $\mathcal{X}_{\leqslant c}$ of $\mathcal{U}$ is homeomorphic to the cone-cell $c$. Therefore the nerve $N(\mathcal{U})$ is isomorphic to the complex $\wis{X}$. 

\begin{lem}\label{lem:contractocover}
The $G$--cover $\mathcal{U}$ is $G$--contractible.
\end{lem}
\begin{proof}
For any $\sigma \in N(\mathcal{U})$ the geometric realisation of the intersection $\bigcap_{U\in \sigma} U$ is a tree by Lemma~\ref{lem:flag}, hence it is $G_{\sigma}$--contractible.
\end{proof}

\noindent The above discussion together with Lemma~\ref{lem:contractocover}  and Theorem~\ref{t:nerve} gives the following.

\begin{tw}
\label{t:Ghe}
Let $X$ be simply connected $C(6)$ graphical $G$--complex satisfying the assumptions of Definition~\ref{def:unionofconecells}. Then $X$ is $G$--homotopy equivalent to the simplicial complex $\wis{X}$.
\end{tw}

\subsection{Graphical small cancellation groups are systolic}

In this section we show that if $X$ satisfies the $C(p)$ small cancellation condition then the complex $W(X)$ is $p$--systolic, and we use the latter to construct models for the classifying spaces $\eeg$ and $E_{\mathcal{VAB}}G$
for a group $G$ acting properly on $X$. 

\begin{tw}\label{t:ss}Suppose $p \geqslant 6$ and let $X$ be a simply connected $C(p)$ graphical complex. Then its Wise complex $\wis{X}$ is $p$--systolic.
\end{tw}

\begin{proof}
The idea as well as the strategy of the proof come from D.\ Wise who proved this theorem for classical $C(6)$ complexes (cf.\ Theorem 10.6 in~\cite{Wise}). We need to show that $\wis{X}$ is simply connected, flag and that links of vertices of $\wis{X}$ are $p$--large. 

Simple connectedness of $\wis{X}$ follows from Theorem~\ref{t:Ghe}. To show that $\wis{X}$ is flag, suppose that $v_0,\ldots,v_n$ are vertices of $\wis{X}$ which are pairwise adjacent. We claim that these vertices span an $n$--simplex of $\wis{X}$. Let $C(\Gamma_0),\ldots, C(\Gamma_n)$ be the corresponding cone-cells in $X$. By our assumption we have $\Gamma_i \cap \Gamma_j \neq \emptyset$  for all $0 \leqslant i,j\leqslant n$ (cone-cells can intersect only at the relators). Thus by Lemma~\ref{lem:flag} the intersection $\bigcap_{i=1}^n \Gamma_i$ is non-empty, and therefore the vertices $v_0, \ldots, v_n$ span an $n$--simplex of $\wis{X}$.

It remains to show that for any vertex $v \in \wis{X}$ the link $\wis{X}_v$ is $p$--large. Let $(v_1,\ldots,v_n)$ be a cycle in $\wis{X}_v$ of length less than $p$. This corresponds to a sequence $C(\Gamma_1),\ldots, C(\Gamma_n)$ of cone-cells such that $\Gamma_i \cap \Gamma_{i+1} \neq \emptyset$ and all $\Gamma_i$ intersect a fixed relator $\Gamma$.

The intersection $\Gamma \cap (\Gamma_1 \cup \ldots \cup \Gamma_n)$ is a connected graph (cf.\ proof of Lemma~\ref{lem:flag}). We claim that it is a tree. Indeed, any non-trivial cycle in $\Gamma \cap (\Gamma_1 \cup \ldots \cup \Gamma_n)$ is a concatenation of at most $n$ pieces, which contradicts the $C(p)$ hypothesis as $n <p$. Now choose vertices $u_i \in \Gamma \cap \Gamma_i \cap \Gamma_{i+1}$ and let $P_{i} \to \Gamma \cap \Gamma_{i}$ be a non-backtracking path joining $u_{i-1} $ to $u_{i}$. The concatenation of paths $ P_1P_2\cdots P_n$ is a cycle in the tree $\Gamma \cap (\Gamma_1 \cup \ldots \cup \Gamma_n)$. It is straightforward to check that there are two nonconsecutive paths $P_j$ and $P_k$ that intersect, see Figure~\ref{fig:validconfig} on the right. Therefore the cone-cells $C(\Gamma_j)$ and $C(\Gamma_k)$ intersect and this gives a diagonal $(v_j,v_k)$ in a cycle $(v_1,\ldots,v_n)$ in $\wis{X}_v$.\end{proof}

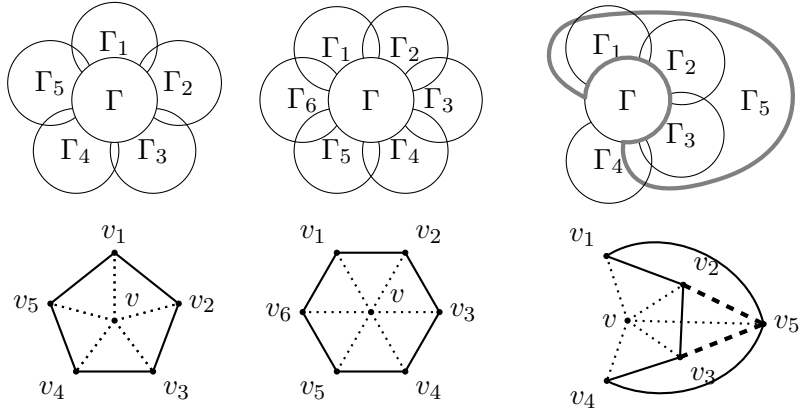
\begin{figure}[!h]
\centering
\begin{tikzpicture}[scale=0.225]
\definecolor{lgray}{rgb} {0.850,0.850,0.850}
\definecolor{dgray}{rgb} {0.650,0.650,0.650}

%validnomral

\draw (2,3) circle [radius=2.5];
\node  at (2,3) {$\Gamma_2$};

\draw (-2,3) circle [radius=2.5];
\node  at (-2,3) {$\Gamma_1$};

\draw (2,-3) circle [radius=2.5];
\node  at (2,-3) {$\Gamma_4$};

\draw (-2,-3) circle [radius=2.5];]
\node  at (-2,-3) {$\Gamma_5$};

\draw (4,0) circle [radius=2.5];
\node  at (4,0) {$\Gamma_3$};

\draw (-4,0) circle [radius=2.5];
\node  at (-4,0) {$\Gamma_6$};

\draw[fill=white] (0,0) circle [radius=2.5];
\node  at (0,0) {$\Gamma$};

%validnormaldual

\begin{scope}[shift={(0,-13)}]

\draw [thick] (-2,4) -- (2,4) -- (4,0.5) -- (2,-3) -- (-2,-3) -- (-4,0.5)-- (-2,4); 

\draw[fill] (-2,4)   circle [radius=0.15];

\draw[fill] (2,4)   circle [radius=0.15];

\draw[fill]  (4,0.5) circle [radius=0.15];

\draw[fill]  (2,-3) circle [radius=0.15];

\draw[fill] (-2,-3)  circle [radius=0.15];

 \draw[fill] (-4,0.5)   circle [radius=0.15];
  \draw[fill] (0,0.5)   circle [radius=0.15];

 \draw [thick, dotted] (0,0.5) -- (-2,4) ;

 \draw [thick, dotted] (0,0.5)  -- (2,4) ;
 \draw [thick, dotted] (0,0.5)-- (4,0.5) ;
 \draw [thick, dotted] (0,0.5)  -- (2,-3);
 \draw [thick, dotted] (0,0.5) -- (-2,-3) ;

 \draw [thick, dotted] (0,0.5) -- (-4,0.5);

\node  [above left] at (-2,4) {$v_1$};
\node [above right] at (0.5,0.5) {$v$};

\node  [above right] at (2,4) {$v_2$};
\node  [right] at (4,0.5) {$v_3$};
\node [below right] at  (2,-3) {$v_4$};
\node [below left] at (-2,-3) {$v_5$};
\node [left] at (-4,0.5) {$v_6$};

\end{scope}

%invalind

\begin{scope}[shift={(-15,0)}]

\draw (0,3.25) circle [radius=2.5];
\node  at (0,3.35) {$\Gamma_1$};

\draw (2.25,-3) circle [radius=2.5];
\node  at (2.25,-3) {$\Gamma_3$};

\draw (-2.25,-3) circle [radius=2.5];]
\node  at (-2.25,-3) {$\Gamma_4$};

\draw (3.75,1) circle [radius=2.5];
\node  at (3.75,1) {$\Gamma_2$};

\draw (-3.75,1) circle [radius=2.5];
\node  at (-3.75,1) {$\Gamma_5$};

\draw[fill=white] (0,0) circle [radius=2.5];
\node  at (0,0) {$\Gamma$};

\end{scope}

%invaliddual
\begin{scope}[shift={(-15,-13)}]

\draw [thick] (0,4) -- (3.75,1) -- (2.25,-3) -- (-2.25,-3) -- (-3.75,1)-- (0,4); 

\draw[fill] (0,4)   circle [radius=0.15];

\draw[fill]  (3.75,1) circle [radius=0.15];

\draw[fill]  (2.25,-3) circle [radius=0.15];

\draw[fill] (-2.25,-3)  circle [radius=0.15];

 \draw[fill] (-3.75,1)   circle [radius=0.15];

  \draw[fill] (0,0)   circle [radius=0.15];

 \draw [thick, dotted]  (0,0)  -- (0,4) ;
 \draw [thick, dotted] (0,0) -- (3.75,1) ;
 \draw [thick, dotted] (0,0)  -- (2.25,-3);
 \draw [thick, dotted] (0,0) -- (-2.25,-3) ;

 \draw [thick, dotted] (0,0) -- (-3.75,1);

 \draw [thick, dotted]  (0,0)  -- (0,4) ;
 \draw [thick, dotted] (0,0) -- (3.75,1) ;
 \draw [thick, dotted] (0,0)  -- (2.25,-3);
 \draw [thick, dotted] (0,0) -- (-2.25,-3) ;

\node [above right] at (0,0.25) {$v$};
\node  [above] at (0,4) {$v_1$};
\node  [right] at (3.75,1) {$v_2$};
\node [below right] at  (2.25,-3) {$v_3$};
\node [below left] at (-2.25,-3) {$v_4$};
\node [left] at (-3.75,1) {$v_5$};

 \draw [thick, dotted] (0,0) -- (-3.75,1);

\end{scope}

%validtrue

\begin{scope}[shift={(15,0)}, rotate=20]

\begin{scope}[shift={(0,-0.15)}]

\begin{scope}[ rotate=74, scale=1.75]

\node  at (-0.3,-4.25) {$\Gamma_5$};
\draw [ultra thick, gray]  (0,-1) [out=120, in=15]  to (-2,0.25) [out=195, in=105]  to (-2.75,-3) [out=285, in=180] to (0,-5.5) [out=0, in=270] to (3,0)[out=90, in=-15] to (2,2.5) [out=165, in=60] to (-4+4,1);
\end{scope}
\end{scope}

\draw (0,3.25) circle [radius=2.5];
\node  at (0,3.35) {$\Gamma_1$};

\draw (2.25,-3) circle [radius=2.5];
\node  at (2.25,-3) {$\Gamma_3$};

\draw (-2.25,-3) circle [radius=2.5];
\node  at (-2.25,-3) {$\Gamma_4$};

\draw (3.75,1) circle [radius=2.5];
\node  at (3.75,1) {$\Gamma_2$};

\draw [fill=white]
 (0,0) circle [radius=2.5];
\node  at (0,0) {$\Gamma$};

\draw [ultra thick, gray, domain=-118.25:158.75] plot ({2.5*cos(\x)}, {2.5*sin(\x)});

\end{scope}

%validtruedual
\begin{scope}[shift={(15,-13)},rotate=18]

\draw [thick] (0,4) -- (3.75,1) -- (2.25,-3) -- (-2.25,-3) [out=-45, in=235] to (7.5,-2.65) [out=85, in=15] to (0,4); 

\draw[fill] (0,4)   circle [radius=0.15];

\draw[fill]  (3.75,1) circle [radius=0.15];

\draw[fill]  (2.25,-3) circle [radius=0.15];

\draw[fill] (-2.25,-3)  circle [radius=0.15];

 \draw[fill] (7.5,-2.65)   circle [radius=0.15];

  \draw[fill] (0,0)   circle [radius=0.15];

 \draw [thick, dotted] (0,0)  -- (0,4) ;
 \draw [thick, dotted] (0,0) -- (3.75,1) ;
 \draw [thick, dotted] (0,0)  -- (2.25,-3);
 \draw [thick, dotted] (0,0) -- (-2.25,-3) ;

 \draw [thick, dotted] (0,0) -- (7.5,-2.65);

 \draw [ultra thick, dashed] (3.75,1) -- (7.5,-2.65);
  \draw [ultra thick, dashed] (2.25,-3) -- (7.5,-2.65);

  \node [left] at (0,0) {$v$};
\node  [above left] at (0,4) {$v_1$};
\node  [above right] at (3.75,1) {$v_2$};
\node [below right] at  (2.25,-3) {$v_3$};
\node [below left] at (-2.25,-3) {$v_4$};
\node [right] at (7.5,-2.65) {$v_5$};

\end{scope}

\end{tikzpicture}
\caption{Vertex links of $W(X)$ and the corresponding subcomplexes of a $C(6)$ complex $X$. On the left there is an illegal configuration leading to a cycle of length $5$ without diagonals. In the middle and on the right legal configurations are shown.}

\label{fig:validconfig}
\end{figure}

\begin{de}\label{de:graphufl}We say that a graphical complex $X$ is (uniformly) locally finite, if after subdividing each cone-cell $C(\Gamma_i)$ into triangles spanned by the edges of $\Gamma_i$ and the apex of the cone $C(\Gamma_i)$, the resulting complex is a (uniformly) locally finite simplicial complex. 
\end{de}

It follows directly from the construction that $X$ is uniformly locally finite if and only if $W(X)$ is so. Consequently, since $X$ and $W(X)$ are $G$--homotopy equivalent, the $G$--action on $X$ is proper if and only if the $G$--action on $W(X)$ is proper. Finally, the $G$--action on $X$ is cocompact if and only if the $G$--action on $W(X)$ is so.
These observations lead to the following corollary, which is interesting in its own right. 

\begin{cor}\label{c:gscsys} Let $G$ be a group acting properly and cocompactly on a simply connected %uniformly locally finite 
$C(p)$ graphical complex for $p\geqslant 6$. Then $G$ acts properly and cocompactly on a %uniformly locally finite 
systolic complex, i.e.\ $G$ is a systolic group.
\end{cor}

Being systolic implies many properties including e.g.\ biautomaticity \cite[Theorem 13.1]{JS2}. For further
results see e.g.\ \cites{JS2,JS3,PP,Osge,OS} and references therein.

We now state and prove the main theorem of this section.

\begin{tw}\label{twclasssmall} Let a group $G$ act properly on a simply connected uniformly locally finite $C(6)$ graphical complex $X$%, where $p \geqslant 6$
. Then: \begin{enumerate}
\item the complex $X$ is a model for $\eg$, 
\item there exists a $3$--dimensional model for $\eeg$,
\item there exists a $4$--dimensional model for $E_{\mathcal{VAB}}G$, provided the action is additionally cocompact.
\end{enumerate}

\end{tw}

\begin{proof}
(1) By Theorem~\ref{t:Ghe} the group $G$ acts properly on a uniformly locally finite systolic complex $W(X)$,
%(see Definition~\ref{de:graphufl})
and hence by Theorem~\ref{modelpiotra} the complex $W(X)$ is a model for $\eg$. Therefore $X$ is a model for $\eg$ as well, since $X$ and $W(X)$ are $G$--homotopy equivalent. 

(2) By Corollary~\ref{commenough} it is enough to find 
%a $2$--dimensional model for $\underline{E}G$, and 
for every $[H] \in [\mathcal{VCY} \setminus \mathcal{FIN}]$ a $2$--dimensional models for $\underline{E}N_G[H]$ and $E_{\mathcal{G}[H]} N_G[H]$. By (1) the complex $X$ may serve as a model for $\underline{E}N_G[H].$  Notice that $G$ acts properly on a systolic complex $\wis{X}$, hence by Lemma~\ref{commmodels}.\ref{jedencommmodels} there exists a $2$--dimensional model for $E_{\mathcal{G}[H]} N_G[H]$.

(3) Since $G$ acts properly and cocompactly on a systolic complex, by Lemma~\ref{lemmanm} it satisfies conditions (NM1) and (NM2) (cf.\ Section~\ref{s:vab2}). Therefore proceeding exactly as in the proof of Theorem~\ref{twvab2}, we obtain a model for $E_{\mathcal{VAB}}G$ of dimension $\mathrm{max}\{4,d \}$ where $d$ is the dimension of a model for $\eeg$. By (2) the latter can be chosen to be at most $3$, hence the claim.
\end{proof}

\section{Examples}
\label{sec:ex}

In this section we provide few classes of examples of groups to which our theory applies.
When relevant, we mention that our constructions give new bounds on dimensions
of classifying spaces.

\subsection{Graphical small cancellation presentations}
A \emph{graphical presentation} $\mathcal P=\langle S \; | \; \varphi \rangle$ is a graph
\[{\bf \Gamma}=\bigsqcup_{i \in I}\Gamma_i,\] and an immersion \[\varphi \colon {\bf \Gamma} \to R_S,\] where
every $\Gamma_i$ is finite and connected, and $R_S$ is a rose, i.e.\ a wedge of circles with edges (cycles) labelled by a set $S$. Alternatively, the map $\varphi \colon {\bf \Gamma} \to R_S$, called a \emph{labelling}, may be thought of
as an assignment: to every edge of ${\bf \Gamma}$ we assign a direction (orientation) and an
element of $S$. 

A graphical presentation $\mathcal P$ defines a group \[G= G(\mathcal{P})= \pi_1(R_S)/\left<\left< 
\varphi_{\ast}(\pi_1 (\Gamma_i))_{i\in I} \right>\right>.\] In other words $G$ is the quotient of the 
free group $F(S)$ by the normal closure of the group generated by all words (over $S\cup S^{-1}$) read along cycles in ${\bf \Gamma}$ (where an oriented edge labelled by $s\in S$ is identified with the edge of the opposite
orientation and the label $s^{-1}$).
A \emph{piece} is a path $P$ labelled by $S$ such that there exist two immersions $p_1 \colon P \to {\bf \Gamma}$ and $p_2\colon P \to {\bf \Gamma}$, and there is no automorphism $\Phi \colon {\bf \Gamma}
\to {\bf \Gamma}$ such that $p_1=\Phi \circ p_2$. The presentation $\mathcal P$ satisfies the \emph{$C(p)$ small
	cancellation condition}, for $p\geqslant 6$, if no cycle in ${\bf \Gamma}$ is covered by less than $p$ pieces; see eg.\ \cite{Gru} for a systematic treatment.

Consider the following  graphical complex (see Definition~\ref{def:graphcomplex}):
\[
X=R_S \cup_{\varphi} \bigsqcup_{i \in I} C(\Gamma_i). \]
The fundamental group of $X$ is isomorphic to $G$. In the universal cover $\widetilde X$ of $X$ there might be
multiple copies of cones $C(\Gamma_i)$ whose attaching maps differ by lifts of Aut$(\Gamma_i)$.
After identifying all such copies, we obtain the complex $\widetilde X^{\ast}$. The group
$G$ acts geometrically, but not necessarily freely on $\widetilde X^{\ast}$.
If $\mathcal P$ is a $C(p)$ graphical small cancellation presentation then the complex
$\widetilde X^{\ast}$ is a $C(p)$ small cancellation complex. Moreover, the complex $\widetilde X^{\ast}$ satisfies the assumptions of Definition~\ref{def:unionofconecells} as long as the map $\varphi \colon {\bf \Gamma} \to R_S$ is surjective. This happens precisely when the presentation $\mathcal{P}$ has no free generators.

Graphical small cancellation presentations provide a powerful tool for constructing groups with often unexpected properties, see e.g.\ \cite{O-sc}. 
For such groups with torsion our result concerning the model for $\eg$ is new.
If a $C(6)$ graphical small cancellation group is torsion-free then it admits a model for
$\eeg$ of dimension at most three, by the work of D.\ Degrijse \cite[Corollary 3]{Die15}. There are however
$C(6)$ graphical small cancellation groups to which Degrijse's result does not apply. 
In such cases our constructions of low-dimensional $\eeg$ and $E_{\mathcal{VAB}}G$ are the only general tools available.

\subsection{Groups acting on $\mathcal{VH}$--complexes} Not all groups acting geometrically on graphical
small cancellation complexes possess graphical small cancellation presentations. The simplest example is $\mathbb{Z}^2$.
It acts simply transitively on a tessellation of the plane by regular hexagons (a simple example of a $C(6)$ complex), but possesses no graphical $C(6)$ presentation. The following class of
examples is more interesting.

The notion of \emph{$\mathcal{VH}$--complexes} was introduced by D.\ Wise \cite{WiseVH}. Recall that a square
complex $X$, i.e.\ a combinatorial $2$--complex whose cells are squares, is a $\mathcal{VH}$--complex if the following holds. The edges of $X$
can be partitioned into two classes $\mathcal{V}$ and $\mathcal{H}$ called \emph{vertical} and \emph{horizontal} edges respectively, such that every square has two opposite vertical and two opposite horizontal edges. %in particular the gluing map respects the types of edges. 

Let $X$ be a simply connected $\mathcal{VH}$--complex that is a $\mathrm{CAT}(0)$ space with respect to the standard piecewise Euclidean metric. We now show how to turn $X$ into a simply connected $C(6)$ graphical small cancellation complex. Subdivide every square into $24$
triangles as shown in Figure~\ref{f:VH} on the left. 
\begin{figure}[h!]
	\centering
	\includegraphics[width=0.9\textwidth]{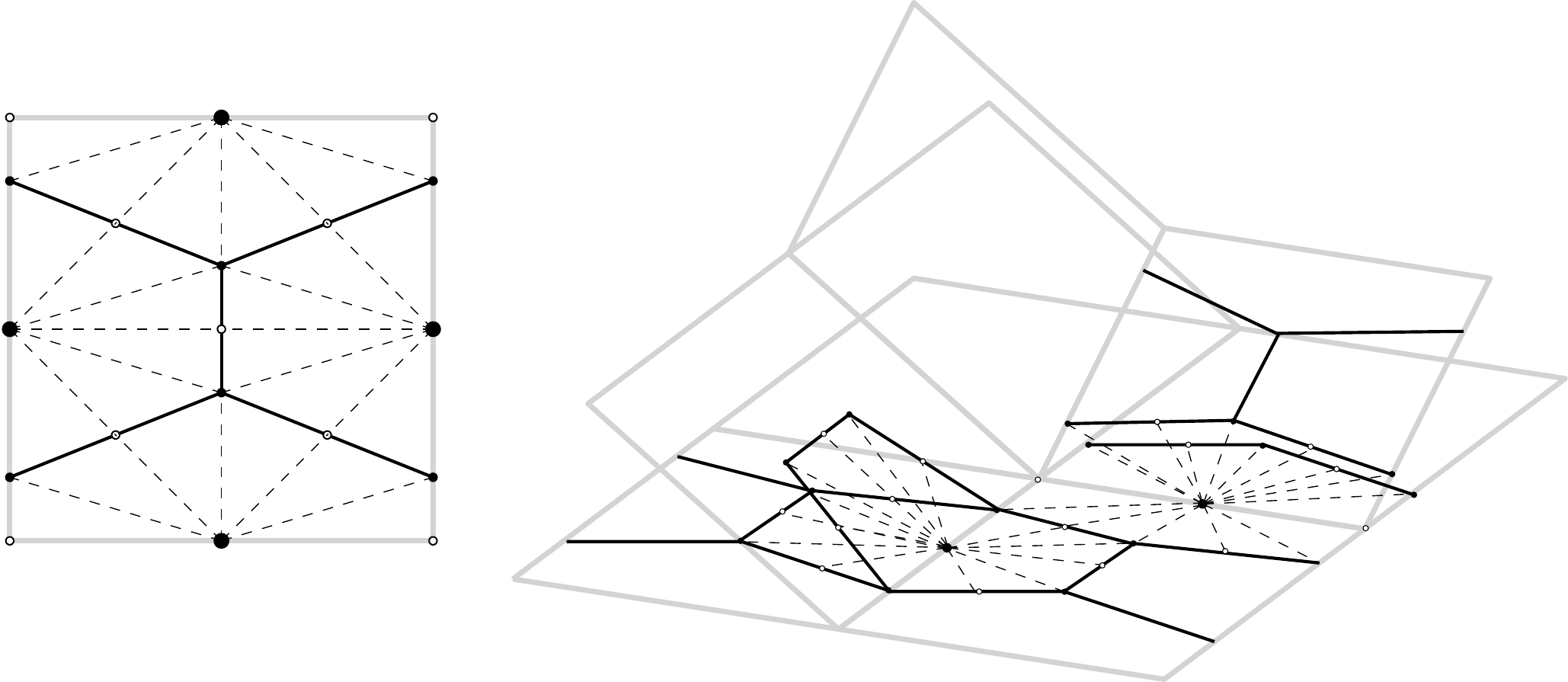}
	\caption{The subdivision of a $\mathcal{VH}$--square into $24$ triangles (left), and the $C(6)$ graphical small cancellation complex structure on a CAT(0) $\mathcal{VH}$--complex. Two
		cones on relators are highlighted: one with a vertical and one with a horizontal apex.}
	\label{f:VH}
\end{figure}
More precisely, the subdivision is invariant with respect to $\mathcal{VH}$--isometries of the square,
the vertical edges are subdivided into four sub-edges each, and the horizontal edges are subdivided into
two sub-edges each. This defines a triangulation of $X$. Call the vertices of this triangulation, being
mid-points of vertical and horizontal edges \emph{vertical} and \emph{horizontal apexes}, respectively.
Consider links of apexes. Such a link is a graph of girth $12$. Two such links intersect in a subgraph
(possibly empty) of diameter at most $2$; see Figure~\ref{f:VH} on the right. Therefore, the complex $X$ 
has a structure of the union of cones on links of apexes (relators). This defines the $C(6)$ graphical small 
cancellation complex $X^{\ast}$. It is clear that every $\mathcal{VH}$--automorphism, i.e.\ an automorphism respecting types of edges of $X$ induces an automorphism of $X^{\ast}$. 

\begin{tw}\label{t:VH}
	Let $X$ be a simply connected $\mathcal{VH}$--complex. Then the complex $X^{\ast}$ is a $C(6)$ graphical small cancellation complex. In particular, every group of $\mathcal{VH}$--automorphisms of $X$ acts
	by automorphisms on $X^{\ast}$. One action is proper and/or cocompact if and only if the other is so.
\end{tw}

T.\ Elsner and P.\ Przytycki \cite{EP} showed that a group acting properly or geometrically on a 
simply connected $\mathcal{VH}$--complex acts, respectively, properly or geometrically on a $3$--dimensional
systolic complex. Theorem~\ref{t:VH} together with Theorem~\ref{t:ss} provide a higher dimensional systolic complex in such a case. Nevertheless, the theorem above equips $\mathcal{VH}$--groups with a new $2$--dimensional structure, extending in a way the Elsner-Przytycki result. 

Of course, $\mathcal{VH}$--complexes carry a natural CAT(0) metric so that constructions of the corresponding
low-dimensional models for $\eg$ and $\eeg$ are available by \cite{Lu09}. Our results provide a $4$--dimensional 
model for $E_{\mathcal{VAB}}G$ for groups acting geometrically on such complexes.

\subsection{Lattices in \texorpdfstring{$\widetilde{A}_2$--}{}buildings} Here we present another example of a group acting
properly on a graphical small cancellation complex.
An $\widetilde A_2$--building is a building with apartments isomorphic to the equilaterally triangulated plane $\mathbb{E}_{\Delta}^2$, see Definition~\ref{def:flat}. 

Consider such a building $Y$. Let $Y'$ be its barycentric subdivision. Define a \emph{dual graph} $\Theta$ of $Y$ as follows.
Vertices of $\Theta$ are edges of $Y$ and triangles of $Y$. There is an edge in $\Theta$ between
every edge of $Y$ and a triangle of $Y$ containing this edge; see Figure~\ref{f:A2}.

The link in $Y'$ of any vertex of $Y$ is a $12$--large graph (a subdivision of a spherical building)
that may be considered as a subgraph of $\Theta$. The complex $Y'$ is thus obtained by attaching
cones on such links %(that are one-balls around vertices of $Y$ in $Y'$) 
to the graph $\Theta$.
Two such cones intersect in a set of diameter at most $2$. Therefore $Y'$ may be seen
as a $C(6)$ graphical small cancellation complex. Lattices in $\mathrm{Isom}(Y)$ act naturally
on $Y'$. Notice that such lattices may be very different from groups in the previous example because
they may have Kazhdan's property (T).

Each $\widetilde A_2$--building possesses a natural structure of a systolic $2$--dimensional complex 
or even a CAT(0) complex. Our results provide a $4$--dimensional 
model for $E_{\mathcal{VAB}}G$ for lattices in its isometry group.

In fact, by exactly the same construction as above one equips any $2$--dimensional $p$--systolic
complex with a structure of a $C(p)$ graphical complex.

\begin{figure}[h!]
	\centering
	\includegraphics[width=0.57\textwidth]{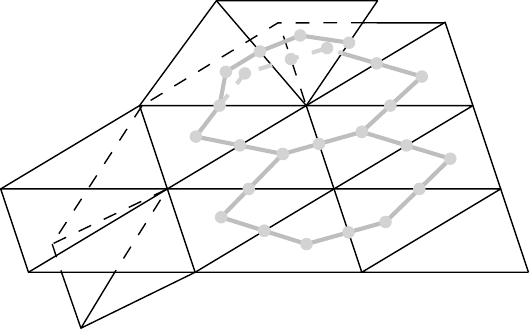}
	\caption{A part of an $\widetilde{A}_2$--building together with a part of its dual graph (thick gray).}
	\label{f:A2}
\end{figure}

\subsection{A $3$--dimensional systolic example} As the last example we present a non-hyperbolic group $G$ acting geometrically
on a $3$--dimensional systolic pseudomanifold $X$ which does not admit a $G$--invariant $\mathrm{CAT}(0)$
metric. 

We start with a simplex of groups $\mathcal G_{54}$, introduced by J.\ \' Swi\polhk atkowski in \cite{Sw} (for some details
on complexes of groups we refer the reader to \cite{JS2,Sw}).
Let $T_{54}$ be a $6$--large triangulation of the flat $2$--torus consisting of $54$ equilateral triangles;
see Figure~\ref{f:g54} on the right (with the opposite sides of the hexagon and the appropriate vertices identified). Let $G_{54}$ be a group of automorphisms of $T_{54}$ generated by reflections
with respect to edges of triangles. For the $G_{54}$--action on $T_{54}$ the stabilisers of triangles
are trivial, the stabilisers of edges are isomorphic to $\mathbb Z_2$, and the stabilisers of vertices are isomorphic to the dihedral group $D_3$. %(all isometries of an equilateral triangle). 
The quotient $T_{54}/G_{54}$ is a single triangle.

The $3$--simplex of groups $\mathcal G_{54}$ is defined as follows. The group of the $3$--simplex is trivial, the triangle groups are $\mathbb Z_2$, the edge groups are $D_3$, the vertex groups are $G_{54}$ and the inclusion maps correspond to inclusions
of respective stabilisers in the $G_{54}$--action on $T_{54}$; see Figure~\ref{f:g54} on the left.

\begin{figure}[h!]
	\centering
	\includegraphics[width=0.82\textwidth]{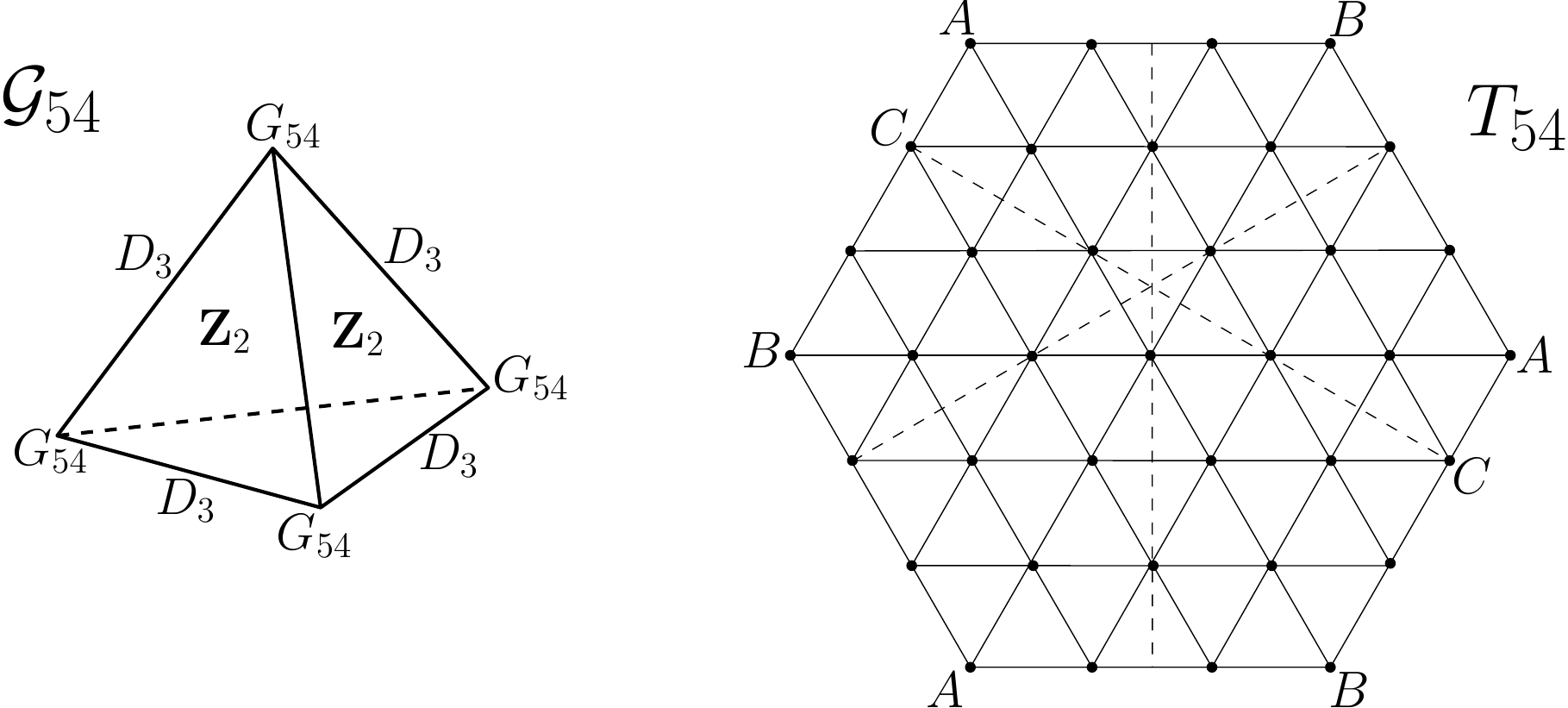}
	\caption{The simplex of groups $\mathcal G_{54}$ (left), and the local development $T_{54}$ (the development of the triangle of groups in the link of a vertex).}
	\label{f:g54}
\end{figure}

Since $\mathcal G_{54}$ is a locally $6$--large simplex of groups (see \cite[Section 6]{JS2}) it is developable by \cite[Theorem 6.1]{JS2}.
Its fundamental group $\overline G=\pi_1(\mathcal G_{54})$ acts geometrically (with the corresponding stabilisers of faces) on an
infinite $3$--dimensional systolic pseudomanifold $X$, whose vertex links are all isomorphic to 
the torus $T_{54}$. The quotient of this action is a $3$--simplex. 

Using \' Swi\polhk atkowski's construction we now define a new simplex of groups $\mathcal{G}_{54}^{\ast}$,
whose fundamental group acts on the barycentric subdivision $X'$ of the pseudomanifold ${X}$, transitively on $3$--simplices. It is obtained by assigning appropriate groups
to faces of a simplex of the barycentric subdivision of the $3$--simplex underlying $\mathcal{G}_{54}$. Let ${G}_{54}^{\ast}$ be a group of isometries of the barycentric subdivision 
$T_{54}'$ of the torus $T_{54}$ generated by reflections with respect to all edges. That is,
besides the elements of $G_{54}$ we consider also reflections with respect to lines like, for example, the dashed ones
in Figure~\ref{f:g54}. Observe that in this case the stabiliser of a triangle in $T_{54}'$ is trivial,
stabilisers of edges are $\mathbb Z_2$, and the stabilisers of vertices are as follows:
the stabiliser of a barycentre of a triangle of $T_{54}$ is $D_3$; the stabiliser of 
a barycentre of an edge of $T_{54}$ is $\mathbb Z_2^2$; the stabiliser of a vertex of 
$T_{54}$ is $D_6$. The $3$--simplex of groups $\mathcal{G}_{54}^{\ast}$ is now defined 
as follows. We consider a $3$--simplex in the barycentric subdivision
of a tetrahedron $P$ underlying $\mathcal G_{54}$, see Figure~\ref{f:g54*}.
The $3$--simplex group is trivial.  The triangle faces groups are $\mathbb Z_2$.
The assignment of the edge and vertex groups is shown in Figure~\ref{f:g54*}.

\begin{figure}[h!]
	\centering
	\includegraphics[width=0.65\textwidth]{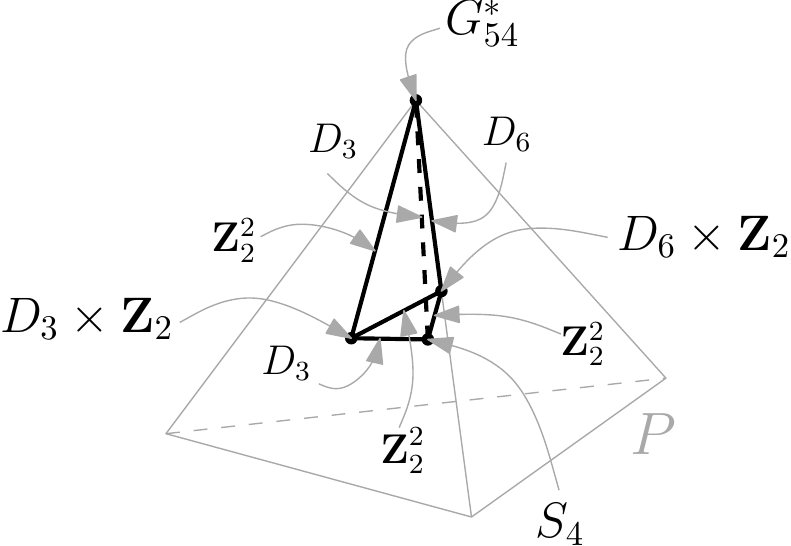}
	\caption{The simplex of groups $\mathcal G_{54}^{\ast}$.}
	\label{f:g54*}
\end{figure}

The fundamental group $G$ of $\mathcal{G}_{54}^{\ast}$ acts on $X'$ with the corresponding
stabilisers of cells and with the quotient being a $3$--simplex in the barycentric subdivision of $X$. 

\begin{prop}The complex $X$ does not admit a $G$--invariant $\mathrm{CAT}(0)$ metric.
\end{prop}

\begin{proof}Suppose such a metric exists.
	By the high transitivity of the $G$--action every edge of $X$ has the same length.
	It follows that all triangles in $X$ are equilateral. Hence, by the $\mathrm{CAT}(0)$ property,
	angles between edges in triangles are at most $\frac{\pi}{3}$, that is, the angle length of every
	edge in the link of a vertex of $X$ does not exceed $\frac{\pi}{3}$. Every such link is isomorphic to
	the barycentric subdivision $T_{54}'$ of $T_{54}$ and the vertex group $G_{54}^{\ast}$ acts transitively
	on edges. Therefore, all the edges in $T_{54}$ have the same length.
	Consider now the straight line connecting the vertices labeled $C$ in Figure~\ref{f:g54}.
	This is a homotopically non-trivial loop in the link of length strictly less that $2\pi$.
	It follows from the fact that, by the $\mathrm{CAT}(1)$ property of the link, every segment of this line contained in a single triangle has length smaller then the length $\frac{\pi}{3}$ of edges of this triangle. This contradicts the fact that the metric is $\mathrm{CAT}(0)$.
\end{proof}

It is relatively easy to observe that $X$ contains flats and hence the group $G$ is not hyperbolic \cite{Wie}. We believe that $G$ acts geometrically on a high dimensional $\mathrm{CAT}(0)$ cube complex. It seems that methods developed in the current article provide the only way of constructing low-dimensional models for the classifying spaces 
$\eeg$ and $E_{\mathcal{VAB}}G$. There are other examples of non-hyperbolic systolic groups (of high dimension) to which our theory applies.

%=======================================================================

\end{document}